\documentclass[12pt]{amsart}
\usepackage{amssymb}
\usepackage[shortlabels]{enumitem}
\usepackage[all]{xy}
\usepackage{xcolor}
\usepackage{mathrsfs}
\usepackage{tikz}
\usetikzlibrary{arrows}
\usepackage[margin=1in]{geometry} 
\usepackage{hyperref}
\usepackage{eucal}
\usepackage{contour}
\usepackage[normalem]{ulem}
\usepackage{dsfont}
\hypersetup{bookmarksdepth=2}
\hypersetup{colorlinks=true}
\hypersetup{linkcolor=blue}
\hypersetup{citecolor=blue}
\hypersetup{urlcolor=blue}

\makeatletter
\@addtoreset{equation}{section}
\makeatother

\numberwithin{equation}{section}
\newtheorem{theorem}[equation]{Theorem}

\newtheorem{proposition}[equation]{Proposition}
\newtheorem{lemma}[equation]{Lemma}
\newtheorem{corollary}[equation]{Corollary}
\newtheorem{conjecture}[equation]{Conjecture}

\theoremstyle{definition}
\newtheorem{remark}[equation]{Remark}
\newtheorem{example}[equation]{Example}
\newtheorem{definition}[equation]{Definition}

\newtheorem{question}[equation]{Question}

\newcommand{\cA}{\mathcal{A}}
\newcommand{\fA}{\mathfrak{A}}

\newcommand{\cB}{\mathcal{B}}

\newcommand{\bC}{\mathbf{C}}
\newcommand{\cC}{\mathcal{C}}
\newcommand{\fC}{\mathfrak{C}}

\newcommand{\cD}{\mathcal{D}}

\newcommand{\bF}{\mathbf{F}}
\newcommand{\cF}{\mathcal{F}}

\newcommand{\sF}{\mathscr{F}}

\newcommand{\rK}{\mathrm{K}}

\newcommand{\cM}{\mathcal{M}}

\newcommand{\bN}{\mathbf{N}}

\newcommand{\bO}{\mathbf{O}}
\newcommand{\cO}{\mathcal{O}}

\newcommand{\bP}{\mathbf{P}}
\newcommand{\cP}{\mathcal{P}}

\newcommand{\bQ}{\mathbf{Q}}

\newcommand{\bR}{\mathbf{R}}
\newcommand{\cR}{\mathcal{R}}

\newcommand{\bS}{\mathbf{S}}
\newcommand{\cS}{\mathcal{S}}
\newcommand{\fS}{\mathfrak{S}}

\newcommand{\bT}{\mathbf{T}}
\newcommand{\cT}{\mathcal{T}}

\newcommand{\bV}{\mathbf{V}}

\newcommand{\bW}{\mathbf{W}}

\newcommand{\fX}{\mathfrak{X}}

\newcommand{\cY}{\mathcal{Y}}

\newcommand{\bZ}{\mathbf{Z}}
\newcommand{\cZ}{\mathcal{Z}}

\newcommand{\fm}{\mathfrak{m}}

\newcommand{\arxiv}[1]{\href{http://arxiv.org/abs/#1}{{\tiny\tt arXiv:#1}}}

\newcommand{\DOI}[1]{\href{http://doi.org/#1}{\color{purple}{\tiny\tt DOI:#1}}}

\contourlength{1pt}

\contourlength{0.8pt}
\newcommand{\myuline}[1]{%
  \uline{\phantom{#1}}%
  \llap{\contour{white}{#1}}%
}
\DeclareMathOperator{\uRep}{\text{\myuline{\rm Rep}}}
\DeclareMathOperator{\uPerm}{\ul{Perm}}

\let\ul\underline
\let\ol\overline
\let\wh\widehat
\renewcommand{\phi}{\varphi}
\DeclareMathOperator{\uEnd}{\ul{End}}
\DeclareMathOperator{\uAut}{\ul{Aut}}

\DeclareMathOperator{\Ind}{Ind}
\DeclareMathOperator{\lev}{lev}
\DeclareMathOperator{\mlev}{mlev}
\DeclareMathOperator{\End}{End}
\DeclareMathOperator{\Aut}{Aut}
\DeclareMathOperator{\Hom}{Hom}

\DeclareMathOperator{\Rep}{Rep}

\DeclareMathOperator{\dcl}{dcl}
\DeclareMathOperator{\acl}{acl}

\DeclareMathOperator{\im}{im}

\DeclareMathOperator{\dom}{dom}
\DeclareMathOperator{\Spec}{Spec}
\DeclareMathOperator*{\Coend}{Coend}
\DeclareMathOperator{\rank}{rk}

\newcommand{\SL}{\mathbf{SL}}
\newcommand{\SO}{\mathbf{SO}}
\newcommand{\id}{\mathrm{id}}
\renewcommand{\Vec}{\mathrm{Vec}}
\newcommand{\Ver}{\mathrm{Ver}}
\newcommand{\GL}{\mathbf{GL}}
\newcommand{\bbone}{\mathds{1}}
\newcommand{\defn}[1]{\emph{#1}}

\newcommand{\bone}{\mathbf{1}}
\newcommand{\Mod}{\mathrm{Mod}}
\newcommand{\Comod}{\mathrm{Comod}}

\newcommand{\bb}{{\bullet}}
\newcommand{\ww}{{\circ}}
\newcommand{\fin}{\mathrm{fin}}
\newcommand{\sm}{\mathrm{sm}}

\title[The Drinfeld center of an oligomorphic tensor category]{The Drinfeld center of \\ an oligomorphic tensor category}

\author{Pavel Etingof}
\address{Department of Mathematics, MIT, Cambridge, MA, USA}
\email{\href{mailto:etingof@math.mit.edu}{etingof@math.mit.edu}}
\urladdr{\url{https://math.mit.edu/~etingof/}}
\thanks{PE was partially supported by NSF grants DMS-2001318 and DMS-2502467}

\author{Andrew Snowden}
\address{Department of Mathematics, University of Michigan, Ann Arbor, MI, USA}
\email{\href{mailto:asnowden@umich.edu}{asnowden@umich.edu}}
\urladdr{\url{http://www-personal.umich.edu/~asnowden/}}
\thanks{AS was supported by NSF grant DMS-2301871.}

\date{\today}

\begin{document}

\begin{abstract}
Recently, Harman and the second author introduced a new construction of pre-Tannakian tensor categories based on oligomorphic groups. We develop tools for analyzing the Drinfeld centers of these categories, and compute the center explicitly in a number of cases. In particular, we find several finitely tensor-generated pre-Tannakian categories (including the Delannoy category) that are identified with their own center via the canonical functor; prior to this work, we knew no such examples besides the category of vector spaces.
\end{abstract}

\maketitle
\tableofcontents

\section{Introduction}

Recently, Harman and the second author introduced a new construction of pre-Tannakian tensor categories based on oligomorphic groups \cite{repst}. Some of these categories, such as the Delannoy category \cite{line}, are quite different from previously known pre-Tannakian categories. A general direction of research is the investigation of the structure and properties of these new oligomorphic tensor categories. In this paper, we study their Drinfeld centers. We develop methods for studying the center, and use this theory to determine the centers of several categories.

\subsection{Descriptions of some centers}

We begin by stating some of our explicit results. The simplest ones to state are in the following theorem. We write $\cZ(\cT)$ for the Drinfeld center of a monoidal category (see \S \ref{ss:center}).

\begin{theorem} \label{mainthm1}
Let $\cT$ be one of the following pre-Tannakian categories:
\begin{enumerate}
\item the Delannoy category (\S \ref{ss:delannoy})
\item the circular Delannoy category (\S \ref{ss:delannoy})
\item an arboreal tensor category (\S \ref{ss:arboreal})
\item a Jacobi tensor category (\S \ref{ss:cantor}, \S \ref{ss:jacobi})
\item the second Delannoy category (\S \ref{ss:second}).
\end{enumerate}
Then the natural functor $\cT \to \cZ(\cT)$ is an equivalence.
\end{theorem}

Prior to this theorem, we knew no example of a finitely tensor-generated pre-Tannakian category $\cT$ for which $\cT \to \cZ(\cT)$ is an equivalence, besides the case $\cT=\Vec$. However, there are non-finitely generated examples coming from profinite groups (see Remark~\ref{rmk:profinite}).

Of course, not all oligomorphic tensor categories have this property. To describe the center in general, we introduce a class of objects in $\cZ(\cT)$ called \defn{Yetter--Drinfeld modules}. These are close analogs of Yetter--Drinfeld modules for finite groups (reviewed in \S \ref{ss:finite-center}), and can be described explicitly: the category of Yetter--Drinfeld modules is a direct sum of oligomorphic categories for open subgroups of the initial group. These modules figure into our next set of calculations:

\begin{theorem} \label{mainthm2}
Let $\cT$ be one of the following pre-Tannakian categories:
\begin{enumerate}
\item Deligne's interpolation category of the symmetric groups (\S \ref{ss:deligne})
\item an interpolation category for a family of finite classical groups (\S \ref{s:classical})
\item Kriz's quantum Delannoy category (\S \ref{ss:kriz}).
\end{enumerate}
Then $\cZ(\cT)$ coincides with the category of Yetter--Drinfeld modules.
\end{theorem}

We note that the center of Deligne's category had previously been determined in \cite{FL, FHL}. All other cases in these theorems are new, to the best of our knowledge.

As we explain below, the general principle underlying these results is that the center comes from ``small'' conjugacy classes in the oligomorphic group. The relevant groups in Theorem~\ref{mainthm1} have no small classes, except for the identity. The case of profinite groups serves as a useful toy example where the principle can be more easily observed; see Remark~\ref{rmk:profinite}. (We note that many of our results apply to pro-oligomorphic groups, a class which includes profinite groups.)

\subsection{Overview of method} \label{ss:overview}

We now describe our general method for analyzing the center of an oligomorphic tensor category. Let $G$ be an oligomorphic group with a measure $\mu$ valued in a field $k$ (see \S \ref{s:oligo} for background). The main construction of \cite{repst} produces an additive $k$-linear symmetric monoidal category $\uPerm(G, \mu)$ of ``permutation modules.'' Suppose that $\mu$ is quasi-regular and that nilpotent endomorphisms in $\uPerm(G, \mu)$ have trace zero. Then $\uPerm(G, \mu)$ has an abelian envelope $\cT=\uRep(G, \mu)$, which is a pre-Tannakian category. We note that all categories appearing in Theorems~\ref{mainthm1} and~\ref{mainthm2} arise in this way, except for the second Delannoy category.

The general theory from \cite{repst} gives an explicit description of $\cT$ as the category of finite length smooth modules over the completed group algebra of $G$. In particular, this realizes $\cT$ as a concrete category: its objects are vector spaces with extra structure. In \S \ref{ss:Opi}, we give a corresponding concrete description of the Drinfeld center, as follows. Suppose $R$ is an algebra in $\cT$. An \defn{$\cO(\pi)$-structure} on $R$ is a rule $\beta$ that assigns to each (finitary smooth) $G$-set $X$ and elements $x,y \in X$ an element $\beta^X_{x,y}$ in $R$, such that certain axioms hold (see Definition~\ref{defn:pi-struct}). If $M$ is an object of $\cT$ then giving a half-braiding on $M$ (i.e., realizing $M$ as an object of $\cZ(\cT)$) is equivalent to giving an $\cO(\pi)$-structure on $\uEnd(M)$. The elements $\beta^X_{x,y}$ essentially specify the half-braiding on Schwartz spaces.

The central idea in our approach is to associate to an object of $\cZ(\cT)$ a conjugation-stable subset of $G$, called its support. The category $\cT$ is constructed using finitary smooth $G$-sets, while the conjugation action of $G$ on itself is typically very far from smooth. We therefore need a bridge between the two worlds. This is provided by the notion of \defn{c-functor}: such is a rule $\fA$ that assigns to each transitive smooth $G$-set $X$ a $G$-subset $\fA(X)$ of $X \times X$ such that certain axioms hold (see Definition~\ref{defn:cset}). If $D \subset G$ is a conjugacy set (that is, a subset of $G$ stable by conjugation) then we obtain a c-functor $\fA_D$ by taking $\fA_D(X)$ to be the set of pairs $(x, gx)$ with $x \in X$ and $g \in D$. We show (Theorem~\ref{thm:cset}) that, under mild technical assumptions, $D \mapsto \fA_D$ is a bijection from closed conjugacy sets to c-functors.

Now, let $M$ be an object of $\cZ(\cT)$ and let $\beta$ be the $\cO(\pi)$-structure on $\uEnd(M)$. We then obtain a c-functor $\fA_M$ by taking $\fA_M(X)$ to be the set of pairs $(x,y)$ such that $\beta^X_{x,y} \ne 0$. The support of $M$ is defined to be this c-functor, or the corresponding conjugacy set. We note that if $G$ is a finite group then $\cZ(\cT)$ is equivalent to the category of $G$-equivariant sheaves on $G$ (with respect to the conjugation action), and our notion support is the usual support of the sheaf. The support of $M$ satisfies a certain representation theoretic smallness condition that we call \defn{$\cT$-small}; this is the source of all our leverage for understanding $\cZ(\cT)$. If the only $\cT$-small conjugacy sets are $\emptyset$ and $\{1\}$ then the functor $\cT \to \cZ(\cT)$ is an equivalence (Theorem~\ref{thm:Ztriv}); this is what happens in Theorem~\ref{mainthm1}. More generally, if every $\cT$-small conjugacy set contains an isolated point then the Yetter--Drinfeld modules account for all of $\cZ(\cT)$ (Theorem~\ref{thm:YD}); this is what happens in Theorem~\ref{mainthm2}.

To apply these results, we need to control $\cT$-small conjugacy sets. Our approach to this is to first show that $\cT$-small is equivalent to a combinatorially defined smallness condition, and then to analyze this condition directly using the combinatorics of the oligomorphic group. In \S \ref{s:strong}, we give a uniform treatment of this process when $G$ comes from a Fra\"iss\'e class satisfying the strong approximation property. This covers the Delannoy category, its circular variant, the arboreal categories, and Deligne's category. The remaining cases are treated using similar ideas, but the details become more ad hoc.

In \S \ref{s:nonqr}, we explain how to handle the case when the measure $\mu$ is not quasi-regular. We apply this to determine the Drinfeld center of the second Delannoy category.

\subsection{Tensor category terminology}

We fix a field $k$ throughout the paper. All tensor categories are essentially small and $k$-linear. We write $\bone$ for the tensor unit. A \defn{pre-Tannakian category} is a $k$-linear abelian category equipped with a $k$-bilinear symmetric monoidal structure such that all objects have finite length, all objects are rigid (i.e., have duals), and $\End(\bone)=k$; these conditions imply that $\Hom$ spaces are finite dimensional.

\subsection{Notation}

We list some important notation:
\begin{description}[align=right,labelwidth=2.5cm,leftmargin=!]
\item[ $k$ ] the coefficient field
\item[ $\bone$ ] the tensor unit
\item[ $\bbone$ ] a one-point set
\item[ $\Ind{\cT}$ ] the ind-category of $\cT$
\item[ $\cZ(-)$ ] the Drinfeld center of a monoidal category (\S \ref{ss:center})
\item[ $\cC(X)$ ] the Schwartz space on a $G$-set $X$ (\S \ref{ss:measure})
\item[ $\cY$ ] the category of Yetter--Drinfeld modules (\S \ref{ss:yd})
\end{description}

\section{Background on Drinfeld centers}

In this section, we review the Drinfeld center and its connection to the fundamental group of a pre-Tannakian category. We also recall the explicit description of the Drinfeld center for $\Rep{G}$ when $G$ is finite.

\subsection{The fundamental group} \label{ss:fund}

Let $\cC$ be a $k$-linear rigid symmetric monoidal category, let $\cT$ be a pre-Tannakian category, and let $F \colon \cC \to \cT$ be a $k$-linear symmetric monoidal functor. Let $\cC_*$ be a set of isomorphism class representatives in $\cC$. We define a commutative Hopf algebra $H=\Coend(F)$ in $\Ind{\cT}$ as the quotient of $\bigoplus_{X\in \cC_*} F(X)\otimes F(X)^*$ by the sum of images of morphisms 
\begin{displaymath}
F(f)\otimes 1-1\otimes F(f)^* \colon F(X) \otimes F(Y)^* \to F(Y) \otimes F(Y)^* \oplus F(X) \otimes F(X)^*
\end{displaymath}
over all $X,Y\in \cC_*$ and morphisms $f \colon X \to Y$. The coproduct in $H$ is induced by the standard coproduct on $F(X)\otimes F(X)^*$, the product by the tensor product in $\cC$, and the antipode by duality in $\cC$. There is a natural coaction of $H$ on every object of $\cC$, and $F$ factors through the symmetric monoidal functor $\cC \to \Comod^{\fin}_H$, where the target is the category of $H$-comodules in $\cT$ (not $\Ind{\cT})$. The affine group scheme\footnote{Recall that, by definition, the category of affine schemes in a pre-Tannakian category $\cT$ is anti-equivalent to the category of commutative algebras in $\Ind{\cT}$ via the functor $S \mapsto \cO(S)$ with inverse $R \mapsto \Spec{R}$, and similarly for group schemes and commutative Hopf algebras.} $\Spec(H)$ is naturally identified with $\uAut(F)$, and so $H = \cO(\uAut(F))$.

An important special case is when $\cC = \cT$ and $F$ is the identity functor. The affine group scheme $\Spec(H)$ is then called the \defn{(Deligne) fundamental group} of $\cT$ and denoted $\pi(\cT)$. We refer to \cite[\S 8]{Deligne1} and \cite[\S 2.4]{Etingof2} for additional background on the fundamental group.

\begin{remark}
The above discussion can be partially extended to the non-symmetric case. Suppose $\cC$ is a $k$-linear rigid monoidal category, $\cT$ is a locally finite $k$-linear rigid braided monoidal abelian category with $\End(\bone)=k$, and $F$ is a $k$-linear monoidal functor. Then $H=\Coend(F)$ is a Hopf algebra, and $F$ factors via the monoidal functor $\cC \to \Comod^{\fin}_H$.

Now suppose $\cC$ and $F$ are braided. Then $H$ is braided commutative. There is a notion of dyslectic (or local) $H$-comodule. The category $\cD'$ of such comodules is braided, and $F$ factors through the braided monoidal functor $\cC \to \cD'$. See \cite{Pareigis1, Pareigis2} for details.
\end{remark}

\subsection{Drinfeld center} \label{ss:center}

Let $\cC$ be a monoidal category. A \defn{half-braiding} on an object $Z$ of $\cC$ is a functorial isomorphism\footnote{In some references, the map goes in the opposite direction. We follow the convention of \cite[\S 7.13]{EGNO}.}
\begin{displaymath}
\gamma_X \colon X \otimes Z \to Z \otimes X, \qquad X \in \cC
\end{displaymath}
such that
\begin{displaymath}
\gamma_{X \otimes Y} = (\gamma_X \otimes 1_Y) \circ (1_X \otimes \gamma_Y),
\end{displaymath}
where we have dropped the associativity isomorphisms. The \defn{Drinfeld center} of $\cC$, denoted $\cZ(\cC)$ is the category of pairs $(Z, \gamma)$ where $Z$ is an object of $\cC$ and $\gamma$ is a half-braiding on $Z$. It is well-known that $\cZ(\cC)$ is a braided monoidal category with a canonical forgetful functor $\cZ(\cC) \to \cC$. Moreover, a braiding on $\cC$ is equivalent to a monoidal splitting of this functor. In particular, if $\cC$ is braided then there is a natural functor $\cC \to \cZ(\cC)$. See \cite[\S 7.13]{EGNO} and \cite[\S 8.5]{EGNO} for these statements and additional background.

Suppose now that $\cT$ is a pre-Tannakian category, let $\pi=\pi(\cT)$ be the fundamental group, and let $H=\cO(\pi)$. The category $\Mod_H$ of $H$-modules in $\Ind{\cT}$ is braided monoidal: the braiding $C_{X,Y}$ is given by composing the braiding $c_{X,Y}$ in $\cT$ with the R-matrix morphism 
\begin{displaymath}
R_{X,Y} = (\mu_X\otimes 1_Y) \circ (1_X\otimes \Delta_Y) \colon X\otimes Y\to X\otimes Y,
\end{displaymath}
where $\mu_X \colon X \otimes H \to X$ is the action of $H$ on $X$ and $\Delta_Y \colon Y\to H \otimes Y$ is the canonical coaction of $H$ on $Y$; i.e., $C_{X,Y} = c_{X,Y} \circ R_{X,Y}$. Moreover, note that $C_{X,Y}$ makes sense if $X$ is just an object of $\cT$. This gives rise to a functor
\begin{displaymath}
\Phi \colon \Mod_H \to \cZ(\Ind{\cT})
\end{displaymath}
given by $\Phi(Y)=(Y, \gamma)$, where $\gamma_X=C_{X,Y}$. It is not hard to check that $\Phi$ is naturally a monoidal functor, and moreover it preserves the braiding. We also have the functor
\begin{displaymath}
\Psi \colon \cZ(\Ind{\cT}) \to \Mod_H
\end{displaymath}
given by $\Psi(Y,\gamma) = (Y,\mu)$, where $\mu \colon Y \otimes H \to Y$ is given by 
\begin{displaymath}
\mu=(1_Y \otimes \varepsilon) \circ \gamma_H,
\end{displaymath}
where we view $H$ as an $H$-comodule using the coproduct, and $\varepsilon$ is the counit. We have the following result about these constructions; see, e.g., \cite{LZ}.

\begin{proposition} \label{prop:Opi-center}
The functors $\Phi$ and $\Psi$ are mutually inverse equivalences of braided monoidal categories between $\Mod_{\cO(\pi)}$ and $\cZ(\Ind{\cT})$.
\end{proposition} 

The functors $\Phi$ and $\Psi$ do not affect the underlying object of $\Ind{\cT}$. Hence, in the above proposition, $\cZ(\cT)$ is equivalent to the category $\Mod_{\cO(\pi)}^{\fin}$ of $\cO(\pi)$-modules in $\cT$.

\begin{remark}
Proposition~\ref{prop:Opi-center} remains valid in the braided case, as shown in \cite{LZ}. Precisely, suppose $\cT$ is a locally finite $k$-linear braided rigid monoidal abelian category with $\End(\bone)=k$, and let $H=\Coend(\id_{\cT})$. Then $\cZ(\Ind{\cT})$ and the category $\Mod_H$ of right $H$-modules are equivalent as braided monoidal categories.
\end{remark}

\begin{remark} \label{rmk:IndZ}
In general $\cZ(\Ind{\cT})$ is significantly larger than $\Ind{\cZ(\cT)}$. For example, suppose $\cT=\Rep(G)$ where $G$ is an adjoint simple complex algebraic group. Then $\Ind{\cZ(\cT)}$ is the category of locally finite $G$-equivariant $\cO(G)$-modules, whose only simple objects are simples $V \in \cT$ with trivial action of $\cO(G)$. (We note, however, that there are non-trivial extensions among simples in $\cZ(\cT)$, unlike in $\cT$.) On the other hand, $\cZ(\Ind{\cT})$ is the category of all $G$-equivariant $\cO(G)$-modules, which for instance contains the simple object $\cO(C)$ where $C$ is any semisimple conjugacy class of $G$. 
 
Similarly, consider Deligne's interpolation category $\cT=\uRep(\fS_t)$ over the complex numbers \cite{Deligne3}, with $t \not\in \bN$. It is shown in \cite{FHL} that the category $\cZ(\cT)$ is semisimple. However, the category $\cZ(\Ind{\cT})$ is not semisimple and has much richer structure. We give an example of an interesting object. The invariant ring $\cO(\pi)^{\pi}$ is the Grothendieck ring of $\cT$, which is isomorphic to $\bC[X_1, X_2, \ldots]$. For a sequence $x=(x_1, x_2, \ldots)$, with $x_i \in \bC$, let $\fm_x$ be the maximal ideal in $\bC[X_1, X_2, \ldots]$ generated by $X_i-x_i$ for $i\ge 1$. We can then consider the $\cO(\pi)$-module $\cO(\pi)/\fm_x \cO(\pi) \in \cZ(\Ind{\cT})$, and its simple quotients (which exist by Zorn's lemma, since it is finitely generated). For sufficiently general $x$, the simple quotients are infinite dimensional (meaning that underlying object in $\cT$ has infinite length); in fact, we expect that $\cO(\pi)/\fm_x \cO(\pi)$ is itself simple for sufficiently generic $x$.
\end{remark} 

\subsection{Finite groups} \label{ss:finite-center}

Let $G$ be a finite group, and consider the category $\cT=\Rep(G)$ of finite dimensional representations of $G$ over the field $k$. We now recall the description of the Drinfeld center of $\cT$ in terms of Yetter--Drinfeld modules. Our analysis of the center of oligomorphic tensor categories is inspired by this case.

A \defn{Yetter--Drinfeld module} for $G$ is a $k[G]$-module $V$ equipped with a decomposition
\begin{displaymath}
V = \bigoplus_{g \in G} V_g
\end{displaymath}
such that $hV_g = V_{hgh^{-1}}$ for all $g,h \in G$. Such a module $V$ defines an object of $\cZ(\Ind{\cT})$: if $W$ is a representation of $G$ then the half-braiding is given by
\begin{displaymath}
W \otimes V \to V \otimes W, \qquad w \otimes v \mapsto v \otimes gw
\end{displaymath}
for $v \in V_g$. In fact, $\cZ(\Ind{\cT})$ is equivalent to the category of Yetter--Drinfeld modules. It is not difficult to prove this directly, but it also follows from the general description of the center in terms of the fundamental group. Indeed, $\cO(\pi)$ is the algebra of $k$-valued functions on $G$, with the action of $G$ induced by conjugation. As an algebra, we have $\cO(\pi) = \prod_{g \in G} k$, and so an $\cO(\pi)$-module in $\cT$ is a $G$-graded representation of $G$, i.e., a Yetter--Drinfeld module. This discussion yields a very explicit description of the center: letting $g_1, \ldots, g_r$ be representatives for the conjugacy classes of $G$, we have
\begin{displaymath}
\cZ(\cT) \cong \bigoplus_{i=1}^r \Rep(Z(g_i)),
\end{displaymath}
where $Z(g_i)$ is the centralizer of $g_i$ in $G$.

\begin{remark} \label{rmk:profinite}
Suppose that $G = \varprojlim G_i$ is a profinite group and $\cT=\Rep{G}$. In this case, a Yetter--Drinfeld module is a discrete $k[G]$-module equipped with a $G_i$-grading (compatible with conjugation) for each $i$, such that for $i>j$ the $G_i$-grading refines the $G_j$-grading. Such modules again describe the center of $\Ind{\cT}$. Let $G_{\sm} \subset G$ be the elements with open centralizer, or, equivalently, finite conjugacy class; we call such elements \defn{c-smooth} (see \S \ref{ss:csmooth}). One can show that if $V$ is a finite dimensional Yetter--Drinfeld module then $V$ is naturally $G_{\sm}$-graded, and this grading defines the braiding. In particular, if $G_{\sm}$ is trivial then all finite dimensional Yetter--Drinfeld modules are trivial (i.e., supported at the identity), and the functor $\cT \to \cZ(\cT)$ is an equivalence. This happens for $G=\SL_2(\bZ_p)/\{\pm 1\}$, for instance.
\end{remark}

\section{Oligomorphic groups and tensor categories} \label{s:oligo}

In this section, we recall background on oligomorphic groups and their associated tensor categories, following \cite{repst}. In \S \ref{ss:cset}, we establish the correspondence between c-functors and conjugacy sets, and in \S \ref{ss:csmooth} we briefly discuss c-smooth elements. These two subsections contain the only new material in this section.

\subsection{Oligomorphic groups} \label{ss:oligo}

A \defn{pro-oligomorphic group} is a topological group $G$ such that (a) $G$ is Hausdorff; (b) $G$ is non-archimedean (open subgroups form a neighborhood basis of the identity); and (c) $G$ is Roelcke precompact ($U \backslash G/V$ is a finite set for all open subgroups $U$ and $V$). An \defn{oligomorphic permutation group} is a permutation group $(G, \Omega)$ such that $G$ has finitely many orbits on $\Omega^n$ for all $n$.

The two notions are closely linked. Suppose $(G, \Omega)$ is an oligomorphic permutation group. Given a finite subset $A$ of $\Omega$, let $G(A)$ be the subgroup of $G$ consisting of elements that fix each element of $A$. These subgroups form a neighborhood basis of the identity for a pro-oligomorphic topology on $G$ \cite[\S 2.2]{repst}. Conversely, suppose $G$ is pro-oligomorphic. Then for any open subgroup $U$ of $G$, the action of $G$ on $G/U$ is oligomorphic \cite[\S 2.3]{repst}. Thus, if there exists $U$ such that this action is faithful, then $(G, G/U)$ is an oligomorphic permutation group. In general, $G$ is a dense subgroup of the inverse limit of oligomorphic permutation groups. We will ultimately be interested in oligomorphic permutation groups, however many constructions depend only on the topology and not $\Omega$, so it is convenient to work in the pro-oligomorphic setting.

Let $G$ be a pro-oligomorphic group. An action of $G$ on a set $X$ is \defn{smooth} if the stabilizer of each element of $X$ is an open subgroup of $G$, and \defn{finitary} if $G$ has finitely many orbits on $X$. We use the term ``$G$-set'' to mean ``set equipped with a finitary and smooth action of $G$.'' A key property is that the product of two $G$-sets is again a $G$-set \cite[Proposition~2.8]{repst}. A \defn{$\wh{G}$-set} is a set equipped with a finitary and smooth action of some open subgroup of $G$; shrinking the open subgroup does not change the $\wh{G}$-set. See \cite[\S 2.5]{repst} for details.

Suppose $(G, \Omega)$ is oligomorphic. We say that a $G$-set is \defn{$\Omega$-smooth} if each orbit on it is isomorphic to an orbit on some $\Omega^n$. The class of $\Omega$-smooth $G$-sets is closed under disjoint union, products, and passing to $G$-subsets. Every $G$-set is a quotient of an $\Omega$-smooth one. Indeed, a transitive $G$-set $G/U$ is a quotient of the $\Omega$-smooth $G$-set $G/G(A)$, where $A \subset \Omega$ is a finite set with $G(A) \subset U$. Typically, $\Omega$-smooth $G$-sets are easier to understand than all $G$-sets, which is why the notion is useful.

\begin{example}
Let $\fS$ be the group of all permutations of $\Omega=\{1,2,\ldots\}$, i.e., the infinite symmetric group. This is the first example of an oligomorphic permutation group. Let $\fS(n)$ be the subgroup of $\fS$ fixing each of $1, \ldots, n$. A subgroup of $\fS$ is open if and only if it contains $\fS(n)$ for some $n$. By \cite[Proposition~14.1]{repst}, every open subgroup of $\fS$ is conjugate to one of the form $H \times \fS(n)$, where $H$ is a subgroup of the finite symmetric group $\fS_n$.

Let $\Omega^{[n]}$ denote the subset of $\Omega^n$ where the coordinates are distinct. It is not difficult to see that the $\Omega^{[n]}$ are exactly the $\Omega$-smooth transitive $\fS$-sets. It follows from \cite[Proposition~14.1]{repst} that every transitive $\fS$-set is isomorphic to $\Omega^{[n]}/H$ for some subgroup $H$ of $\fS_n$ (which acts on $\Omega^{[n]}$ by permuting coordinates). An $\wh{\fS}$-subset of $\Omega$ is a finite or cofinite subset. See \cite[Figure~1]{repst} for more interesting examples of $\wh{\fS}$-sets.
\end{example}

\subsection{Measures and integration} \label{ss:measure}

Fix a pro-oligomorphic group $G$ and a field $k$. The following is a key definition from \cite[\S 3.2]{repst}.

\begin{definition}
A \defn{measure} for $G$ valued in $k$ is a rule $\mu$ that assigns to each $\wh{G}$-set $X$ a quantity $\mu(X)$ in $k$ such that the following conditions hold (in which $X$ and $Y$ are $\wh{G}$-sets):
\begin{enumerate}
\item If $X$ and $Y$ are isomorphic then $\mu(X)=\mu(Y)$.
\item We have $\mu(\bbone)=1$ where $\bbone$ is the one-point set.
\item We have $\mu(X \amalg Y)=\mu(X)+\mu(Y)$.
\item For $g \in G$ we have $\mu(X)=\mu(X^g)$, where $X^g$ is the conjugate of $X$ by $g$.
\item If $X$ and $Y$ are transitive $U$-sets, for some open subgroup $U$, and $Y \to X$ is a map of $U$-sets with fiber $F$ (over some point), then $\mu(Y)=\mu(F) \cdot \mu(X)$.
\end{enumerate}
\end{definition}

Fix a measure $\mu$. Let $X$ be a $\wh{G}$-set. We say that $\phi \colon X \to k$ is a \defn{Schwartz function} if it is invariant under some open subgroup of $G$. In this case, $\phi$ takes on finitely many values $a_1, \ldots, a_n$. Let $Y_i=\phi^{-1}(a_i)$. We define the \defn{integral} of $\phi$ by
\begin{displaymath}
\int_X \phi(x) \, dx = \sum_{i=1}^n a_i \mu(Y_i).
\end{displaymath}
We refer to \cite[\S 4.3]{repst} for the basic properties of this construction.

We write $\cC(X)$ for the space of Schwartz functions on $X$. Integration defines a linear map $\cC(X) \to k$. More generally, if $f \colon Y \to X$ is a map of $\wh{G}$-sets that is equivariant for some open subgroup then integration over the fibers defines a push-forward map $f_* \colon \cC(Y) \to \cC(X)$. There is also a pull-back map $f^* \colon \cC(X) \to \cC(Y)$, which does not depend on the measure. These maps have many of the expected properties; see \cite[\S 4.4]{repst} for details.

Let $X$ and $Y$ be $G$-sets. A \defn{$Y \times X$ matrix} is a Schwartz function $A \colon Y \times X \to k$. A matrix defines a linear map $A \colon \cC(X) \to \cC(Y)$ by
\begin{displaymath}
(A \phi)(y) = \int_X A(y,x) \phi(x) \, dx.
\end{displaymath}
Suppose $Z$ is another $G$-set and $B$ is a $Z \times Y$ matrix. We define $BA$ to be the $Z \times X$ matrix given by
\begin{displaymath}
(BA)(z,x) = \int_Y B(z,y) A(y,x) \, dy.
\end{displaymath}
This has the expected properties of matrix multiplication; see \cite[\S 7]{repst}. In particular, the linear map $BA \colon \cC(X) \to \cC(Z)$ is the composition of the linear maps $A$ and $B$.

\subsection{Permutation modules} \label{ss:perm}

Let $G$ be a pro-oligomorphic group and let $\mu$ be a measure valued in the field $k$. Following \cite[\S 8]{repst}, we define a category $\uPerm(G, \mu)$ of ``permutation modules'' as follows. The objects are $\cC(X)$, where $X$ is a $G$-set. A morphism $\cC(X) \to \cC(Y)$ is a $G$-invariant $Y \times X$ matrix. Composition is given by matrix multiplication. This category is additive and has a symmetric monoidal structure. On objects, we have
\begin{displaymath}
\cC(X) \oplus \cC(Y) = \cC(X \amalg Y), \qquad \cC(X) \otimes \cC(Y) = \cC(X \times Y).
\end{displaymath}
On morphisms, direct sum and tensor product are given by the familiar block matrix and Kronecker product constructions. The category $\uPerm(G, \mu)$ is rigid, and each object is self-dual. The dual of a matrix is given by transpose. The categorical dimension of $\cC(X)$ is equal to $\mu(X)$. See \cite[\S 8]{repst} for proofs and additional details.

Suppose $f \colon Y \to X$ is a map of $G$-sets. The indicator function of the graph of $f$ can be regarded as a $Y \times X$ matrix. We write $f^* \colon \cC(X) \to \cC(Y)$ for the morphism in $\uPerm(G, \mu)$ corresponding to this matrix, and $f_* \colon \cC(Y) \to \cC(X)$ for the morphism corresponding to the transpose matrix. The linear transformations induced by these matrices are the operations $f^*$ and $f_*$ and so the notation is consistent. See \cite[\S 7.7]{repst} for details.

Let $X$ be a $G$-set. Let $\Delta \colon X \to X \times X$ be the diagonal and let $f \colon X \to \bbone$ be the unique map to the one-point $G$-set. Note that $\cC(\bbone)=\bone$ is the tensor unit. We have maps
\begin{align*}
\Delta_* &\colon \cC(X) \otimes \cC(X) \to \cC(X), &
f_* &\colon \cC(X) \to \bone, \\
\Delta^* &\colon \cC(X) \to \cC(X) \otimes \cC(X), &
f^* &\colon \bone \to \cC(X).
\end{align*}
These maps give $\cC(X)$ the structure of a special commutative Frobenius algebra \cite[\S 9.6]{repst}, or, equivalently, an \'etale algebra. See \cite[\S 4,5]{discrete} for background on \'etale algebras.

\subsection{The completed group algebra} \label{ss:cga}

Let $G$ and $\mu$ be as in \S \ref{ss:perm}. We say that $\mu$ is \defn{regular} if $\mu(X) \ne 0$ for all transitive $G$-sets $X$, and \defn{quasi-regular} if $\mu$ restricts to a regular measure on some open subgroup of $G$. Suppose that $\mu$ is quasi-regular and that nilpotent morphisms in $\uPerm(G, \mu)$ have trace~0. Then $\uPerm(G, \mu)$ has an abelian envelope, i.e., there is a pre-Tannakian category $\uRep(G, \mu)$ and a fully faithful tensor functor
\begin{displaymath}
\uPerm(G, \mu) \to \uRep(G, \mu)
\end{displaymath}
satisfying a universal property; see \cite[\S 13.3]{repst}. We now recall the explicit construction of $\cT=\uRep(G, \mu)$ from \cite[Part~III]{repst}. We assume in what follows that $G$ is first countable.

Let $A=A(G, \mu)$ be the inverse limit of the Schwartz spaces $\cC(G/U)$, over all open subgroups $U$ of $G$. The space $A$ is naturally an associative algebra, where the multiplication is essentially given by convolution of functions; see \cite[\S 10.3]{repst}. We note that convolution is defined using integrals, and thus depends on the measure. We call $A$ the \defn{completed group algebra} of $G$. We note that there is a natural inclusion of the ordinary group algebra $k[G]$ into $A$ \cite[\S 10.4]{repst}, and so any $A$-module is also a $k[G]$-module.

We say that an $A$-module $M$ is \defn{smooth} if for every $x \in M$ the action of $A$ on $x$ factors through some $\cC(G/U)$. The category of smooth $A$-modules is a Grothendieck abelian category \cite[Proposition~10.7]{repst}, and any smooth $A$-module is the union of its finite length submodules \cite[Theorem~13.2]{repst}. We note that the theory of integration can be extended to functions valued in a smooth module \cite[\S 11.9]{repst}. Let $\cT$ be the category of finite length smooth $A$-modules, and identify $\Ind{\cT}$ with the category of all smooth $A$-modules. If $\mu$ is regular then $\cT$ is semi-simple \cite[Theorem~13.2]{repst}.

If $X$ is a $G$-set then the Schwartz space $\cC(X)$ is naturally an $A$-module \cite[\S 11.4]{repst}. In fact, we have a fully faithful $k$-linear functor
\begin{displaymath}
\Phi \colon \uPerm(G, \mu) \to \cT
\end{displaymath}
sending each Schwartz space to itself \cite[\S 11.5]{repst}. The $A$-module $\cC(X)$ is generated by the point masses $\delta_x$ for $x \in X$; in fact, one only needs one point mass from each $G$-orbit \cite[Proposition~11.8]{repst}. Thus $\cC(X)$ is a finitely generated $A$-module, and therefore of finite length. The Schwartz spaces have a natural mapping property: if $M$ is any smooth $A$-module then
\begin{displaymath}
\Hom_A(\cC(X), M) = \Hom_G(X, M),
\end{displaymath}
where on the right side we are simply taking the space of $G$-equivariant functions \cite[Corollary~11.18]{repst}. In particular, giving a map of $A$-modules $\cC(G/U) \to M$ is equivalent to giving an element of the invariant space $M^U$. Every object of $\cT$ is a quotient of some $\cC(X)$ \cite[Proposition~11.11]{repst}.

The category $\Ind{\cT}$ carries a tensor product $\otimes$ \cite[\S 12]{repst}. If $M$ and $N$ are smooth $A$-modules then $M \otimes N$ is a smooth $A$-module equipped with a ``strongly bilinear map'' $q \colon M \times N \to M \otimes N$ \cite[Definition~12.1]{repst}. Essentially, this means that $q$ is $G$-equivariant, and commutes with integration in each variable. There is a natural injective map $M \otimes_k N \to M \otimes N$, and the image generates $M \otimes N$ as an $A$-module \cite[Proposition~12.4]{repst}. We have a natural isomorphism $\cC(X) \otimes \cC(Y) = \cC(X \times Y)$ \cite[Proposition~12.6]{repst}, and, in fact, $\Phi$ is a symmetric monoidal functor. Since every module is a quotient of a Schwartz space, one can use presentations by Schwartz spaces to compute the tensor product. The tensor product preserves $\cT$, and gives it the structure of a pre-Tannakian category \cite[Theorem~13.2]{repst}. The abelian envelope property is proved in \cite[Theorem~13.13]{repst}.

\subsection{Fra\"iss\'e limits} \label{ss:fraisse}

The Fra\"iss\'e limit construction is an important source of oligomorphic groups. We will use this perspective in some of the examples we discuss, so we briefly review the key definitions. We refer to \cite[\S 6.2]{repst} for additional details.

Let $\sF$ be a class of finite relational structures. For example, $\sF$ could be the class of finite totally ordered sets, or finite graphs. Given embeddings $i \colon A \to B$ and $j \colon A \to A'$ of structures in $\sF$, an \defn{amalgamation} of $B$ and $A'$ over $A$ is a triple $(B', i', j')$ where $B'$ is a structure in $\sF$ and $i' \colon A' \to B'$ and $j' \colon B \to B'$ are embeddings such that $j'i=i'j$ and $B'=\im(i') \cup \im(j')$. Thus an amalgamation is a way of gluing $B$ and $A'$ together which identifies the copy of $A$ in each. Note that an amalgamation is allowed to identify more than $A$. For example, if $i' \colon A' \to B$ is an embedding then $(B, i', \id_B)$ is an amalgamation of $A'$ and $B$ over $A=\emptyset$.

We say that $\sF$ is a \defn{Fra\"iss\'e class} if it satisfies the following conditions:
\begin{enumerate}
\item If $B$ belongs to $\sF$ and $A \to B$ is an embedding of structures then $A$ belongs to $\sF$. In particular, $\sF$ contains the empty structure.
\item Given embeddings $i \colon A \to B$ and $j \colon A \to A'$ there exists at least one amalgamation $(B', i', j')$ in $\sF$. (This is called the \defn{amalgamation property}.)
\item Up to isomorphism, $\sF$ has only finitely many members of a given cardinality.
\end{enumerate}
We recall two more definitions. A structure $\Omega$ is \defn{homogeneous} if any two embeddings of a finite structure into $\Omega$ differ by an automorphism of $\Omega$. The \defn{age} of a structure is the class of finite structures that embed into it.

Let $\sF$ be a Fra\"iss\'e class. Then there is a countable homogeneous structure $\Omega$ with age $\sF$; moreover, any other such structure is isomorphic to $\Omega$. The structure $\Omega$ is called the \defn{Fra\"iss\'e limit} of $\sF$. The automorphism group $G$ of $\Omega$ acts oligomorphically on $\Omega$. For a finite structure $A$ in $\sF$, we let $\Omega^{[A]}$ denote the set of embeddings $A \to \Omega$. This is a transitive $\Omega$-smooth $G$-set, and every transitive $\Omega$-smooth $G$-set is of this form \cite[\S 6.5]{repst}. Given embeddings $A \to B$ and $A \to A'$ in $\sF$, the $G$-orbits on the fiber product
\begin{displaymath}
\Omega^{[B]} \times_{\Omega^{[A]}} \Omega^{[A']}
\end{displaymath}
are naturally in bijection with the amalgamations of $A'$ and $B$ over $A$, up to isomorphism \cite[Lemma~6.11]{repst}.

\begin{example}
Let $\sF$ be the class of finite totally ordered sets. The Fra\"iss\'e limit is the set $\bQ$ of rational numbers with its standard order. It is clear that the age of $\bQ$ is $\sF$, and it is not difficult to verify directly that $\bQ$ is homogeneous. The group $\Aut(\bQ, <)$ of automorphisms of the ordered set $(\bQ, <)$ is an important example of an oligomorphic group. It gives rise to the Delannoy category; see \S \ref{ss:delannoy}.
\end{example}

\subsection{Conjugacy sets} \label{ss:cset}

Let $G$ be a pro-oligomorphic group. A \defn{conjugacy set} in $G$ is a subset of $G$ that is stable under conjugation. As explained in \S \ref{ss:overview}, we will want to be able to understand conjugacy sets in $G$ terms of (finitary smooth) $G$-sets. The following is the key definition that enables this.

\begin{definition} \label{defn:cset}
A \defn{c-functor} is a rule $\fA$ assigning to each transitive $G$-set $X$ a $G$-stable subset $\fA(X)$ of $X \times X$ such that for any map $Y \to X$ of transitive $G$-sets we have the following:
\begin{enumerate}
\item $\fA(Y)$ maps to $\fA(X)$.
\item Given $(x_1,x_2) \in \fA(X)$ and $y_1$ lifting $x_1$, there exists $y_2$ lifting $x_2$ with $(y_1,y_2) \in \fA(Y)$.
\item Given $(x_1,x_2) \in \fA(X)$ and $y_2$ lifting $x_2$, there exists $y_1$ lifting $x_1$ with $(y_1,y_2) \in \fA(Y)$.
\end{enumerate}
Note that these conditions imply that $\fA(Y)$ maps onto $\fA(X)$.
\end{definition}

Suppose $D \subset G$ is a conjugacy set. Define $\fA_D(X)$ to be the set of pairs $(x,gx)$ where $x \in X$ and $g \in D$. Then $\fA_D$ is easily seen to be a c-functor. We now show that this construction is (essentially) bijective in nice situations. We say that an oligomorphic permutation group $(G, \Omega)$ is \defn{closed} if $G$ is a topologically closed subgroup of the symmetric group on $\Omega$. Concretely, this means that if $g \colon \Omega \to \Omega$ is a bijection such that for every finite subset $A$ of $\Omega$ there is some element $h_A \in G$ such that $gx=h_Ax$ for all $x \in A$ then $g \in G$. Similarly, we say that a conjugacy set is \defn{closed} if it is topologically closed in $G$.

\begin{theorem} \label{thm:cset}
Let $(G, \Omega)$ be a closed oligomorphic permutation group with $\Omega$ countable. Then $D \mapsto \fA_D$ defines a bijection
\begin{displaymath}
\{ \text{closed conjugacy sets} \} \to
\{ \text{c-functors} \}.
\end{displaymath}
\end{theorem}

\begin{proof}
Suppose $D$ and $D'$ are closed conjugacy sets with $\fA_D = \fA_{D'}$. We have $(1,g) \in \fA_D(G/U)$ if and only if $g \in D U$, and so $D U = D' U$ for all $U$. The closure of any $X \subset G$ is the intersection of $XU$ over all open subgroups $U$. Thus $D=D'$, since both are closed. Therefore the map is injective. To complete the proof, we must show that it is surjective. We break the argument into a number of steps. Fix a c-functor $\fA$ throughout. (We point the reader to Remark~\ref{rmk:cset-finite} for the argument when $G$ is finite, which is much simpler.)

\textit{(a) Constructing the conjugacy set.} Define $D$ to be the set of elements $g \in G$ such that $(1, g) \in \fA(G/U)$ for all open subgroups $U$. We will eventually show that $\fA=\fA_D$. At the moment, we simply show that $D$ is a closed conjugacy set. Closed is clear: if $g$ belongs to the closure of $D$ then for each open subgroup $U$ we have $g \in D U$, and so $(1,g) \in \fA(G/U)$, and so $g \in D$.

We now show that $D$ is stable under conjugation. Thus let $g \in D$ and $h \in G$ be given. We show that $hgh^{-1} \in D$. For this, we must show that $(1, hgh^{-1}) \in \fA(G/U)$ for all open subgroups $U$. Thus  let $U$ be given. We have $(1,g) \in \fA(G/h^{-1}Uh)$ by assumption, and so $(h, hg) \in \fA(G/h^{-1}Uh)$. We now apply the isomorphism of $G$-sets
\begin{displaymath}
G/h^{-1}Uh \to G/U, \qquad a \mapsto ah^{-1}.
\end{displaymath}
Note that $\fA$ respects isomorphisms by Definition~\ref{defn:cset}(a,b). We thus see that $(1, hgh^{-1}) \in \fA(G/U)$, as required.

\textit{(b) Partially defined maps.} Let $\cF$ be the set of all injections $\sigma \colon A \to \Omega$, with $A$ a finite subset of $\Omega$, for which there exists $g \in G$ such that $(1, g) \in \fA(G/G(A))$ and $\sigma = g \vert_A$. We call $g$ a \defn{witness} for $\sigma$, and call $A$ the \defn{domain} of $\sigma$, denoted $\dom{\sigma}$. Given $\sigma, \tau \in \cF$, we say that $\tau$ \defn{extends} $\sigma$ if $\dom{\sigma} \subset \dom{\tau}$ and $\tau \vert_{\dom{\sigma}} = \sigma$. We think of elements of $\cF$ as finite approximations to elements of $D$.

\textit{(c) Enlarging the domain.} Given $\sigma \in \cF$ and any finite subset $B$ of $\Omega$, we claim that there exists $\tau \in \cF$ extending $\sigma$ with $B \subset \dom{\tau}$. Let $A=\dom{\sigma}$ and let $g$ be a witness to $\sigma$. Let $h \in G/G(A \cup B)$ be a lift of $g \in G/G(A)$ such that $(1,h) \in \fA(G/G(A \cup B))$, which exists by Definition~\ref{defn:cset}(b). We can then take $\tau$ to be the restriction of $h$ to $A \cup B$.

\textit{(d) Enlarging the image.} Given $\sigma \in \cF$ and any finite subset $B$ of $\Omega$, we claim that there exists $\tau \in \cF$ extending $\sigma$ with $B \subset \im{\tau}$. Once again, let $A=\dom{\sigma}$ and let $g \in G$ witness $\sigma$. We have an isomorphism of $G$-sets
\begin{displaymath}
G/G(A) \to G/G(gA), \qquad a \mapsto ag^{-1}.
\end{displaymath}
Since $(1,g)$ belongs to $\fA(G/G(gA))$, it follows that $(g^{-1}, 1)$ belongs to $\fA(G/G(gA))$. Let $h^{-1} \in G/G(gA \cup B)$ be a lift of $g^{-1} \in G/G(gA)$ such that $(h^{-1}, 1)$ belongs to $\fA(G/G(gA \cup B))$. This exists by Definition~\ref{defn:cset}(c). Note that $h^{-1}$ agrees with $g^{-1}$ on $gA$, and so $h^{-1}gA=A$, and $h$ agrees with $g$ on $A$. Applying the isomorphism
\begin{displaymath}
G/G(gA \cup B) \to G/G(A \cup h^{-1}B), \qquad a \mapsto ah,
\end{displaymath}
we see that $(1,h)$ belongs to $\fA(G/G(A \cup h^{-1}B))$. We can now take $\tau$ to be the restriction of $h$ to $A \cup h^{-1}B$.

\textit{(e) Constructing elements of $D$.} Given $\sigma \in \cF$, we claim there is an element $g \in D$ such that $g \vert_{\dom{\sigma}} = \sigma$. This is a standard ``back and forth'' argument. Let $A=A_1 \subset A_2 \subset \cdots$ be a chain of finite subsets of $\Omega$ that union to $\Omega$; this is where we use the countability hypothesis. We inductively define $\sigma_i \in \cF$, for $i \ge 0$, as follows. First $\sigma_0=\sigma$. Now, having defined $\sigma_{i-1}$, we let $\sigma_i$ be any extension of $\sigma_{i-1}$ such that $A_i \subset \dom{\sigma_i}$ and $A_i \subset \im{\sigma_i}$; such a $\sigma_i$ exists by steps (c) and (d).

We now define $g \colon \Omega \to \Omega$ by $gx=\sigma_i(x)$ for $x \in \dom{\sigma_i}$. This is well-defined since the $\sigma_i$'s are compatible. Since the $\sigma_i$'s are injective, it follows that $g$ is injective. Since the image of $\sigma_i$ contains $A_i$, it follows that the image of $g$ contains $A_i$ for all $i$, and so $g$ is surjective. Thus $g$ is bijective. If $g_i \in G$ is a witness for $\sigma_i$ then $g \vert_{A_i} = g_i \vert_{A_i}$, and so $g \in G$ since $G$ is closed. Since $(1, g_i) \in \fA(G/G(A_i))$, it follows that $(1, g) \in \fA(G/G(A_i))$ for all $i$. Since the $G(A_i)$'s are a cofinal system of open subgroups, it follows that $(1, g) \in \fA(G/U)$ for all open subgroups $U$; this uses Definition~\ref{defn:cset}(a). Thus $g \in D$. This establishes the claim.

\textit{(f) Completing the proof.} We now show that $\fA=\fA_D$. It suffices to show $\fA(G/U) = \fA_D(G/U)$ for all open subgroups $U$; in fact, it suffices to treat a cofinal family by Definition~\ref{defn:cset}(a). Thus fix $U=G(A)$, where $A$ is a finite subset of $\Omega$. Since $\fA(G/U)$ and $\fA_D(G/U)$ are both $G$-sets, it suffices to show that they contain the same elements of the form $(1,g)$. Now, if $(1,g) \in \fA_D(G/U)$ then we have $g=hu$ for some $h \in D$ and $u \in U$. Since $h \in D$, we have $(1,h) \in \fA(G/U)$, and so $(1,g) \in \fA(G/U)$, as required. Conversely, suppose $(1,g) \in \fA(G/U)$. By step~(e) above, there exists $h \in D$ such that $h \vert_A=g \vert_A$. Since $(1,h)$ belongs to $\fA_D(G/U)$, so does $(1,g)$.
\end{proof}

\begin{remark}
A topological group is called \textit{Ra\u{\i}kov complete} if it is complete with respect to its natural two-sided uniform structure\footnote{The two-sided uniformity is the join of the left and right uniformities. The completion of a group in the two-sided uniformity is a group, while the completion in the right uniformity is only a monoid in general. The monoid $G^*$ appearing in the proof of Theorem~\ref{thm:ZS} is the completion of $G$ with respect to its right uniformity.}; see \cite[\S 3.6]{AT}. Theorem~\ref{thm:cset} holds more generally for a first-countable Ra\u{\i}kov complete pro-oligomorphic group. We note that an oligomorphic permutation group is Ra\u{\i}kov complete if and only if it is closed.
\end{remark}

\begin{remark} \label{rmk:cset-finite}
Let $G$ be a finite group and let $\fA$ be a c-functor. We give a direct proof in this case that $\fA=\fA_D$ for some conjugacy set $D$. Let $D$ be the set of all elements $g \in G$ such that $(1,g)$ belongs to $\fA(G)$. We show that $D$ is a conjugacy set. Suppose $g \in D$ and $h \in G$. Since $(1,g) \in \fA(G)$, so is $(h, hg)$. Now, right multiplication by $h^{-1}$ gives an isomorphism of left $G$-sets $G \to G$, and so $(1,hgh^{-1})$ belongs to $\fA(G)$. Thus $hgh^{-1} \in D$, and so $D$ is a conjugacy set. It is clear that $\fA(G) = \fA_D(G)$. Since every transitive $G$-set is a quotient of $G$, it follows that $\fA=\fA_D$.
\end{remark}

\subsection{Smooth conjugacy classes} \label{ss:csmooth}

Let $G$ be a pro-oligomorphic group. We say that $g \in G$ is \defn{c-smooth} if its centralizer $Z(g)$ is an open subgroup; this means that $g$ is a smooth element for the action of $G$ on itself by conjugation. The set $G_{\sm}$ of c-smooth elements of $G$ forms a normal subgroup of $G$. The c-smooth elements of $G$ are used to define our Yetter--Drinfeld modules in \S \ref{ss:yd}.

If $X$ is a subset of $G$ then a point $g \in X$ is \defn{isolated} if there is an open subgroup $U$ of $G$ such that $gU \cap X=\{g\}$. Isolated points are a source of c-smooth elements:

\begin{proposition} \label{prop:isolated-smooth}
Let $D \subset G$ be a conjugacy set, and let $g \in D$ be an isolated point. Then $g$ is c-smooth.
\end{proposition}

\begin{proof}
Let $U$ be an open subgroup such that $gU \cap D=\{g\}$, and let $V=U \cap gUg^{-1}$. We claim $V$ centralizes $g$, which will complete the proof. Let $h \in V$ be given. Certainly $hgh^{-1}$ belongs to $D$. Writing $hgh^{-1} = g(g^{-1}hgh^{-1})$, and using the fact that $g^{-1}hg$ and $h^{-1}$ both belong to $U$, we see that $hgh^{-1} \in gU$. Thus $hgh^{-1} \in gU \cap D=\{g\}$, and so $h$ centralizes $g$.
\end{proof}

\begin{remark} \label{rmk:cmooth-not-isolated}
We note that, in general, a c-smooth conjugacy class need not contain an isolated point\footnote{The basic idea in this remark was suggested to us by Gemini Pro.}. Indeed, let $H=\Aut(\bQ,<)$ be the automorphism group of the rational numbers as a total order, let $\Gamma$ be the space of all functions $\bQ \to \bZ/2\bZ$, and let $G=H \ltimes \Gamma$; this is (one version of) the wreath product of $H$ with $\bZ/2\bZ$. The group $G$ acts on the set $\Omega=\bQ \times \bZ/2\bZ$, where $H$ acts naturally on $\bQ$ and fixes $\bZ/2\bZ$ and $\gamma \in \Gamma$ acts by $\gamma \cdot (x, y) = (x, \gamma(x)+y)$, and $(G, \Omega)$ is an oligomorphic permutation group.

For $x \in \bR$, let $\alpha_x \in \Gamma$ be the indicator function of the interval $[x, \infty)$, i.e., $\alpha_x(y)$ is~0 if $y<x$ and~1 otherwise. Similarly define $\beta_x$, but using the interval $(x, \infty)$. Note that $\alpha_x=\beta_x$ if and only if $x$ is irrational. If $x$ is rational then $\alpha_x$ and $\beta_x$ are fixed by the open subgroup $H(x)$ of $H$, and thus they are c-smooth elements of $G$. One easily sees that $C=\{\alpha_x\}_{x \in \bQ}$ is a conjugacy class in $G$. If $(x_n)_{n \ge 1}$ is an increasing sequence of rational numbers converging to the rational number $x$ (in the usual topology) then $\alpha_{x_n}$ converges to $\alpha_x$. This shows that no element of $C$ is isolated in $C$, and thus establishes the first sentence of this remark. Note that if the sequence $(x_n)$ is instead decreasing then $\alpha_{x_n}$ converges to $\beta_x$.

Let $D$ be the closure of $C$ in $G$. One can show that $D$ consists of the elements $\alpha_x$ and $\beta_x$, with $x \in \bR$, together with the elements $\alpha_{-\infty}$ and $\alpha_{\infty}$, which, by convention, are the constant functions in $\Gamma$ taking values~1 and~0. Thus, as a set, $D$ consists of the real line together with an additional copy of $\bQ$, as well as the two extra points $\pm \infty$. As a space, $D$ is compact, totally disconnected, perfect (i.e., has no isolated points) and metrizable, and is thus homeomorphic to the Cantor set.
\end{remark}

\begin{remark} \label{rmk:isolated}
It is a consequence of the Baire category theorem that, in a complete metric space, any non-empty countable closed set has an isolated point. If $(G, \Omega)$ is a closed oligomorphic group with $\Omega$ countable then $G$ is a complete metric space, and so this statement can be applied. In particular, a closed conjugacy set that is a finite union of c-smooth conjugacy classes has an isolated point. Indeed, any (finitary smooth) $G$-set is countable, and so any c-smooth conjugacy class is countable.
\end{remark}

\begin{example} \label{ex:sym-conj}
Let $\fS$ be the infinite symmetric group, i.e., the group of all permutations of the set $\Omega = \{1,2,\ldots\}$. An element of $\fS$ is c-smooth if and only if it is finitary, i.e., fixes all but finitely many elements of $\Omega$. Thus the group of c-smooth elements is $\bigcup_{n \ge 1} \fS_n$, the ``small'' infinite symmetric group.

It is not difficult to describe all conjugacy classes in $\fS$ and the closure relations between them. Given $g \in \fS$, let $N_g(i)$ be the number of $i$-cycles in $g$ (allowing $i=\infty$). This defines a function
\begin{displaymath}
N_g \colon \{1,2,3,\ldots,\infty\} \to \{0,1,2,\ldots, \infty\}.
\end{displaymath}
This is a complete conjugacy invariant, i.e., $g$ and $h$ are conjugate if and only if $N_g=N_h$. Moreover, a function $N$ occurs as $N_g$ for some $g \in \fS$ if and only if $N(\infty)>0$ or $\sum_{i \ge 1} N(i)=\infty$. Given a function $N$ as above, let $C(N)$ be the corresponding conjugacy class, i.e., the set of $g$'s with $N_g=N$. Let $M$ be a second such function. One can show that $C(M)$ belongs to the closure of $C(N)$ if and only if $M(i) \le N(i)$ for all finite $i$, and if $M(\infty)>0$ then either $N(\infty)>0$ or $N(i)>0$ for infinitely many $i$.

For example, let $C \subset \fS$ be the set of transpositions. This is a c-smooth conjugacy class. We have $C=C(N)$ where $N(1)=\infty$, $N(2)=1$ and $N(i)=0$ for all $i \ge 3$. The closure $D$ of $C$ is the set $C \cup \{1\}$. The elements of $C$ are isolated in $D$, while~1 is not; this is a special case of Proposition~\ref{prop:finitary-isolated}.
\end{example}

\begin{remark}
If $(G, \Omega)$ is oligomorphic then $G_{\sm}$ is locally finite, in the sense that any finitely generated subgroup is finite. Indeed, if $g_1, \ldots, g_n \in G_{\sm}$ then there is an open subgroup $U$ centralizing each $g_i$, and so the $g_i$'s are contained in the group $\Aut_U(\Omega)$ of automorphisms of $\Omega$ as a $U$-set, which is finite \cite[Proposition~2.8]{repst}.  If $G$ is pro-oligomorphic, then $G_{\sm}$ is locally profinite.
\end{remark}

\section{Drinfeld centers: the quasi-regular case} \label{s:qr}

This section is the heart of the paper: we study the Drinfeld center of an oligomorphic tensor category associated to a quasi-regular measure. We give an explicit description of objects of the center in terms of $\cO(\pi)$-structures. Using this, we define the class of Yetter--Drinfeld modules, and a notion of support for an object in the center. We then prove two important theorems (Theorem~\ref{thm:Ztriv} and~\ref{thm:YD}), which essentially assert that if one has tight enough control on support then Yetter--Drinfeld modules account for all of the center.

\subsection{Set-up}

Let $G$ be a pro-oligomorphic group. Let $\mu$ be a quasi-regular measure for $G$ valued in the field $k$ such that nilpotent endomorphisms in $\uPerm(G, \mu)$ have trace~0. Let $\cT = \uRep(G, \mu)$ be the abelian envelope of $\uPerm(G, \mu)$. We identify $\cT$ with the category of finite length smooth modules over the completed group algebra $A(G)=A(G, \mu)$, and we identify $\Ind{\cT}$ with the category of all smooth $A(G)$-modules; see \S \ref{ss:cga}. We let $\pi$ be the fundamental group of $\cT$, and write $\cO(\pi)$ for its coordinate ring (\S \ref{ss:fund}).

\subsection{$\cO(\pi)$-structures} \label{ss:Opi}

We begin by giving a concrete description of $\cO(\pi)$, based on the following concept:

\begin{definition} \label{defn:pi-struct}
Let $R$ be an associative unital algebra in $\Ind{\cT}$. A \defn{$\cO(\pi)$-structure} on $R$ is a rule $\beta$ that assigns to every $G$-set $X$ and pair of elements $x,y \in X$ an element $\beta^X_{x,y} \in R$, such that the following conditions hold:
\begin{enumerate}
\item We have $g \beta^X_{x,y}=\beta^X_{gx,gy}$ for $g \in G$.
\item We have $\beta^X_{\ast,\ast}=1$ if $X=\bbone$ is a one-point $G$-set.
\item We have $\beta^{X \times Y}_{(x_1,y_1), (x_2,y_2)}=\beta^X_{x_1,x_2} \cdot \beta^Y_{y_1,y_2}$.
\item If $f \colon Y \to X$ is a map of $G$-sets, $x \in X$, and $y \in Y$, then
\begin{displaymath}
\beta^X_{f(y),x} = \int_{f^{-1}(x)} \beta^Y_{y,z} \, dz, \qquad
\beta^X_{x,f(y)} = \int_{f^{-1}(x)} \beta^Y_{z,y} \, dz.
\end{displaymath}
\end{enumerate}
\end{definition}

Note that $\cO(\pi)$-structures are functorial: if $\beta$ is an $\cO(\pi)$-structure on $R$ and $\phi \colon R \to S$ is an algebra homomorphism then $\phi(\beta)$ is an $\cO(\pi)$-structure on $S$. The following proposition explains the significance of the above definition.

\begin{proposition} \label{prop:Opi-algebra}
The algebra $\cO(\pi)$ carries a tautological $\cO(\pi)$-structure $\alpha$ that is universal, in the following sense: if $R$ is an algebra in $\Ind{\cT}$ with an $\cO(\pi)$-structure $\beta$ then there is a unique algebra homomorphism $\phi \colon \cO(\pi) \to R$ such that $\beta=\phi(\alpha)$.
\end{proposition}

\begin{proof}
Define $\alpha^X_{x,y}$ to be the image of $\delta_{x,y}$ under the canonical map $\cC(X \times X) \to \cO(\pi)$. This clearly satisfies Definition~\ref{defn:pi-struct}(a). It satisfies Definition~\ref{defn:pi-struct}(b,c) by definition of the algebra structure on $\cO(\pi)$. Finally, suppose $f \colon Y \to X$ is a map of $G$-sets, $y \in Y$, and $x \in X$. Then the two elements
\begin{displaymath}
f_*(\delta_y) \otimes \delta_x = \delta_{f(y),x}, \qquad
\delta_y \otimes f^*(\delta_x) = \int_{f^{-1}(x)} \delta_{y,z} \, dz
\end{displaymath}
have the same image in $\cO(\pi)$, by the definition of coend. This gives the first identity in Definition~\ref{defn:pi-struct}(d), and the second is similar. Thus $\alpha$ is an $\cO(\pi)$-structure. Note that since every object of $\cT$ is a quotient of some permutation module, the maps $\cC(X \times X) \to \cO(\pi)$ are jointly surjective. Since $\cC(X \times X)$ is generated by point masses, it follows that the $\alpha^X_{x,y}$ generate $\cO(\pi)$ as an object of $\Ind{\cT}$, i.e., as an $A(G)$-module.

Now, let $R$ and $\beta$ be given. For each $G$-set $X$, we have a map in $\cT$
\begin{displaymath}
\phi_X \colon \cC(X \times X) \to R, \qquad \delta_{x,y} \mapsto \beta^X_{x,y}.
\end{displaymath}
This uses the mapping property for Schwartz space (\S \ref{ss:cga}) and Definition~\ref{defn:pi-struct}(a). Define
\begin{displaymath}
\tilde{\phi} \colon \bigoplus_X \cC(X \times X) \to R
\end{displaymath}
to be the sum of these maps, where $X$ varies over isomorphism class representatives of $G$-sets. Suppose $f \colon Y \to X$ is a map of $G$-sets, $a \in \cC(Y)$, and $b \in \cC(X)$. Then we have
\begin{displaymath}
\tilde{\phi}(f_*a \otimes b) = \tilde{\phi}(a \otimes f^*b), \qquad
\tilde{\phi}(f^*b \otimes a) = \tilde{\phi}(b \otimes f_*a)
\end{displaymath}
by Definition~\ref{defn:pi-struct}(d). Suppose $g \colon Y \to Z$ is another map of $G$-sets, and $c \in \cC(Z)$. Then, by the above, we have
\begin{displaymath}
\tilde{\phi}(f_* g^*c \otimes a) = \tilde{\phi}(c \otimes g_*f^* a).
\end{displaymath}
Note that $g_* f^*$ is the dual of $f_* g^*$. Now, every map $\cC(Z) \to \cC(X)$ is a linear combination of maps of the form $f_* g^*$. We thus see that for any map $A \colon \cC(Z) \to \cC(X)$ in $\cT$ we have
\begin{displaymath}
\tilde{\phi}(Ac \otimes a) = \tilde{\phi}(c \otimes A^* a),
\end{displaymath}
where $A^* \colon \cC(X) \to \cC(Z)$ is the dual of $A$. It follows that $\tilde{\beta}$ defines a map
\begin{displaymath}
\phi \colon \Coend(\cP \to \cT) \to R,
\end{displaymath}
where $\cP=\uPerm(G, \mu)$, and $\cP \to \cT$ is the inclusion. The map $\phi$ is a unital algebra homomorphism by Definition~\ref{defn:pi-struct}(b,c). Since every object of $\cT$ is a quotient of an object of $\cP$, the above coend is isomorphic to $\cO(\pi)$. It follows from the construction that $\phi(\alpha)=\beta$. Uniqueness of $\phi$ follows since the $\alpha$'s generate $\cO(\pi)$.
\end{proof}

We now record some simple properties of $\cO(\pi)$-structures.

\begin{proposition} \label{prop:pi-basic}
Let $R$ be an algebra in $\Ind{\cT}$ with an $\cO(\pi)$-structure $\beta$.
\begin{enumerate}
\item Let $f \colon Y \to X$ be an isomorphism of $G$-sets. Then $\beta^Y_{y_1,y_2} = \beta^X_{f(y_1), f(y_2)}$.
\item Let $X$ be a $G$-set and let $x,y \in X$. If $x$ and $y$ belong to the same orbit $W$ then $\beta^X_{x,y}=\beta^W_{x,y}$. Otherwise, $\beta^X_{x,y}=0$.
\item Let $X$ be a transitive $G$-set, and let $x$, $y$, and $z$ be elements of $X$ with $y \ne z$. Then $\beta^X_{x,y}$ is idempotent, and $\beta^X_{x,y} \beta^X_{x,z}=0$.
\end{enumerate}
\end{proposition}

\begin{proof}
(a) This follows directly from Definition~\ref{defn:pi-struct}(d).

(b) Let $W$ be the orbit containing $x$ and let $f \colon W \to X$ be the inclusion. Definition~\ref{defn:pi-struct}(d) gives the stated result.

(c) Let $\Delta(X)$ be the diagonal in $X \times X$. We have
\begin{displaymath}
\beta^X_{x,y} \beta^X_{x,y} = \beta^{X \times X}_{(x,x),(y,y)} = \beta^{\Delta(X)}_{(x,x), (y,y)} = \beta^X_{x,y},
\end{displaymath}
where in the first step we used Definition~\ref{defn:pi-struct}(c), in the second the above statement (b), and in the third the above (a). Similarly, we have
\begin{displaymath}
\beta^X_{x,y} \beta^X_{x,z} = \beta^{X \times X}_{(x,x),(y,z)} = 0,
\end{displaymath}
since $(x,x)$ and $(y,z)$ belong to different orbits on $X \times X$.
\end{proof}

\subsection{Tensor products of $\cO(\pi)$-structures} \label{ss:pi-tensor}

Recall that $\cO(\pi)$ is a Hopf algebra, and that the canonical map $\cC(X) \otimes \cC(X)^* \to \cO(\pi)$ is compatible with coproducts. The coproduct on the former is the map
\begin{displaymath}
\cC(X \times X) \to \cC(X \times X \times X \times X), \qquad \delta_{x,y} \mapsto \int_X \delta_{x,u,u,y} \, du.
\end{displaymath}
It follows that the coproduct on $\cO(\pi)$ satisfies
\begin{displaymath}
\alpha^X_{x,y} \mapsto \int_X \alpha^X_{x,u} \otimes \alpha^X_{u,y} \, du.
\end{displaymath}
Now, let $R$ and $S$ be algebras in $\Ind{\cT}$ with $\cO(\pi)$-structures $\beta$ and $\gamma$. The coproduct on $\cO(\pi)$ defines an $\cO(\pi)$-structure $\epsilon$ on $R \otimes S$. From the above formula for the coproduct, we see that $\epsilon$ is given explicitly by
\begin{displaymath}
\epsilon^X_{x,y} = \int_X \beta^X_{x,u} \otimes \gamma^X_{u,y} \, du.
\end{displaymath}
For the sake of completeness, we give a direct proof that $\epsilon$ is indeed an $\cO(\pi)$-structure.

\begin{proposition}
The rule $\epsilon$ defines an $\cO(\pi)$-structure on $R \otimes S$.
\end{proposition}

\begin{proof}
The first three conditions of Definition~\ref{defn:pi-struct} are simple verifications left to the reader. We explain the fourth. Thus let $f \colon Y \to X$ be a map of $G$-sets, let $y \in Y$, and let $x \in X$. We have
\begin{align*}
\epsilon^X_{f(y),x}
&= \int_X \beta^X_{f(y),u} \otimes \gamma^X_{u,x} \, du
= \int_X \int_{f^{-1}(u)} \beta^Y_{y,z} \otimes \gamma^X_{u,x} \, dz \, du
= \int_Y \beta^Y_{y,z} \otimes \gamma^X_{f(z),x} \, dz \\
&= \int_Y \int_{f^{-1}(x)} \beta^Y_{y,z} \otimes \gamma^Y_{z,w} \, dw \, dz
= \int_{f^{-1}(x)} \epsilon^Y_{y,w} \, dw,
\end{align*}
which verifies the first identity in (d). The second is similar.
\end{proof}

\subsection{$\cO(\pi)$-modules and the Drinfeld center}

We have seen, very generally, that $\cZ(\Ind{\cT})$ is identified with the category of $\cO(\pi)$-modules (Proposition~\ref{prop:Opi-center}). We now explain how this works concretely in the present case.

Let $M$ be an object of $\Ind{\cT}$, and let
\begin{displaymath}
\uEnd(M) = \bigcup_{U \subset G} \End_{A(U)}(M),
\end{displaymath}
where $U$ varies over open subgroups of $G$ and $A(U) \subset A(G)$ is the completed group algebra of $U$. This is naturally an object of $\Ind{\cT}$ by \cite[Proposition~12.13]{repst}, and easily seen to carry the structure of an algebra in this category. If $M$ has finite length then we have a natural identification of $\uEnd(M)$ with $M^* \otimes M$ \cite[Corollary~12.17]{repst}. If $R$ is any algebra in $\Ind{\cT}$ then giving an $R$-module structure on $M$ internal to $\Ind{\cT}$ is equivalent to giving an algebra homomorphism $R \to \uEnd(M)$ internal to $\Ind{\cT}$.

From the above discussion, we see that giving an $\cO(\pi)$-module structure on $M$ is equivalent to giving an $\cO(\pi)$-structure on $\uEnd(M)$. Concretely, this means that for each $x,y \in X$ we have a map $\beta^X_{x,y} \colon M \to M$ that is equivariant for $A(U)$, for some $U$ depending on $X$, $x$ and $y$, and these elements satisfy the conditions of Definition~\ref{defn:pi-struct}. The half-braiding on $M$ is given on Schwartz spaces by
\begin{displaymath}
\cC(X) \otimes M \to M \otimes \cC(X), \qquad
\delta_x \otimes m \mapsto \int_X \beta^X_{x,y}(m) \otimes \delta_y \, dy.
\end{displaymath}
It is not difficult to verify directly from the definition of $\cO(\pi)$-structure that this does indeed satisfy the axioms of a half-braiding, at least relative to permutation modules. The formula for the half-braiding with respect to a general module $N$ is given by
\begin{displaymath}
N \otimes M \to M \otimes N, \qquad
n \otimes m \mapsto \int_{G/U} \beta^{G/U}_{1,g}(m) \otimes gn \, dg,
\end{displaymath}
where $U$ is any open subgroup stabilizing $n$. To see this, use the naturality of the half-braiding and the map $\cC(G/U) \to N$ satisfying $\delta_1 \mapsto n$.

\subsection{Yetter--Drinfeld modules} \label{ss:yd}

Let $G_{\sm}$ be the subgroup of c-smooth elements in $G$ (\S \ref{ss:csmooth}). The group $G$ acts smoothly on $G_{\sm}$ by conjugation. Suppose $C \subset G_{\sm}$ is a finitary $G$-subset, i.e., a finite union of c-smooth conjugacy classes. For a $G$-set $X$ and elements $x,y \in X$, define an element $\beta^X_{x,y}$ of $\cC(C)$ by
\begin{displaymath}
\beta^X_{x,y} = \int_{y=gx} \delta_g \, dg,
\end{displaymath}
where the integral is over those elements $g \in C$ satisfying the stated equation. In other words, $\beta^X_{x,y}$ is the Schwartz function on $C$ that maps~$g$ to~1 if $y=gx$, and~0 otherwise.

\begin{proposition}
The above elements define an $\cO(\pi)$-structure on $\cC(C)$.
\end{proposition}

\begin{proof}
We verify the conditions of Definition~\ref{defn:pi-struct}.

(a) For $h \in G$, we have
\begin{displaymath}
h \beta^X_{x,y} = \int_{y=gx} \delta_{hgh^{-1}} \, dg
= \int_{y=h^{-1}ghx} \delta_g \, dg = \int_{hy=ghx} \delta_g \, dg = \beta^X_{hx, hy},
\end{displaymath}
as required. Note that $G$ acts on $C$ by conjugation, which explains the integrand in the second step.

(b) If $X$ is a one-point $G$-set then $\beta^X_{x,x} = \int_C \delta_g \, dg$, which is the identity element of $\cC(C)$.

(c) We have
\begin{displaymath}
\beta^X_{x_1,x_2} \beta^Y_{y_1,y_2} = \int_{x_2=gx_1} \int_{y_2=hy_1} \delta_g \delta_h \, dg \, dh
= \int_{(x_2,y_2)=g(x_1,y_1)} \delta_g \, dg = \beta^{X \times Y}_{(x_1,y_1),(x_2,y_2)}.
\end{displaymath}
In the second step, we used the fact that $\delta_g \delta_h$ is $\delta_g$ if $g=h$, and vanishes otherwise.

(d) Let $f \colon Y \to X$ be a map of $G$-sets, let $x \in X$, and let $y \in Y$. Let $F=f^{-1}(x)$. Then
\begin{displaymath}
\beta^X_{f(y),x} = \int_{x=gf(y)} \delta_g \, dg = \int_{gy \in F} \delta_g \, dg = \int_F \int_{z=gy} \delta_g \, dg \, dz  = \int_F \beta^Y_{y,z} \, dz,
\end{displaymath}
as required. The other condition is similar.
\end{proof}

We thus see that if $M$ is a $\cC(C)$-module then $M$ is naturally an object of $\cZ(\Ind{\cT})$. The half-braiding is given on Schwartz spaces by
\begin{displaymath}
N \otimes M \to M \otimes N, \qquad n \otimes m \mapsto \int_C \delta_g m \otimes gn \, dg.
\end{displaymath}
We can restate this in a slightly nicer way. Put $M_g=\delta_g M$. Then $M$ is the direct integral of the $M_g$'s, in the sense that for any $m \in M$ we have $m=\int_C \delta_g m \, dg$ and $\delta_g m \in M_g$. It thus suffices to specify the half-braiding for $m \in M_g$. In this case, it is given by
\begin{displaymath}
n \otimes m \mapsto m \otimes gn.
\end{displaymath}
This is directly analogous to the case of finite groups (\S \ref{ss:finite-center}).

\begin{definition} \label{defn:yd}
A \defn{Yetter--Drinfeld} module is a $\cC(C)$-module in $\cT$, for some finitary $G$-subset $C \subset G_{\sm}$. The \defn{category of Yetter--Drinfeld modules} is the 2-colimit
\begin{displaymath}
\cY = \cY(G, \mu) = \varinjlim \Mod^{\fin}_{\cC(C)}
\end{displaymath}
over $C \subset G_{\sm}$. Recall that the $\fin$ superscript means we consider modules in $\cT$.
\end{definition}

\begin{remark} \label{rmk:yd}
We make a number of remarks on the above definition.

(a) If $C \subset C'$ are finitary $G$-subsets of $G_{\sm}$ then the pull-back map $\cC(C') \to \cC(C)$ is an algebra homomorphism, and easily seen to be compatible with the $\cO(\pi)$-structures. This is how the transition maps in the 2-colimit are defined.

(b) Let $\cC(G_{\sm})$ be the Schwartz space on $G_{\sm}$; by definition, this consists of smooth functions with finitary support. Explicitly, if $\{C_i\}_{i \in I}$ are the $G$-orbits on $G_{\sm}$ then
\begin{displaymath}
\cC(G_{\sm}) = \bigoplus_{i \in I} \cC(C_i).
\end{displaymath}
This is an object of $\Ind{\cT}$. Pointwise multiplication defines a non-unital algebra structure on $\cC(G_{\sm})$. This multiplication has approximate units\footnote{The Hecke algebra of a $p$-adic group has a similar structure; see \cite[\S 3.4]{Bump}.}: for any finite subset $J$ of $I$, there is an element $1_J$ of $\cC(G_{\sm})$ that acts as the identity on these factors. By a $\cC(G_{\sm})$-module, we mean one that respects the approximate identities, i.e., for any element $m$ we have $1_J \cdot m=m$ for some $J$. We have $\cY = \Mod^{\fin}_{\cC(G_{\sm})}$ and $\Ind{\cY} = \Mod_{\cC(G_{\sm})}$.

(c) The category $\cY$ is equivalent to the direct sum of the categories of $\cC(C_i)$-modules over $i \in I$. In general, for an open subgroup $U$ of $G$, the category of $\cC(G/U)$-modules in $\cT$ is equivalent to $\uRep(U, \mu)$. We thus have
\begin{displaymath}
\cY = \bigoplus_{i \in I} \uRep(Z(g_i), \mu),
\end{displaymath}
where $g_i$ is any element of $C_i$. Note the similarity to the case of finite groups (see \S \ref{ss:finite-center}).


(d) The category $\cY$ is not intrinsic to $\cT$: it depends on the choice of $G$. It is possible for two oligomorphic groups to give rise to the same pre-Tannakian category $\cT$, but for the categories of Yetter--Drinfeld modules to be inequivalent; we will see such an example in \S \ref{s:classical}. There is always a ``best'' choice of pro-oligomorphic group for a given $\cT$, namely, the oligomorphic fundamental group $\pi^{\rm olig}(\cT)$ constructed in \cite{discrete}. Thus one can canonically associate to $\cT$ the category of Yetter--Drinfeld modules for $\pi^{\rm olig}(\cT)$. However, it can be difficult to compute $\pi^{\rm olig}(\cT)$, and so this approach seems to be of limited use for computing $\cZ(\cT)$.
\end{remark}

As we have seen, every Yetter--Drinfeld module carries a natural half-braiding. Thus there is a natural functor
\begin{displaymath}
\Ind{\cY} \to \cZ(\Ind{\cT}).
\end{displaymath}
This functor is obviously faithful. We now investigate fullness. Recall that if $X$ is a subset of $G$ then a point $g \in X$ is \defn{isolated} if there exists an open subgroup $U$ such that $gU \cap X =\{g\}$.

\begin{proposition}
Let $C \subset G_{\sm}$ be a finitary $G$-subset. Suppose that every non-empty $G$-subset of $C$ contains an isolated point. Then the natural map $\cO(\pi) \to \cC(C)$ is surjective.
\end{proposition}

\begin{proof}
Let $g \in C$ be an isolated point, let $C_1$ be the conjugacy class of $g$, and write $C=C_1 \sqcup C_2$. Let $U$ be an open subgroup such that $gU \cap C = \{g\}$. We have
\begin{displaymath}
\beta^{G/U}_{1,g} = \int \delta_h \, dh = \delta_g,
\end{displaymath}
since the integral is taken over elements $h$ in $gU \cap C =\{g\}$. We thus see that the image of $\cO(\pi)$ in $\cC(C)$ contains the element $\delta_g$, and so it contains all of $\cC(C_1)$. Since $\cO(\pi)$ surjects onto $\cC(C_2)$ (by induction on the number of orbits), the result follows.
\end{proof}

\begin{corollary} \label{cor:YD-full}
Suppose every finitary $G$-subset of $G_{\sm}$ contains an isolated point. Then the natural functor $\Ind{\cY} \to \cZ(\Ind{\cT})$ is fully faithful.
\end{corollary}

\begin{remark}
Every finitary $G$-subset of $G_{\sm}$ contains an isolated point if and only if every c-smooth conjugacy class is discrete. Indeed, the former clearly implies the latter, while the reverse follows from the fact that, in any topological space, a finite union of discrete sets contains an isolated point.
\end{remark}

When the conclusion of the corollary holds, we can regard ``Yetter--Drinfeld'' as a condition on objects of $\cZ(\Ind{\cT})$, as opposed to requiring extra data. For all examples considered in this paper, and, indeed, all examples we know, the conclusion of the corollary does hold.

We now examine the tensor product of Yetter--Drinfeld modules. If $C_1$ and $C_2$ are finite unions of smooth conjugacy classes in $G$ then so is the set $C_1 C_2$ of all products $gh$ with $g \in C_1$ and $h \in C_2$. Moreover, the multiplication map $C_1 \times C_2 \to C_1 C_2$ is a map of $G$-sets (for the conjugation action).

\begin{proposition}
Let $C_1$ and $C_2$ be finite unions of smooth conjugacy classes in $G$, let $C=C_2 C_1$, and let $f \colon C_1 \times C_2 \to C$ be the multiplication map, i.e., $f(h_1,h_2)=h_2 h_1$. Then
\begin{displaymath}
f^* \colon \cC(C) \to \cC(C_1) \otimes \cC(C_2)
\end{displaymath}
is a homomorphism of $\cO(\pi)$-algebras.
\end{proposition}

\begin{proof}
The pullback map on Schwartz spaces is always an algebra homomorphism. We verify that it is compatible with the $\cO(\pi)$-structures. Let $\beta$, $\gamma$, and $\epsilon$ denote the $\cO(\pi)$-structures on $C_1$, $C_2$, and $C$, and let $\xi$ be the $\cO(\pi)$-structure on $\cC(C_1) \otimes \cC(C_2)$. We must verify $\xi=f^*(\epsilon)$. We have
\begin{displaymath}
\xi^X_{x,y} = \int_X \beta^X_{x,u} \otimes \gamma^X_{u,y} \, du
= \int_X \int_{u=h_1x} \int_{y=h_2u} \delta_{h_1,h_2} \, dh_2 \, dh_1 \, du
= \int_{x=h_2h_1y} \delta_{h_1,h_2} \, d(h_1, h_2)
\end{displaymath}
where $h_1 \in C_1$ and $h_2 \in C_2$. In the first step, we used the explicit formula for $\xi$ from \S \ref{ss:pi-tensor}, and in the second the definitions of $\beta$ and $\gamma$. The final integral is taken over all $(h_1,h_2) \in C_1 \times C_2$ satisfying the stated equation. We have
\begin{displaymath}
f^*(\epsilon^X_{x,y}) = \int_{y=gx} f^*(\delta_g) \, dg
= \int_{y=gx} \int_{g=h_2h_1} \delta_{h_1,h_2} \, d(h_1, h_2) \, dg
= \int_{x=h_2h_1y} \delta_{h_1,h_2} \, d(h_1, d_2)
\end{displaymath}
This agrees with $\xi^X_{x,y}$, which completes the proof.
\end{proof}

The proposition implies that the category $\Ind{\cY}$ of Yetter--Drinfeld modules admits a monoidal structure. This monoidal structure is compatible with the one on $\cZ(\Ind{\cT})$. Note that when the functor $\Ind{\cY} \to \cZ(\Ind{\cT})$ is fully faithful, these statements simply mean that the class of Yetter--Drinfeld modules in $\cZ(\Ind{\cT})$ is closed under tensor product.

\subsection{Support} \label{ss:support}

Let $R$ be an $\cO(\pi)$-algebra. For a transitive $G$-set $X$, define $\fA_R(X)$ to be the set of pairs $(x,y) \in X \times X$ such that $\beta^X_{x,y} \ne 0$. Recall from Definition~\ref{defn:cset} the notion of c-functor.

\begin{proposition}
The rule $\fA_R$ is a c-functor.
\end{proposition}

\begin{proof}
Let $f \colon Y \to X$ be a map of $G$-sets, and suppose $(y,z)$ belongs to $\fA_R(Y)$, meaning $\beta^Y_{y,z}$ is non-zero. By Definition~\ref{defn:pi-struct}(d), we have
\begin{displaymath}
\beta^X_{f(y), f(z)} = \int_{f^{-1}(f(z))} \beta^Y_{y,u} \, du
\end{displaymath}
Multiplying by $\beta^Y_{y,z}$ and using Proposition~\ref{prop:pi-basic}, we find
\begin{equation} \label{eq:beta1}
\beta^X_{f(y), f(z)} \cdot \beta^Y_{y,z} = \beta^Y_{y,z}.
\end{equation}
Thus $\beta^X_{f(y), f(z)}$ is non-zero. We have therefore shown that $f$ maps $\fA_R(Y)$ into $\fA_R(X)$, which verifies Definition~\ref{defn:cset}(a).

Now, let $(x_1, x_2)$ be an element of $\fA(X)$, and let $y_1$ be a lift of $x_1$ to $Y$. By Definition~\ref{defn:pi-struct}(d), we have
\begin{displaymath}
\beta^X_{x_1,x_2} = \int_{f^{-1}(x_2)} \beta^Y_{y_1,y_2} \, dy_2.
\end{displaymath}
Since the left side is non-zero, the integrand is not identically zero. Thus there exists a lift $y_2$ of $x_2$ such that $\beta^Y_{y_1,y_2}$ is non-zero, meaning $(y_1,y_2) \in \fA_R(Y)$. This verifies Definition~\ref{defn:cset}(b), and condition (c) is similar. Thus $\fA_R$ is a c-functor.
\end{proof}

\begin{definition}
The \defn{support} of an $\cO(\pi)$-algebra $R$ is $\fA_R$. The \defn{support} of an $\cO(\pi)$-module is the support of its endomorphism algebra.
\end{definition}

Recall that in nice situations, c-functors correspond to closed conjugacy sets in $G$ (Theorem~\ref{thm:cset}); we can then think of the support of $R$ as such an object. This informs our intuition about support.

\begin{example} \label{ex:support}
We give a few examples of support.

(a) Consider $R=\bone$ as an $\cO(\pi)$-algebra via the augmentation map. Then $\beta^X_{x,y}$ is~1 if $x=y$ and~0 otherwise. Thus $\fA_R(X)$ is the diagonal in $X$. Note that $\fA_R=\fA_D$ with $D=\{1\}$.

(b) Suppose $C$ is a finite union of c-smooth conjugacy classes and $R=\cC(C)$ as in \S \ref{ss:yd}. Then $\fA_R=\fA_C$, that is, $\fA_R(X)$ consists of pairs $(x,gx)$ with $x \in X$ and $g \in C$. This follows directly from the definition of $\beta$. This specializes to the previous example if $C=\{1\}$.

(c) Consider $\cO(\pi)$ as an algebra over itself. Assuming $\cT$ is semi-simple\footnote{This should not be necessary, but it makes the justification easier.}, $\alpha^X_{x,y} \ne 0$ for all transitive $G$-sets $X$ and $x,y \in X$. Thus $\fA_R(X)=X \times X$ for all $X$, and so $\fA_R=\fA_D$ with $D=G$. To see this, let $X$ be a transitive $G$-set. Then we have $\cC(X) = \bone \oplus M$, where $\bone$ does not occur in $M$. The kernel of the natural map
\begin{displaymath}
(M \otimes M^*) \oplus (M \otimes \bone) \oplus (\bone \otimes M^*) \oplus (\bone \otimes \bone) = \cC(X) \otimes \cC(X)^* \to \cO(\pi)
\end{displaymath}
is contained within the first three summands on the left. The projection map $\cC(X \times X) \to \bone$ is the augmentation map. Since $\delta_{x,y}$ has non-zero image under this map, the result follows.
\end{example}

We say that an $\cO(\pi)$-algebra $R$ is \defn{trivial} if the structure map factors through the augmentation $\cO(\pi) \to \bone$. Concretely, this means $\beta^X_{x,y}$ is~1 if $x=y$ and~0 otherwise, for all $G$-sets $X$ and elements $x,y \in X$. Similarly, we say that an $\cO(\pi)$-module $M$ is \defn{trivial} if its endomorphism algebra is. Note that $M$ is trivial if and only if the half-braiding on $M$ is induced from the symmetry, i.e., $M$ belongs to the essential image of $\Ind{\cT} \to \cZ(\Ind{\cT})$. The following proposition shows that triviality can be detected by support.

\begin{proposition} \label{prop:triv-supp}
Let $R$ be an $\cO(\pi)$-algebra. Then $R$ is trivial if and only if its support is contained in the diagonal, i.e., $\fA_R(X)$ is contained in the diagonal of $X \times X$ for all $X$.
\end{proposition}

\begin{proof}
Throughout the proof we regard $\bone$ as an $\cO(\pi)$-algebra via the augmentation map. It is clear that if $\bone \to R$ is a map of $\cO(\pi)$-algebras then $\fA_R$ is contained in the diagonal. Conversely, suppose that $\fA_R$ is contained in the diagonal, i.e., $\beta^X_{x,y}$ is non-zero only for $x=y$. Let $W=\{w\}$ be a one-point $G$-set. We have
\begin{displaymath}
1 = \beta^W_{f(x),w} = \int_X \beta^X_{x,y} \, dy = \beta^X_{x,x}.
\end{displaymath}
The first equation comes from Definition~\ref{defn:pi-struct}(b), the second from Definition~\ref{defn:pi-struct}(d) applied to the map $X \to W$, and the third from the fact that $\beta^X_{x,y}$ vanishes if $x \ne y$. It follows that $\bone \to R$ is a map of $\cO(\pi)$-algebras, as required.
\end{proof}

We now explain that objects in $\cZ(\cT)$ have ``small'' support. This is what will eventually provide us leverage for understanding $\cZ(\cT)$. The following is the precise notion of ``small'' that is relevant.

\begin{definition} \label{defn:Tsmall}
A c-functor $\fA$ is \defn{$\cT$-small} if there exists a $G$-set $Y$ such that for every $G$-set $X$ there is an injection $\cC(\fA(X)) \to \cC(X \times Y)$ in $\cT$.
\end{definition}

\begin{example} \label{ex:csmooth-Tsmall}
Let $C$ be a c-smooth conjugacy class. Then we have a surjection of $G$-sets $f \colon X \times C \to \fA_C(X)$ given by $(x, g) \mapsto (x, gx)$. We thus have an injection $f^* \colon \cC(\fA_C(X)) \to \cC(X \times C)$ in $\cT$, and so $\fA_C$ is $\cT$-small. More generally, if $D$ is a finite union of c-smooth conjugacy classes then $\fA_D$ is $\cT$-small.
\end{example}

By definition, $\cC(\fA(X))$ is a $G$-subset of $\cC(X \times X)$. The latter is typically much larger than $\cC(X \times Y)$, if $Y$ is fixed and $X$ is allowed to vary; for instance, $X \times X$ will typically have many more orbits than $X \times Y$. Thus $\cT$-small is potentially a very restrictive condition.

\begin{proposition} \label{prop:supp-Tsmall}
Let $R$ be an $\cO(\pi)$-algebra in $\Ind{\cT}$. Then for any transitive $G$-set $X$, the map
\begin{displaymath}
\cC(\fA_R(X)) \to \cC(X) \otimes R, \qquad \delta_{x,y} \mapsto \delta_x \otimes \beta^X_{x,y}
\end{displaymath}
is injective. In particular, if $R$ is in $\cT$ then the support of $R$ is $\cT$-small.
\end{proposition}

\begin{proof}
Let $\phi \in \cC(\fA_R(X))$ belong to the kernel. We treat $\phi$ as a Schwartz function on $X \times X$ that is supported on $\fA_R(X)$. We thus have
\begin{displaymath}
0 = \int_{X \times X} \phi(x,y) (\delta_x \otimes \beta^X_{x,y}) \, d(x,y).
\end{displaymath}
It follows that for any $x \in X$ we have
\begin{displaymath}
0 = \int_X \phi(x,y) \beta^X_{x,y} \, dy
\end{displaymath}
Multiplying this equation by $\beta^X_{x,y}$ and using Proposition~\ref{prop:pi-basic}, we find
\begin{displaymath}
\phi(x,y) \beta^X_{x,y} = 0.
\end{displaymath}
This holds for all $(x,y) \in X$. If $(x,y) \in \fA_R(X)$ then $\beta^X_{x,y} \ne 0$, and so $\phi(x,y)=0$. We have thus shown that the map is injective. If $R$ belongs to $\cT$, there exists a $G$-set $Y$ and an injection $R \to \cC(Y)$ in $\cT$, which shows that $\fA_R$ is $\cT$-small.
\end{proof}

Putting the above discussion together, we obtain the following criterion for the Drinfeld center of $\cT$ to come entirely from $\cT$.

\begin{theorem} \label{thm:Ztriv}
Suppose that every $\cT$-small c-functor is contained in the diagonal. Then the natural functor $\cT \to \cZ(\cT)$ is an equivalence.
\end{theorem}

\begin{proof}
Let $M$ be an object of $\cZ(\cT)$. The support of $M$ is $\cT$-small (Proposition~\ref{prop:supp-Tsmall}), and therefore, by our assumption, contained in the diagonal. Thus the $\cO(\pi)$-module structure on $M$ factors through the augmentation (Proposition~\ref{prop:triv-supp}), which means that $M$ belongs to the essential image of $\cT$ in $\cZ(\cT)$. The result thus follows.
\end{proof}

\begin{remark}
If $C$ is a c-smooth conjugacy class in $G$ then $\fA_C$ is a $\cT$-small c-functor. Thus if every $\cT$-small c-functor is contained in the diagonal then there are no c-smooth conjugacy classes, other than the identity, and so there are no interesting Yetter--Drinfeld modules, i.e., we have $\cY = \cT$. It is therefore not altogether surprising that $\cZ(\cT)$ coincides with $\cT$ in this case.
\end{remark}

\subsection{The main theorem} \label{ss:Zmain}

We have just given a criterion for the center of $\cT$ to come entirely from $\cT$ (Theorem~\ref{thm:Ztriv}). We now prove a more flexible theorem, which will tell us that the center comes from Yetter--Drinfeld modules under certain hypotheses. Throughout \S \ref{ss:Zmain}, we assume that $G$ is a closed oligomorphic permutation group acting on a countable set $\Omega$. In this case, every c-functor $\fA$ has the form $\fA_D$ for a closed conjugacy set $D \subset G$ (Theorem~\ref{thm:cset}).

The following is the key result. In essence, it says that if an $\cO(\pi)$-algebra $R$ has an isolated point in its support, then $R$ admits a non-zero direct factor that is an algebra over one of the $\cO(\pi)$-algebras $\cC(C)$ from \S \ref{ss:yd}. We recall that an isolated point in a conjugacy set is necessarily c-smooth (Proposition~\ref{prop:isolated-smooth}).

\begin{proposition} \label{prop:isolated}
Let $R$ be a commutative algebra in $\Ind{\cT}$ with an $\cO(\pi)$-structure $\beta$, let $D \subset G$ be the closed conjugacy set with $\fA_R=\fA_D$, and suppose $g_0$ is an isolated point of $D$. Let $C \subset G$ be the conjugacy class of $g_0$.
\begin{enumerate}
\item Let $g \in C$ and let $U$ be an open subgroup. Then $\beta^{G/U}_{1,g}$ is independent of $U$ provided $U$ isolates $g$, meaning $gU \cap D=\{g\}$. Denote the common value by $\beta_g$.
\item The element $\beta_g$ is a non-zero idempotent, and satisfies $h\beta_g = \beta_{hgh^{-1}}$ for $h \in G$.
\item If $g,h \in C$ are distinct then $\beta_g \beta_h=0$.
\item The element $e=\int_C \beta_g dg$ is a non-zero $G$-invariant idempotent of $R$.
\item Let $X$ be a $G$-set, and let $x,y \in X$, and let $g \in C$. Then $\beta_g \beta^X_{x,y}$ is equal to $\beta_g$ if $gx=y$, and~0 otherwise.
\item The natural map $\cC(C) \to eR$ given by $\delta_g \mapsto \beta_g$ is a homomorphism of $\cO(\pi)$-algebras, for the $\cO(\pi)$-structure on $\cC(C)$ defined in \S \ref{ss:yd}.
\end{enumerate}
\end{proposition}

\begin{proof}
(a) Let $U$ and $V$ be two open subgroups isolating $g$; we note that at least one exists since $g$ is isolated in $D$. Replacing $U$ with $U \cap V$, we assume $U \subset V$. By Definition~\ref{defn:pi-struct}(d), we have
\begin{displaymath}
\beta^{G/V}_{1,g} = \int_{gV/U} \beta^{G/U}_{1,h} dh.
\end{displaymath}
Now, suppose $\beta^{G/U}_{1,h}$ is non-zero for some $h \in gV$. Then $(1,h)$ belongs to $\fA_R(G/U)$, which means there is $a \in S$ and $x \in G/U$ such that $(1,h)=(x, ax)$, which means there is $u \in U$ such that $h=au$. We also have $h=gv$ for some $v \in V$. We thus see $a=gvu^{-1}$ belongs to $gV \cap S$, and is therefore equal to $g$. Hence $h=g$ in $G/U$. The above equation thus reduces to $\beta^{G/V}_{1,g}=\beta^{G/U}_{1,g}$, which completes the proof of this statement.

(b) Let $V$ be an open subgroup isolating $g$. We have already seen that $\beta_g=\beta^{G/V}_{1,g}$ is idempotent (Proposition~\ref{prop:pi-basic}), and it is non-zero since $g \in D$, and therefore $(1,g) \in \fA_R(G/V)$. Now, let $h \in G$, and put $U=V \cap h^{-1}Vh$. Then
\begin{displaymath}
h \beta_g = h \beta^{G/U}_{1,g} = \beta^{G/U}_{h,hg} = \beta^{G/hUh^{-1}}_{1,hgh^{-1}} = \beta_{hgh^{-1}}.
\end{displaymath}
In the third step, we used the isomorphism of $G$-sets $G/U \to G/hUh^{-1}$, and the invariance of $\beta$ under isomorphisms (Proposition~\ref{prop:pi-basic}).

(c) Let $U$ be an open subgroup isolating both $g$ and $h$. Then $\beta_g=\beta^{G/U}_{1,g}$ and $\beta_h=\beta^{G/U}_{1,h}$ and $gU \ne hU$. Thus $\beta_g\beta_h=0$ by Proposition~\ref{prop:pi-basic}.

(d) This follows immediately from (b) and (c).

(e) Let $x,y \in X$ be given. Let $U$ be an open subgroup stabilizing $x$ and isolating $g$. Let $f \colon G/U \to X$ be the map $f(h)=hx$. By Definition~\ref{defn:pi-struct}(d), we have
\begin{displaymath}
\beta^X_{x,y} = \int_{hx=y} \beta^{G/U}_{1,h} \, dh,
\end{displaymath}
where the integral is over elements $h \in G/U$ satisfying the stated equation. Now multiply by $\beta_g=\beta^{G/U}_{1,g}$. By Proposition~\ref{prop:pi-basic}, $\beta^{G/U}_{1,g} \beta^{G/U}_{1,h}$ is non-zero only if $gU=hU$, in which case it is $\beta^{G/U}_{1,g}$. The result follows.

(f) Let $f \colon \cC(C) \to eR$ be the map in question, which exists by the mapping property for $\cC(C)$ (see \S \ref{ss:cga}) and statement (b) above. Since the $\beta_g$'s are orthogonal idempotents of $eR$ that integrate to the identity, it follows that $f$ is an algebra homomorphism. We now establish compatibility with the $\cO(\pi)$-structure. Let $\gamma$ be the $\cO(\pi)$-structure on $\cC(C)$; note that the $\cO(\pi)$-structure on $eR$ is simply $e \beta$. Let $X$ be a $G$-set and let $x,y \in X$ be given. We have
\begin{displaymath}
f(\gamma^X_{x,y}) = \int_{y=gx} f(\delta_g) \, dg = \int_{y=gx} \beta_g \, dg.
\end{displaymath}
On the other hand, we have
\begin{displaymath}
e \beta^X_{x,y} = \int_C \beta_g \beta^X_{x,y} \, dg = \int_{y=gx} \beta_g \, dg,
\end{displaymath}
where in the final step we used (e). Thus $f(\gamma^X_{x,y})=e \beta^X_{x,y}$, as required.
\end{proof}

\begin{corollary} \label{cor:isolated}
Let $M$ be an object of $\cZ(\Ind{\cT})$. Let $D \subset G$ be the closed conjugacy set such that $\fA_D$ is the support of $M$, and suppose that $D$ contains an isolated point $g_0$. Then, as an object of $\cZ(\Ind{\cT})$, $M$ decomposes as $M_1 \oplus M_2$ where $M_1$ is a non-zero object of $\Ind{\cY}$.
\end{corollary}

\begin{proof}
Let $R \subset \uEnd(M)$ be the image of $\cO(\pi)$. Let $e \in R$ be the idempotent produced by Proposition~\ref{prop:isolated}. We then have decompositions $R=R_1 \oplus R_2$ and $M=M_1 \oplus M_2$, where $R_1$ and $M_1$ are the image of $e$, and $R_2$ and $M_2$ are the image of $1-e$. Since $e$ is non-zero and $M$ is a faithful $R$-module, $M_1$ is non-zero. Let $C$ be the conjugacy class of $g_0$. By Proposition~\ref{prop:isolated}, we have an $\cO(\pi)$-algebra homomorphism $\cC(C) \to R_1$, and so $M_1$ is naturally a Yetter--Drinfeld module, as required.
\end{proof}

We say that a closed conjugacy set $D \subset G$ is \defn{$\cT$-small} if $\fA_D$ is $\cT$-small. We now come to our main result on $\cZ(\cT)$.

\begin{theorem} \label{thm:YD}
Suppose that every $\cT$-small closed conjugacy set in $G$ contains an isolated point. Then the natural functor $\cY \to \cZ(\cT)$ is an equivalence.
\end{theorem}

\begin{proof}
First, suppose that $C \subset G_{\sm}$ is a finitary $G$-subset, and let $D$ be the closure of $C$ in $G$. We have $\fA_D=\fA_C$, and this c-functor is $\cT$-small (Example~\ref{ex:csmooth-Tsmall}). Thus, by our assumption, $D$ has an isolated point, and it must belong to $C$. It thus follows from  Corollary~\ref{cor:YD-full} that the functor $\cY \to \cZ(\cT)$ is fully faithful.

We now prove essential surjectivity by induction on length. Let $M$ be a non-zero object of $\cZ(\cT)$. The support of $M$ is a $\cT$-small conjugacy set (Proposition~\ref{prop:supp-Tsmall}), and thus (the corresponding closed conjugacy set) contains an isolated point by assumption. Thus, by Corollary~\ref{cor:isolated}, $M$ decomposes as $M_1 \oplus M_2$ where $M_1$ is a non-zero Yetter--Drinfeld module. Since $M_1$ is non-zero, the length of $M_2$ is less than that of $M$, and so, by induction, $M_2$ is a Yetter--Drinfeld module. This completes the proof.
\end{proof}

\begin{remark}
Cconsider the oligomorphic group $G$ from Remark~\ref{rmk:cmooth-not-isolated}. One can show that $G$ admits a regular measure $\mu$ (over any field of characteristic $\ne 2$), and that there is a corresponding semi-simple pre-Tannakian category $\cT$. Let $C$ be the conjugacy class from the remark, and let $D$ be the closure of $C$. We have seen that $C$ does not contain any isolated point, and so the same is true for $D$. Since $\fA_C=\fA_D$, it follows from Example~\ref{ex:csmooth-Tsmall} that $D$ is $\cT$-small. Thus the hypothesis of Theorem~\ref{thm:YD} does not apply in this case. It would therefore be interesting to look for a more general version of the theorem that accommodates this example.
\end{remark}

\section{Results related to strong amalgamation} \label{s:strong}

In \S \ref{s:qr}, we introduced the notion of $\cT$-small c-functor (Definition~ \ref{defn:Tsmall}), and saw that if one can control these objects then one obtains a description of the center in terms of Yetter--Drinfeld modules. In this section, we develop some tools for controlling $\cT$-small c-functors when the oligomorphic group comes from a Fra\"iss\'e class satisfying the strong amalgamation property.

\subsection{Strong amalgamation}

Let $(G, \Omega)$ be an oligomorphic permutation group. Given a finite subset $A$ of $\Omega$, recall that $G(A)$ is the subgroup of $G$ consisting of permutations that fix each element of $A$. The \defn{definable closure} of $A$, denoted $\dcl{A}$, is the set $\Omega^{G(A)}$ of $G(A)$-fixed points on $\Omega$. The \defn{algebraic closure} of $A$, denoted $\acl{A}$, is the set of elements of $\Omega$ that have finite orbit under $G(A)$. Both $\dcl{A}$ and $\acl{A}$ are finite sets stable under $G(A)$, and we have
\begin{displaymath}
A \subset \dcl{A} \subset \acl{A}.
\end{displaymath}
We say that $A$ is \defn{definably closed} if $A=\dcl{A}$, and \defn{algebraically closed} if $A=\acl{A}$. For any $A$, the set $\dcl{A}$ is definably closed and the set $\acl{A}$ is algebraically closed.

\begin{proposition} \label{prop:dcl}
The following are equivalent:
\begin{enumerate}
\item Every finite subset of $\Omega$ is definably closed.
\item Every finite subset of $\Omega$ is algebraically closed.
\item Given finite subsets $A$, $B$, and $C$ of $\Omega$, with $A \cap C = B \cap C = \emptyset$, there exists $g \in G(C)$ such that $gA \cap B = \emptyset$.
\end{enumerate}
\end{proposition}

\begin{proof}
(a) $\Rightarrow$ (b) Let $A$ be a finite subset of $\Omega$. Suppose $C=\acl{A}$ is strictly larger than $A$, let $x$ be an element of $C$ not belonging to $A$, and put $B=C \setminus \{x\}$. Since $C$ is stable by $G(A)$, it is also stable by the subgroup $G(B)$. Since $G(B)$ fixes all elements of $C$ except $x$, it must also fix $x$. Thus $\dcl{B}=C$, a contradiction. We conclude that $A$ is algebraically closed.

(b) $\Rightarrow$ (a) is obvious.

(b) $\Leftrightarrow$ (c) is \cite[2.15]{CameronBook}.
\end{proof}

Suppose $\Omega$ is the Fra\"iss\'e limit of a Fra\"iss\'e class $\sF$, and $G=\Aut(\Omega)$. The class $\sF$ is said to have the \defn{strong amalgamation property} if the following condition holds: given embeddings $A \to B$ and $A \to C$ in $\sF$, there are embeddings $B \to D$ and $C \to D$ in $\sF$ that agree on $A$ and only on $A$, i.e., the images of $B \setminus A$ and $C \setminus A$ in $D$ are disjoint. One easily sees $\sF$ has the strong amalgamation property if and only if $G$ satisfies Proposition~\ref{prop:dcl}(c). See \cite[\S 2.7]{CameronBook} for additional background on strong amalgamation.

\subsection{Level} \label{ss:level}

Let $(G, \Omega)$ be an oligomorphic permutation group satisfying the equivalent conditions of Proposition~\ref{prop:dcl}. Recall that a transitive $G$-set is \defn{$\Omega$-smooth} if it appears in $\Omega^n$ for some $n$; we then define the \defn{level} of $X$, denoted $\lev{X}$, to be the minimal such $n$. We say that a $G$-set $X$ is \defn{$\Omega$-smooth} if each orbit on $X$ is $\Omega$-smooth, and we define the \defn{level} of $X$ to be the maximum of the level of orbits, with the convention that $\lev{\emptyset}=-\infty$.

\begin{proposition} \label{prop:level}
The level of the orbit of $x \in \Omega^n$ is the number of distinct coordinates of $x$.
\end{proposition}

\begin{proof}
Without loss of generality, suppose that $x_1, \ldots, x_m$ are distinct and $x_{m+1}, \ldots, x_n$ each belong to $\{x_1, \ldots, x_m\}$. Let $X \subset \Omega^n$ be the orbit of $x$, let $y=(x_1, \ldots, x_m)$, and let $Y \subset \Omega^m$ be the orbit of $y$. Then $X$ and $Y$ are isomorphic $G$-sets, so it suffices to show that $Y$ has level $m$. Clearly, $\lev{Y} \le m$. Suppose we have an injection of $G$-sets $i \colon Y \to \Omega^{\ell}$ with $\ell \le m$. Let $z=(z_1, \ldots, z_{\ell})$ be the image of $y$. Since $i$ is an isomorphism onto its image, $z$ and $y$ have the same stabilizer, call it $H$. By our assumption $(\ast)$, we have $\Omega^H=\{x_1, \ldots, x_m\}$ and $\Omega^H=\{z_1, \ldots, z_{\ell}\}$. Thus $\ell=m$, as required.
\end{proof}

\begin{proposition} \label{prop:level-prod}
Let $X$ and $Y$ be $\Omega$-smooth $G$-sets. Then $\lev(X \times Y)=\lev{X}+\lev{Y}$.
\end{proposition}

\begin{proof}
It suffices to treat the case where $X$ and $Y$ are transitive. Let $r=\lev{X}$ and let $s=\lev{Y}$. Identify $X$ with the orbit of some $x \in \Omega^r$, and identify $y$ with the orbit of some $y \in \Omega^s$. Note that $x$ and $y$ have distinct coordinates by Proposition~\ref{prop:level}. Pick $g \in G$ such that $gy$ has no coordinate in common with $x$, which is possible by Proposition~\ref{prop:dcl}(c). Then $X \times Y$ contains the orbit of $(x, gy) \in \Omega^{r+s}$, which has level $r+s$ by Proposition~\ref{prop:level}.
\end{proof}

\subsection{Finitary elements}

Let $(G, \Omega)$ be as in \S \ref{ss:level}. We say that a permutation $g$ of $\Omega$ is \defn{finitary} if it fixes all but finitely many elements of $\Omega$. We then define the \defn{support} of $g$ to be the finite set of elements $x \in \Omega$ such that $gx \ne x$, and we define the \defn{rank} of $g$ to be the cardinality of its support. If $A$ is the support of $g$ then $gA=A$, and so $g$ induces a permutation of the set $A$; in fact, if $B$ is any subset of $\Omega$ containing $A$ then $gB=B$, and so $g$ induces a permutation of $B$ fixing each element of $B \setminus A$. Note that $g$ has rank~0 if and only if $g$ is the identity.

\begin{proposition} \label{prop:smooth-is-finitary}
An element of $G$ is c-smooth if and only if it is finitary.
\end{proposition}

\begin{proof}
If $g$ is finitary with support $A$ then $G(A)$ centralizes $g$, and so $g$ is c-smooth. Now suppose $g$ is c-smooth. Let $A$ be a finite subset of $\Omega$ such that $G(A) \subset Z(g)$. If $B$ is any finite subset of $\Omega$ containing $A$ then $G(B) \subset Z(g)$, and so $g$ maps $\Omega^{G(B)}=B$ to itself bijectively. In particular, $gA=A$. Taking $B=A \cup \{x\}$, where $x \in \Omega \setminus A$, we see that $gx=x$. Thus $g$ fixes all elements of $\Omega \setminus A$, and is thus finitary.
\end{proof}

\begin{remark}
Even without assuming the conditions of Proposition~\ref{prop:dcl}, finitary elements are c-smooth. Here is an example of a c-smooth element that is not finitary. Let $\fS$ be the infinite symmetric group acting on $\Omega_0=\{1,2,\ldots\}$, and let $H=\langle a \rangle$ be cyclic of order two. Let $G=H \times \fS$ act on $\Omega=H \times \Omega_0$. Since $a$ is central in $G$, it is c-smooth, but it is clearly not finitary. This does not contradict Proposition~\ref{prop:smooth-is-finitary} since $(G, \Omega)$ does not satisfy the conditions of Proposition~\ref{prop:dcl}: indeed, the definable closure of $\{(1,1)\}$ is $\{(1,1), (a,1)\}$.
\end{remark}

\begin{proposition} \label{prop:finitary-classes}
For each $r$, there are only finitely many conjugacy classes of finitary rank $r$ elements of $G$.
\end{proposition}

\begin{proof}
Let $X$ be the set of triples $(A, B, f)$ where $A$ and $B$ are $r$ element subsets of $\Omega$ and $f \colon A \to B$ is a bijection. This is naturally a (finitary) $G$-set. Let $Y$ be the set of finitary rank $r$ elements in $G$, on which $G$ acts by conjugation. Define a function $\phi \colon Y \to X$ by $\phi(g)=(A, gA, g \vert_A)$, where $A$ is the support of $g$. Clearly, $\phi$ is injective and $G$-equivariant. Thus $Y$ has finitely many orbits.
\end{proof}

The next result gives a criterion for the group of finitary elements to be trivial.

\begin{proposition} \label{prop:no-finitary}
Suppose $\Omega$ is the Fra\"iss\'e limit of a Fra\"iss\'e class $\sF$ with strong amalgamation, and $G=\Aut(\Omega)$. Also suppose the following condition holds:
\begin{itemize}
\item[$(\ast)$] Given a finite structure $A \in \sF$, there is an embedding $A \to B$ with $B \in \sF$ such that any automorphism of $B$ fixing each element of $B \setminus A$ is the identity.
\end{itemize}
Then $G$ has no finitary elements except for the identity.
\end{proposition}

\begin{proof}
Let $g$ be a finitary automorphism of $\Omega$, and let $A \subset \Omega$ be its support. Let $B$ be as in $(\ast)$, and choose an embedding $B \to \Omega$ that is the identity on $A$, which is possible since $\Omega$ is homogeneous. Since $g$ is finitary and $B$ contains the support of $g$, it follows that $g$ induces an automorphism of $B$, which is necessarily the identity on $B \setminus A$. Thus $g \vert_B$ is the identity by $(\ast)$. Since $B$ contains the support of $g$, it follows that $g$ is the identity.
\end{proof}

\subsection{Small conjugacy sets} \label{ss:strong-conj}

We now suppose that $(G, \Omega)$ is a closed oligomorphic permutation group, with $\Omega$ countable, satisfying the equivalent conditions of Proposition~\ref{prop:dcl}. We say that a c-functor $\fA$ is \defn{small} if there is an integer $s$ such that $\lev(\fA(X)) \le s + \lev{X}$ for all $\Omega$-smooth transitive $G$-sets $X$. We say that a conjugacy set $D \subset G$ is \defn{small} if $\fA_D$ is. Note that this is a different condition than $\cT$-small (Definition~\ref{defn:Tsmall}), but we will see that the two are closely related.

\begin{proposition} \label{prop:small-conj}
Let $D \subset G$ be a conjugacy set. Then $D$ is small if and only if it is a union of finitely many c-smooth conjugacy classes of $G$.
\end{proposition}

\begin{proof}
Suppose $D$ is a finite union of c-smooth conjugacy classes. All elements of $D$ are finitary by Proposition~\ref{prop:smooth-is-finitary}. Let $s$ be the maximal rank of an element of $D$. Let $X$ be a transitive $\Omega$-smooth $G$-set of level $n$, and fix an embedding $X \subset \Omega^n$. Then $\fA_D(X)$ consists of points $(x, gx)$ with $x \in X$ and $g \in D$. Since $g$ has rank $\le s$, it follows that $(x, gx)$ has at most $n+s$ distinct coordinates. Thus $\fA_D(X)$ has level $\le n+s$ by Proposition~\ref{prop:level}, and so $D$ is small.

Now suppose $D$ is small. Let $s$ be such that $\lev(\fA_D(X)) \le s + \lev{X}$ for all $\Omega$-smooth transitive $G$-sets $X$. Let $g \in D$. Let $x \in \Omega^n$ have distinct coordinates, and let $X$ be the orbit of $x$, which has level $n$ by Proposition~\ref{prop:level}. Then $\fA_D(X)$ contains the orbit of $(x, gx) \in \Omega^{2n}$, and so $(x, gx)$ has at most $n+s$ distinct coordinates by Proposition~\ref{prop:level}. The following lemma implies that $g$ is finitary of rank at most $3s$. We thus see that $D$ is contained in the rank $\le 3s$ locus in $G$, which is a finite union of c-smooth conjugacy classes by Proposition~\ref{prop:smooth-is-finitary} and~\ref{prop:finitary-classes}.
\end{proof}

\begin{lemma} \label{lem:small-conj}
Let $g$ be a permutation of $\Omega$. Suppose there is some $s$ such that $\#(A \cup gA) \le s+ \# A$ for all finite subsets $A$ of $\Omega$. Then $g$ is finitary of rank $\le 3s$.
\end{lemma}

\begin{proof}
Suppose there is a subset $B$ of $\Omega$ of cardinality $3s+1$ such that $g$ does not fix any element of $\Omega$. We define $x_1, \ldots, x_{s+1} \in B$ inductively as follows. Suppose we have defined $x_1, \ldots, x_{i-1}$. Let $C=\{x_1, \ldots, x_{i-1}\}$ and let $D=C \cup gC \cup g^{-1}C$. Note that $D$ has at most $3(i-1)$ elements, and is therefore a proper subset of $B$. We define $x_i$ to be any element of $B \setminus D$. Now let $A=\{x_1, \ldots, x_{s+1}\}$. Then $A$ and $gA$ are disjoint, and so $\#(A \cup gA)=\# A + (s+1)$, a contradiction. The result thus follows.
\end{proof}

\begin{proposition} \label{prop:finitary-isolated}
Let $D \subset G$ be a small conjugacy set, and let $r$ be the maximum rank of an element of $D$. Then any element of $D$ of rank $r$ is isolated.
\end{proposition}

\begin{proof}
Let $g \in D$ have rank $r$, and let $A$ be its support. Suppose $h \in gG(A) \cap D$. Then $hx=gx$ for all $x \in A$, and so $hx \ne x$ for all $x \in A$. Since $h$ is finitary of rank $\le r$, it follows that $h$ must fix all elements of $\Omega \setminus A$. Thus $h=g$. We thus have $gG(A) \cap D=\{g\}$, which shows that $g$ is isolated.
\end{proof}

\begin{remark} \label{rmk:isolated-alt}
If $D$ is a small conjugacy set then so is $\ol{D}$ (since $\fA_D=\fA_{\ol{D}}$), and hence $\ol{D}$ is countable by Proposition~\ref{prop:small-conj}. Hence, the existence of an isolated point in $D$ also follows from Remark~\ref{rmk:isolated}.
\end{remark}

\begin{remark}
The argument from Proposition~\ref{prop:finitary-isolated} can be used to show that if $g$ is a finitary element of rank $r$ then the closure of the conjugacy class $C_g$ of $g$ is a finite union $C_g \cup C_{h_1} \cup \cdots \cup C_{h_s}$, where each $h_i$ is finitary of rank $<r$.
\end{remark}

\subsection{Small versus $\cT$-small} \label{ss:small-Tsmall}

Let $(G, \Omega)$ be as in \S \ref{ss:strong-conj}. Let $\mu$ be a quasi-regular measure for $G$ such that nilpotent endomorphisms in $\uPerm(G, \mu)$ have trace~0, and let $\cT = \uRep(G, \mu)$ be the abelian envelope of $\uPerm(G, \mu)$. Consider the following condition:
\begin{itemize}
\item[(S)] If $X$ and $Y$ are finitary $\Omega$-smooth $G$-sets with $\lev{X}>\lev{Y}$ then there exists a transitive $\Omega$-smooth $G$-set $Z$ such that $X \times Z$ has strictly more orbits than $Y \times Z$.
\end{itemize}
This condition is useful since it is combinatorial in nature, but has representation theoretic consequences, such as the following:

\begin{proposition} \label{prop:schwartz-level}
Suppose (S) holds. Let $X$ and $Y$ be $\Omega$-smooth $G$-sets, and suppose there is an injection $\cC(X) \to \cC(Y)$ in $\cT$. Then $\lev{X} \le \lev{Y}$.
\end{proposition}

\begin{proof}
Suppose $\lev{X}>\lev{Y}$. By (S) there is a transitive $\Omega$-smooth $G$-set $Z=G/U$ such that $X \times Z$ has more orbits than $Y \times Z$. Now, the number of $G$-orbits on $X \times Z$ is equal to the number of $U$-orbits on $X$, which is equal to the dimension of the invariants space $\cC(X)^U$, and similarly for $Y$. Thus $\dim \cC(X)^U > \dim \cC(Y)^U$. However, the map $\cC(X)^U \to \cC(Y)^U$ is injective, a contradiction.
\end{proof}

Using this, we can compare the small and $\cT$-small conditions.

\begin{proposition} \label{prop:small-Tsmall}
Suppose (S) holds, and let $D \subset G$ be a closed conjugacy set. Then $D$ is small if and only if it is $\cT$-small.
\end{proposition}

\begin{proof}
If $D$ is small then it is a finite union of c-smooth conjugacy classes (Proposition~\ref{prop:small-conj}), and thus $\cT$-small (Example~\ref{ex:csmooth-Tsmall}). Conversely, suppose $D$ is $\cT$-small. By definition, there is a $G$-set $Y$ such that for all transitive $G$-sets $X$ we have an injection
\begin{displaymath}
\cC(\fA_D(X)) \to \cC(X \times Y).
\end{displaymath}
Replacing $Y$ by a cover, we assume $Y$ is $\Omega$-smooth. Propositions~\ref{prop:level-prod} and~\ref{prop:schwartz-level} thus imply
\begin{displaymath}
\lev{\fA_D(X)} \le \lev{X} + \lev{Y}
\end{displaymath}
for all $\Omega$-smooth transitive $G$-sets $X$, and so $D$ is small.
\end{proof}

We next explain how to interpret (S) in the Fra\"iss\'e case. Let $\sF$ be a Fra\"iss\'e class. For $A,B \in \sF$, let $N_A(B)$ be the number of amalgamations of $A$ and $B$; this is symmetrical in $A$ and $B$, but we will mainly be interested in the case where $A$ is fixed and $B$ is allowed to grow. Consider the following condition on $\sF$:
\begin{itemize}
\item[(S')] Let $A$ and $B_1, \ldots, B_r$ be structures in $\sF$ with $\# A>\# B_i$ for each $i$. Then there exists some $C \in \sF$ such that $N_A(C)>\sum_{i=1}^r N_{B_i}(C)$.
\end{itemize}

\begin{proposition} \label{prop:Sprime}
Suppose $\Omega$ is the Fra\"iss\'e limit of $\sF$ and $G=\Aut(\Omega)$. If $\sF$ satisfies (S') then $(G, \Omega)$ satisfies (S).
\end{proposition}

\begin{proof}
Let $X$ and $Y$ be finitary $\Omega$-smooth $G$-sets with $\lev{X}>\lev{Y}$. Let $X'$ be a $G$-orbit on $X$ of maximal level, and let $Y_1, \ldots, Y_r$ be the $G$-orbits on $Y$. We have $X'=\Omega^{[A]}$ and $Y_i=\Omega^{[B_i]}$ for some $A$ and $B_i$ in $\sF$, with $\# A>\# B_i$ for each $i$ (see \S \ref{ss:fraisse}). Let $C$ be such that $N_A(C)>\sum_{i=1}^r N_{B_i}(C)$, which exists by (S'). Note that for any $D$ in $\sF$, the number of $G$-orbits on $\Omega^{[C]} \times \Omega^{[D]}$ is $N_D(C)$ (see \S \ref{ss:fraisse}). Thus, putting $Z=\Omega^{[C]}$, we see that $X' \times Z$ has strictly more orbits than $Y \times Z$. Thus (S) holds.
\end{proof}

\subsection{The main theorem} \label{ss:main}

We now come to our main result.

\begin{theorem} \label{thm:sa}
Let $(G, \Omega)$ be a closed oligomorphic permutation group, with $\Omega$ countable. Suppose that the equivalent conditions of Proposition~\ref{prop:dcl} hold and that condition (S) from \S \ref{ss:small-Tsmall} holds. Let $\cT = \Rep(G, \mu)$ be as in \S \ref{ss:small-Tsmall}, and let $\cY$ be the category of Yetter--Drinfeld modules (Definition~\ref{defn:yd}). Then the functor $\cY \to \cZ(\cT)$ is an equivalence.
\end{theorem}

\begin{proof}
A $\cT$-small conjugacy set is small (Proposition~\ref{prop:small-Tsmall}), and therefore has an isolated point (Proposition~\ref{prop:finitary-isolated}). Hence the result follows from Theorem~\ref{thm:YD}.
\end{proof}

\begin{corollary} \label{cor:sa}
Maintain the set-up from Theorem~\ref{thm:sa}, and additionally suppose that $G$ has no finitary elements except the identity. Then $\cT \to \cZ(\cT)$ is an equivalence.
\end{corollary}

\begin{proof}
The only c-smooth element is the identity (Proposition~\ref{prop:smooth-is-finitary}), and so $\cY = \cT$.
\end{proof}

\section{Drinfeld centers of some specific categories}

In this section, we use the theory developed thus far to compute the Drinfeld centers of several oligomorphic tensor categories.

\subsection{Deligne's category} \label{ss:deligne}

Let $\fS$ be the infinite symmetric group acting on $\Omega=\{1,2,\ldots\}$. This clearly satisfies the conditions of Proposition~\ref{prop:dcl}. We can identify $\Omega$ with the Fra\"iss\'e limit of the class $\sF$ of finite unstructured sets. An amalgamation of two finite sets $A$ and $B$ is a quotient of $A \amalg B$ to which $A$ and $B$ both map injectively. The number $N_A(B)$ of amalgamations of $A$ and $B$ is a polynomial in $\# B$ of degree $\# A$. It follows that condition (S') of \S \ref{ss:small-Tsmall} holds, and so (S) does too by Proposition~\ref{prop:Sprime}.

Suppose the coefficient field $k$ has characteristic~0, and fix $t \in k$. There is a unique measure $\mu_t$ for $\fS$ for which $\mu_t(\Omega)=t$. This measure is quasi-regular and the category $\uPerm(\fS, \mu_t)$ has an abelian envelope $\cT=\uRep(\fS, \mu_t)$. These claims are proved in \cite[\S 14.6]{repst}. The category $\cT$ is equivalent to the abelian version of Deligne's category \cite{Deligne3}. By Theorem~\ref{thm:sa}, we see that the natural functor $\cY \to \cZ(\cT)$ is an equivalence.

We now describe $\cZ(\cT)$ explicitly. Let $\lambda$ be a singleton-free partition of $n$, and let $g_{\lambda}$ be an element of $\fS_n \subset \fS$ with cycle type $\lambda$. The $g_{\lambda}$'s, as $\lambda$ varies, are representatives for the c-smooth conjugacy classes of $\fS$ (see Example~\ref{ex:sym-conj}). The centralizer $Z(g_{\lambda})$ of $g_{\lambda}$ in $\fS$ is $H_{\lambda} \times \fS(n)$, where $H_{\lambda}$ is the centralizer of $g_{\lambda}$ in $\fS_n$. We have
\begin{displaymath}
\uRep(Z(g_{\lambda}), \mu_t) \cong \Rep(H_{\lambda}) \boxtimes \uRep(\fS(n), \mu_t).
\end{displaymath}
Moreover, $\uRep(\fS(n), \mu_t)$ is equivalent to $\uRep(\fS_{t-n})$ (the abelian version). We thus have an equivalence of abelian categories
\begin{displaymath}
\cZ(\cT) \cong \bigoplus_{\lambda} \Rep(H_{\lambda}) \boxtimes \uRep(\fS_{t-\vert \lambda \vert}),
\end{displaymath}
where the sum is over all singleton-free partitions $\lambda$ (see Remark~\ref{rmk:yd}(c)). This agrees with the description of $\cZ(\cT)$ given in \cite{FHL}.

\subsection{The Delannoy category} \label{ss:delannoy}

Let $G$ be the group\footnote{For the purposes of measures and tensor categories, it makes no difference if one uses $\bQ$ or $\bR$ when defining $G$. In \cite{repst} and \cite{line}, $\bR$ is used. We use $\bQ$ here since we want the domain of $G$ to be countable.} of all order preserving self-bijections of the rational numbers $\bQ$. This clearly satisfies the conditions of Proposition~\ref{prop:dcl}. We can identify $\bQ$ with the Fra\"iss\'e limit of the class $\sF$ of finite totally ordered sets. It is not difficult to see that the number of amalgamations $N_A(B)$ is asymptotically a polynomial of degree $\# A$ in $\# B$. It follows that (S') holds, and so (S) does too by Proposition~\ref{prop:Sprime}.

For any field $k$, the group $G$ admits exactly four measures. One of these $\mu$ is regular; the others are not quasi-regular (see \S \ref{ss:second} for more). The category $\uPerm(G, \mu)$ admits an abelian envelope $\cT$. This category is called the \defn{Delannoy category}, and was studied in depth in \cite{line}. It has many remarkable properties; for instance, it is semi-simple over any field. It is clear that $G$ has no non-trivial finitary elements, and so Corollary~\ref{cor:sa} shows that the functor $\cT \to \cZ(\cT)$ is an equivalence.

There is a closely related category to mention. Let $\Omega$ be a countable everywhere dense cyclic order, such as the set of roots of unity in the complex unit circle. Let $H$ be the group of all self-bijections of $\Omega$ preserving the cyclic order. This can be identified with the Fra\"iss\'e limit of the class of finite cyclic orders. Once again, it is not difficult to see that the conditions of Proposition~\ref{prop:dcl} hold, as does (S'), and thus (S). The group $H$ admits a unique measure $\nu$ over any field, which is quasi-regular, and the category $\uPerm(H, \nu)$ admits an abelian envelope $\cT'$. This is the \defn{circular Delannoy category} studied in \cite{circle}. There are again no non-trivial finitary elements, so Corollary~\ref{cor:sa} gives an equivalence $\cT' \to \cZ(\cT')$.

\subsection{Arboreal categories} \label{ss:arboreal}

We recall the Fra\"iss\'e class of trees; we refer to \cite{CameronTrees} and \cite[\S 3.2]{arboreal} for additional background. Given a finite tree $T$, let $L(T)$ be its set of leaves. Define a quaternary relation $R$ on $L(T)$ by declaring $R(w,x,y,z)$ to be true if the shortest path from $w$ to $x$ in $T$ shares an edge with the shortest path from $y$ to $z$. A \defn{tree structure} is a finite set equipped with a quaternary relation that is isomorphic to $L(T)$ for some tree $T$. If $T$ is a tree and $T'$ is its \defn{reduction}, meaning we remove all vertices of valence two, then $L(T)=L(T')$ as structures. Moreover, if $T$ and $T'$ are reduced trees and $L(T) \cong L(T')$ then $T \cong T'$. Thus we can (and will) pass between tree structures and reduced trees.

Let $\sF$ be the class of all tree structures. This is a Fra\"iss\'e class with strong amalgamation. Let $\Omega$ be the Fra\"iss\'e limit of $\sF$, and let $G$ be its automorphism group. Since $\sF$ satisfies strong amalgamation, the equivalent conditions of Proposition~\ref{prop:dcl} hold.

\tikzset{leaf/.style={circle,fill=black,draw,minimum size=1mm,inner sep=0pt}}
\tikzset{boron/.style={circle,fill=white,draw,minimum size=1mm,inner sep=0pt}}

\begin{figure}
\begin{tikzpicture}[scale=1, baseline=(current bounding box.center)]
\node[boron] (X) at (0,0) {};
\node[leaf,label=right:{\tiny 1}] (A) at (0.75,0) {};
\node[leaf,label=left:{\tiny 2}] (B) at (0,0.75) {};
\node[leaf,label=left:{\tiny 3}] (C) at (-0.75,0) {};
\node[leaf,label=left:{\tiny 4}] (D) at (0,-0.75) {};
\draw (X)--(A);
\draw (X)--(B);
\draw (X)--(C);
\draw (X)--(D);
\end{tikzpicture}
\hskip 4em
\begin{tikzpicture}[scale=1, baseline=(current bounding box.center)]
\node[boron] (X) at (0,0) {};
\node[boron] (A) at (0.75,0) {};
\node[leaf,label=right:{\tiny 1}] (A1) at (1.25,0.5) {};
\node[leaf,label=right:{\tiny 5}] (A2) at (1.25,-0.5) {};
\node[boron] (B) at (0,0.75) {};
\node[leaf,label=left:{\tiny 2}] (B1) at (-0.5,1.25) {};
\node[leaf,label=right:{\tiny 6}] (B2) at (0.5,1.25) {};
\node[boron] (C) at (-0.75,0) {};
\node[leaf,label=left:{\tiny 7}] (C1) at (-1.25,0.5) {};
\node[leaf,label=left:{\tiny 3}] (C2) at (-1.25,-0.5) {};
\node[boron] (D) at (0,-0.75) {};
\node[leaf,label=right:{\tiny 4}] (D1) at (0.5,-1.25) {};
\node[leaf,label=left:{\tiny 8}] (D2) at (-0.5,-1.25) {};
\draw (X)--(A);
\draw (A)--(A1);
\draw (A)--(A2);
\draw (X)--(B);
\draw (B)--(B1);
\draw (B)--(B2);
\draw (X)--(C);
\draw (C)--(C1);
\draw (C)--(C2);
\draw (X)--(D);
\draw (D)--(D1);
\draw (D)--(D2);
\end{tikzpicture}
\caption{A tree $A$ (left) and its double $B$ (right). There is an embedding $A \to B$ by mapping each leaf to the leaf with the same label.}
\label{f:double}
\end{figure}
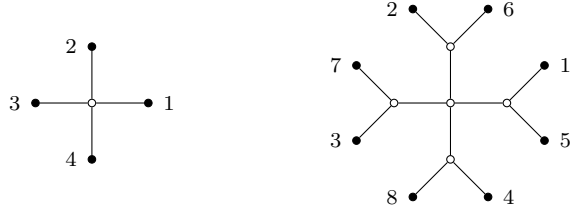

\begin{lemma} \label{lem:arboreal1}
The group $G$ has no non-trivial finitary elements.
\end{lemma}

\begin{proof}
We apply Proposition~\ref{prop:no-finitary}. We must therefore verify the condition $(\ast)$ from there. Thus let $A$ be a tree. Let $B$ be the tree obtained by doubling each leaf of $A$ (see Figure~\ref{f:double}). Choose an embedding $A \to B$ by mapping each leaf of $A$ to one of the new leaves sprouting from it. Then any automorphism of $B$ that fixes each element of $B \setminus A$ clearly fixes all elements of $B$. This verifies $(\ast)$, and so the result follows.
\end{proof}

For $A,B \in \sF$, let $N_A(B)$ be the number of amalgamations of $A$ and $B$ in $\sF$.

\begin{lemma} \label{lem:arboreal2}
Let $A \in \sF$ be a tree with $r$ leaves. There exist positive real numbers $\alpha$ and $\beta$ such that $N_A(B) \le \alpha (\# B)^r$ for all $B$, and $N_A(B) \ge \beta (\# B)^r$ for infinitely many $B$. In particular, $\sF$ satisfies property (S') from \S \ref{ss:small-Tsmall}.
\end{lemma}

\begin{proof}
We first explain the upper bound, though we only sketch the argument. Consider a tree $B$ with a large number of leaves, and let $C$ be an amalgamation of $A$ and $B$. Thus $C$ is obtained by adding at most $r$ new leaves to $B$. Let $a_1, \ldots, a_r$ be the leaves of $A$, and let $b_i$ be the vertex of $B$ that is nearest to $a_i$ (in $C$); if there are multiple choices for $b_i$, just pick one arbitrarily. It is not difficult to see that if we fix $b_1, \ldots, b_r$, then the number of choices for $C$ is bounded above by a function depending only on $r$. Thus $N_A(B)$ is bounded above by a constant (depending on $r$) times the number of choices for $b_1, \ldots, b_r$. Now, if $B$ has $n$ leaves then (since it is reduced) it has at most $2n-2$ total vertices. Thus the number of choices for $b_1, \ldots, b_r$ is at most $(2n)^r$. We thus see that $N_A(B)$ is bounded above by a constant times $(\# B)^r$, as required.

We now establish the lower bound. Let $B_1$ be the tree obtained by doubling each leaf of $A$ (as in Figure~\ref{f:double}), let $B_2$ be obtained by doubling the leaves of $B_1$, and so on. Thus $B_n$ is a tree with $r2^n$ leaves. If for each leaf of $A$ we pick one of the $2^n$ leaves in $B_n$ sprouting from it, this defines an embedding $A \to B_n$. Thus the number of embeddings is at least $2^{rn}$. Since each embedding gives an amalgamation, we find
\begin{displaymath}
N_A(B_n) \ge 2^{rn} = r^{-r} (\# B_n)^r.
\end{displaymath}
Thus the lower bound holds with $\beta=r^{-r}$.
\end{proof}

Let $k$ be a field of characteristic~0, and let $t$ be an element of $k$ that is not a natural number. In \cite[\S 3.6]{arboreal}, a $k$-valued regular measure $\mu_t$ for $G$ is defined\footnote{There is a subtlety here: in \cite{arboreal}, this measure is only defined on the structures in $\sF$, which means that it defines a measure for $G$ relative to a stabilizer class in the sense of \cite[\S 2.6]{repst}. However, since this measure yields a pre-Tannakian category, it necessarily extends uniquely to a regular measure on all of $G$; this is closely related to \cite[Theorem~7.3]{discrete}.} satisfying $\mu_t(\Omega)=t/(t-1)$. It is shown that nilpotent endomorphisms in $\uPerm(G, \mu_t)$ have trace~0 \cite[\S 5.5]{repst}, and so $\uPerm(G, \mu_t)$ has an abelian envelope $\cT$, which is a semi-simple pre-Tannakian category. Corollary~\ref{cor:sa}, combined with Lemmas~\ref{lem:arboreal1} and~\ref{lem:arboreal2}, now shows that the functor $\cT \to \cZ(\cT)$ is an equivalence.

We now discuss a variant. A tree $T$, or the corresponding tree structure $L(T)$, has \defn{level} $\le \ell$ if every vertex of $T$ has valence $\le \ell$. Fix $\ell \ge 3$, and let $\sF_{\ell}$ be the subclass of $\sF$ on structures of level $\le \ell$. This is again a Fra\"iss\'e class with strong amalgamation. This class admits a unique regular measure \cite[\S 3.7]{arboreal}, which gives rise to a semi-simple pre-Tannakian category $\cT$ over any field $k$ in which $\ell!$ is non-zero \cite[Theorem~1.2]{arboreal}. Using arguments similar to the above, one can show that $\cT \to \cZ(\cT)$ is an equivalence.

\subsection{Homeomorphisms of the Cantor set} \label{ss:cantor}

Let $\fX$ be the Cantor set, let $\Omega$ be the set of clopen subsets of $\fX$, and let $G$ be the group of all homeomorphisms $\fX \to \fX$. It is easy to see that the natural action of $G$ on $\Omega$ is oligomorphic, and we regard $G$ as an oligomorphic permutation group in this manner. In \cite{cantor}, the second author constructed a semi-simple pre-Tannakian category attached to $G$, which we recall below. The purpose of \S \ref{ss:cantor} is to compute its Drinfeld center. We begin with some preparatory work of a combinatorial nature.

The set $\Omega$ is a boolean algebra under intersection and union. It is countable and atomless, and any countable atomless boolean algebra is isomorphic to $\Omega$. The action of $G$ on $\Omega$ is compatible with the boolean algebra structure, and, in fact, $G$ is the group of all automorphisms of $\Omega$ as a boolean algebra. In particular, $(G, \Omega)$ is a closed oligomorphic permutation group on a countable set. These statements are all well-known; see, e.g. \cite{Givant}. The group $(G, \Omega)$ does not satisfy the conditions of Proposition~\ref{prop:dcl}: indeed, the definable closure of $\{A,B\}$ contains $A \cap B$, and can therefore be strictly larger than $\{A,B\}$.

We say that $(A_1, \ldots, A_n) \in \Omega^n$ is a \defn{partition} of $\fX$ if each $A_i$ is non-empty, the $A_i$'s are pairwise disjoint, and $\bigcup_{i=1}^n A_i=\fX$. We let $X(n) \subset \Omega^n$ be the set of partitions; we note that $X(n)$ can be identified with the space of continuous surjections $\fX \to [n]$. The natural action of $G$ on $X(n)$ is transitive \cite[Proposition~3.2]{cantor}. The product $X(n) \times X(m)$ decomposes as $\coprod_{r=1}^{nm} X(r)^{m(r)}$, where $m(r)$ is the number of ample subsets of $[n] \times [m]$ of cardinality $r$, where \defn{ample} means that the two projection maps are surjective \cite[Proposition~3.3]{cantor}. Observe that $X(2)$ is identified with the set of proper non-empty clopen subsets of $\fX$, and so $\Omega \cong \bbone \amalg \bbone \amalg X(2)$, where the two extra points correspond to the empty set and $\fX$ itself. We thus see that every orbit on $\Omega^n$ has the form $X(m)$. This shows that every transitive $\Omega$-smooth $G$-set is isomorphic to some $X(m)$.

Let $X$ be an $\Omega$-smooth $G$-set. We define the \defn{(multiplicative) level} of $X$, denoted $\mlev{X}$, to be the maximal $n$ such that $X(n)$ occurs as an orbit on $X$, with the convention that $\mlev(\emptyset)=0$. From the above paragraph, we see that $X(n) \times X(m)$ has level $nm$. It thus follows that
\begin{displaymath}
\mlev(X \times Y) = \mlev(X) \cdot \mlev(Y)
\end{displaymath}
for all $G$-sets $X$ and $Y$. This is very different from how level behaves in the setting of strong amalgamation (cf.\ Proposition~\ref{prop:level-prod}), which is why we use different notation. We now establish two simple properties of level.

\begin{lemma} \label{lem:cantor-1}
Let $A \in \Omega^n$, let $R \subset \Omega$ be the boolean subalgebra generated by $A_1, \ldots, A_n$, and let $r$ be such that $\# R=2^r$ (i.e., $r$ is the $\bF_2$-dimension of $R$). Then the level of the orbit of $A$ is equal to $r$.
\end{lemma}

\begin{proof}
Let $B_1, \ldots, B_r$ be the atoms of $R$. These define a partition of $\fX$. Each $B_i$ can be expressed using the $A$'s and boolean operations, and conversely. Thus the stabilizer of $A \in \Omega^n$ is the same as the stabilizer of $B \in \Omega^r$, and so the orbit of $A$ is isomorphic to the orbit of $B$. The latter is $X(r)$, and so the result follows.
\end{proof}

\begin{lemma} \label{lem:cantor-2}
Let $A,B \in \Omega^n$ be partitions of $\fX$. Let $r$ be the number of pairs $(i,j)$ such that $A_i \cap B_j$ is non-empty. Then the orbit of $(A,B) \in \Omega^{2n}$ has level $r$.
\end{lemma}

\begin{proof}
Let $C_1, \ldots, C_r$ be the non-empty sets of the form $A_i \cap B_j$, and let $R$ be the boolean algebra generated by the $A_i$ and $B_j$. It is clear that every element of $R$ can be uniquely expressed as $\coprod_{i \in S} C_i$ for some $S \subset [r]$. Thus $\# R=2^r$, and so the result follows from Lemma~\ref{lem:cantor-1}.
\end{proof}

We say that a c-functor $\fA$ is \defn{small} if there is some $s$ such that $\mlev{\fA(X(n))} \le sn$ for all $n \ge 1$. We say that a conjugacy set $D$ is \defn{small} if $\fA_D$ is. The following is the key lemma we will need to compute the Drinfeld center in this case.

\begin{lemma} \label{lem:cantor-3}
Any small conjugacy set is contained in $\{1\}$.
\end{lemma}

\begin{proof}
Suppose $D$ is a conjugacy set containing an element $g \ne 1$. We show that $D$ is not small. Let $x \in \fX$ be a point such that $x \ne gx$. Let $U$ be a clopen neighborhood of $x$ such that $V=gU$ is disjoint from $U$. Let $A_1, \ldots, A_n$ be a partition of $U$ and let $B_1, \ldots, B_n$ be a partition of $V$ such that $B_i$ meets $gA_j$ for all $i$ and $j$. To construct the $B_i$'s, choose a partition $gA_j=B_{1,j} \sqcup \cdots \sqcup B_{n,j}$, and then take $B_i=B_{i,1} \sqcup \cdots \sqcup B_{i,n}$. Put $C_1 = A_1 \cup B_1 \cup \fX \setminus (U \cup V)$ and $C_i=A_i \cup B_i$ for $2 \le i \le n$. Thus $C=(C_1, \ldots, C_n)$ is a partition of $\fX$.

Since $C$ is a partition, its orbit is $X(n)$. It follows that $\fA_D(X(n))$ contains the point $(C, gC)$. We have
\begin{displaymath}
gC_i \cap C_j \cap V = gA_i \cap B_j \ne \emptyset.
\end{displaymath}
Thus the orbit of $(C, gC)$ has level $n^2$ by Lemma~\ref{lem:cantor-2}, and so the level of $\fA_D(X(n))$ is $n^2$; note that $\fA_D(X(n))$ is contained in $X(n) \times X(n)$, and thus has level at most $n^2$. We thus see that there is no $s$ such that $\mlev \fA_D(X(n)) \le sn$ for all $n$, and so $D$ is not small.
\end{proof}

The next two lemmas establish a version of property (S) in this context. Recall that a subset of $[r] \times [n]$ is \defn{ample} if it surjects onto each factor. Let $N_r(n)$ be the number of ample subsets of $[r] \times [n]$. This is the number of $G$-orbits on $X(r) \times X(n)$ by \cite[Proposition~3.3]{cantor}.

\begin{lemma} \label{lem:cantor-4}
We have $(2^r-1)^n-r2^{(r-1)n} \le N_r(n) \le 2^{rn}$.
\end{lemma}

\begin{proof}
Obviously, $N_r(n)$ is bounded above by the number of subsets of $[r] \times [n]$, which is $2^{rn}$. We now establish the lower bound. Let $S_1, \ldots, S_n$ be non-empty subsets of $[r]$. Then $S=\coprod_{i=1}^n (S_i \times \{i\})$ is a subset of $[r] \times [n]$ that surjects onto $[n]$. We obtain $(2^r-1)^n$ sets $S$ in this manner. If $S$ does not surject onto $[r]$ then it is contained in $([r] \setminus \{a\}) \times [n]$ for some $a$. For each $a$ there are at most $2^{(r-1)n}$ such $S$. Thus the number of $S$'s which are not ample is at most $r 2^{(r-1)n}$. This completes the proof.
\end{proof}

\begin{lemma} \label{lem:cantor-5}
Let $X$ and $Y$ be $\Omega$-smooth $G$-sets with $\mlev{X}>\mlev{Y}$. Then there exists an $\Omega$-smooth $G$-set $Z$ such that $X \times Z$ has more orbits than $Y \times Z$.
\end{lemma}

\begin{proof}
It suffices to treat the case where $X$ is transitive, say $X=X(r)$. If $r=1$ then $Y=\emptyset$ and the result is clear; thus suppose $r \ge 2$. Write $Y=\coprod_{i=1}^m X(s_i)$, where $1 \le s_i<r$ for each $i$. The number of orbits on $X \times X(n)$ is $N_r(n)$, which is at least $(2^r-1)^n-r2^{(r-1)n}$ by Lemma~\ref{lem:cantor-4}. The number of orbits on $Y \times X(n)$ is $\sum_{i=1}^m N_{s_i}(n)$, which is at most $m 2^{(r-1)n}$ by Lemma~\ref{lem:cantor-4}. We thus see that if $n$ is sufficiently large then $X \times X(n)$ has more orbits than $Y \times X(n)$, as required.
\end{proof}

We are finally ready to discuss tensor categories. Let $k$ be a field of characteristic~0. The group $G$ admits exactly two $k$-valued measures $\mu$ and $\nu$, both of which are regular \cite[\S 4]{cantor}. The category $\uPerm(G, \nu)$ has nilpotent endomorphisms with non-zero trace, and so does not admit an abelian envelope. The category $\uPerm(G, \mu)$ admits a semi-simple pre-Tannakian envelope $\cT = \uRep(G, \mu)$ \cite[Corollary~5.3]{cantor}. We now determine its Drinfeld center.

\begin{theorem}
The functor $\cT \to \cZ(\cT)$ is an equivalence.
\end{theorem}

\begin{proof}
We apply Theorem~\ref{thm:Ztriv}. We must therefore show that a $\cT$-small c-functor $\fA$ is contained in the diagonal. Let $Y$ be a $G$-set such that for every transitive $G$-set $X$ we have an injection
\begin{displaymath}
\cC(\fA(X)) \to \cC(X \times Y).
\end{displaymath}
Arguing as in the proof of Proposition~\ref{prop:schwartz-level}, and making use of Lemma~\ref{lem:cantor-5}, we find
\begin{displaymath}
\mlev{\fA(X(n))} \le \mlev(X(n) \times Y) = sn,
\end{displaymath}
where $s=\mlev{Y}$. Thus $\fA$ is small. Lemma~\ref{lem:cantor-3} now shows that $\fA$ is contained in the diagonal; here we have used the correspondence between c-functors and conjugacy sets (Theorem~\ref{thm:cset}).
\end{proof}

\begin{remark}
Let $\cP=\uPerm(G, \nu)^{\rm kar}$. While $\cP$ does not admit an abelian envelope, one can use similar methods as above to show that $\cP \to \cZ(\cP)$ is an equivalence. We note that the elemental reasoning employed in \S \ref{s:qr} remains valid in $\cP$.
\end{remark}

\subsection{Jacobi categories} \label{ss:jacobi}

We now discuss a generalization of the story from \S \ref{ss:cantor}. Let $\bF$ be a finite field with $q$ elements, and let $k$ be algebraically closed of characteristic~0. Let $\cS$ be the $k$-linearization of the category of finite dimensional $\bF$-vector spaces. The objects of $\cS$ are symbols $[V]$, where $V$ is a finite dimensional $\bF$-vector space, and a morphism $[V] \to [W]$ is a formal $k$-linear combination of $\bF$-linear maps $V \to W$. The category $\cS$ decomposes into a direct sum of categories $\cS_{\chi}$ indexed by homomorphisms $\chi \colon \bF^* \to k^*$, together with one additional piece equivalent to $\Vec$. In \cite{dblexp}, the second author showed that each $\cS_{\chi}$ is naturally a semi-simple pre-Tannakian category.

The categories $\cS_{\chi}$ are studied in \cite{jaccat}, where they are called \defn{Jacobi categories}, due to their close connection to Jacobi sums. It is shown that $\cS_{\chi}$ is equivalent to the Karoubi envelope of a category of the form $\uPerm(H, \mu)$, where $H$ is an oligomorphic group and $\mu$ is a regular measure. The group $H$ consists of the self-homeomorphisms of the Cantor set $\fX$ that preserve a certain $\bF_q^*$-valued measure defined on clopen sets. When $q=2$, the category $\cS_{\chi}$ is equivalent to the category $\cT$ studied in \S \ref{ss:cantor}. Using the same methods, one can show that the functor $\cS_{\chi} \to \cZ(\cS_{\chi})$ is an equivalence in general.

\section{Drinfeld centers of classical interpolation categories} \label{s:classical}

In this section, we compute the Drinfeld center for the interpolation categories of classical groups. We also make some remarks on the quantum Delannoy category, which is a related example.

\subsection{Set-up}

We follow the general set-up from \cite[\S 6.1]{interp}. Fix a finite field $\bF$ of cardinality $q$. Let $\fC$ be the following category. An object is a finite dimensional $\bZ/2$-graded $\bF$-vector space $V=V_0 \oplus V_1$ equipped with a bilinear form $V_0 \times V_1 \to \bF$. A morphism $f \colon V \to W$ is an injective linear map that respects the grading and the forms. The ind-category $\Ind{\fC}$ admits a similar description, but without the condition that the vector space be finite dimensional.

Let $\bV$ be the $\bF$-vector space with basis $\{e_i, f_i\}_{i \ge 1}$. Let $\bV_0$ be the span of the $e_i$'s and let $\bV_1$ be the span of the $f_i$'s. Define a pairing between $\bV_0$ and $\bV_1$ by $\langle e_i, f_j \rangle = \delta_{i,j}$. In this way, $\bV$ is an object of $\Ind{\fC}$. We let $G \subset \GL(\bV_0)$ be the group of invertible matrices that are both row and column finite. We have an action of $G$ on $\bV_1$ by letting $g$ act by the matrix ${}^tg^{-1}$. This defines an action of $G$ on $\bV$, as an object of $\Ind{\fC}$, and it is easily verified that the resulting homomorphism $G \to \Aut(\bV)$ is an isomorphism. The action of $G$ on $\bV$ is oligomorphic, and we give $G$ the topology resulting from this action, as in \S \ref{ss:oligo}. We note that the definable closure of a finite subset of $\bV$ is the homogeneous subspace it generates; thus the conditions of Proposition~\ref{prop:dcl} do not hold.

We let $G(n) \subset G$ be the subgroup consisting of elements that fix each $e_i$ and $f_i$ for $1 \le i \le n$. This is an open subgroup of $G$, and any open subgroup of $G$ contains some $G(n)$. We let $\bV(n)$ be the subspace of $\bV$ spanned by the $e_i$ and $f_i$ with $i>n$. Note that $G(n)$ acts on $\bV(n)$, and $(G(n), \bV(n))$ is isomorphic to $(G, \bV)$. We let $G_n \subset G$ be the subgroup fixing each $e_i$ and $f_i$ with $i>n$. This group is isomorphic to $\GL_n(\bF)$. The groups $G_n$ and $G(n)$ centralize each other.

Let $V_{\ell,m,n}$ be the $\bZ/2$-graded vector space over $\bF$ of dimension $(\ell+m \vert \ell+n)$ equipped with a form that is a perfect pairing on the $(\ell \vert \ell)$ piece and vanishes on the $(m \vert n)$ piece. Every object of $\fC$ is isomorphic to a unique $V_{\ell,m,n}$. We let $X_{\ell,m,n}$ be the space of embeddings $V_{\ell,m,n} \to \bV$. These are exactly the transitive $\bV$-smooth $G$-sets \cite[\S 6.1]{interp}.

\begin{remark}
The group $G$ is one version of the infinite general linear group. Its basic open subgroups $G(n)$ are (finite index subgroups of) Levi subgroups. There is a different (and simpler) version of the infinite general linear group where the basic open subgroups are (finite index subgroups of) parabolic subgroups. However, this group is insufficient for computing the Drinfeld center. See \S \ref{ss:parabolic}.
\end{remark}

\subsection{Smooth conjugacy classes}

We begin by determining the c-smooth elements of $G$. We write $\bF^*$ for the multiplicative group of $\bF$, which we identify with the group of scalar matrices in $G$.

\begin{proposition} \label{prop:gl-csmooth}
Let $g \in G$. The following conditions are equivalent:
\begin{enumerate}
\item $g$ is c-smooth.
\item $g-\alpha$ has finite rank for some (unique) $\alpha \in \bF^*$.
\item $g$ belongs to $\bF^* \cdot G_n$ for some $n$.
\end{enumerate}
Moreover, if $g-\alpha$ has rank $r$ then $g$ is conjugate to an element of $\bF^* \cdot G_{2r}$. In particular, up to conjugacy, there are finitely many such elements.
\end{proposition}

\begin{proof}
It is clear that (c) implies (a) and (b). Suppose now that (a) holds. Then there is some $n$ such that $G(n)$ centralizes $g$. The centralizer of $G(n)$ in $G$ is easily seen to be $\bF^* \cdot G_n$, and so (c) holds. Now suppose that (b) holds. Let $m$ be such that the image of $g-\alpha$ is contained in the span of $e_1, \ldots, e_m$. Since $g$ is row finite, there is some $n \ge m$ such that all non-zero entries of $g-\alpha$ are contained in the top left $n \times n$ block. Thus (c) holds.

Now suppose that $g-\alpha$ has rank $r$, for some $\alpha \in \bF^*$. By (c), we know that the non-zero entries of $g-\alpha$ are contained in the top left $n \times n$ block, for some $n$. By standard linear algebra, we can find an element $h \in G_n$ such that the non-zero entries of $h(g-\alpha)h^{-1}$ are contained in the top left $2r \times 2r$ block. Thus $hgh^{-1}$ belongs to $\bF^* \cdot G_{2r}$, as required.
\end{proof}

We thus see that the group of c-smooth elements of $G$ is the ``small'' infinite general linear group $\bigcup_{n \ge 1} G_n$.

\subsection{Level} \label{ss:gl-level}

Let $X$ be a $\bV$-smooth $G$-set. We define the \defn{level} of $X$, denote $\lev{X}$, to be the maximum value of $2\ell+m+n$ for which $X_{\ell,m,n}$ appears as an orbit on $X$. By convention, $\lev{\emptyset} = -\infty$. It is not difficult to see directly that
\begin{equation} \label{eq:gl-lev}
\lev(X \times Y) = \lev{X} + \lev{Y}.
\end{equation}
This also follows from the analysis of the Burnside ring in \cite[\S 6.3]{interp}.

\begin{proposition} \label{prop:gl-level}
The level of the orbit of $x \in \bV^n$ is the dimension of the homogeneous subspace of $\bV$ generated by $x_1, \ldots, x_n$.
\end{proposition}

\begin{proof}
Let $W$ be the homogeneous subspace of $\bV$ generated by $x_1, \ldots, x_n$. It is clear that an element of $G$ fixes $x$ if and only if it fixes each element of $W$. Thus the orbit of $x$ is isomorphic to the space of embeddings $W \to \bV$. Since $W$ is isomorphic to some $V_{\ell,m,n}$ as an object of $\fC$, the space of embeddings is isomorphic to $X_{\ell,m,n}$. The level of $X_{\ell,m,n}$ is $2\ell+m+n$, which is the dimension of $W$, as required.
\end{proof}

\subsection{Small conjugacy sets}

We say that a c-functor $\fA$ is \defn{small} if there is some $s$ such that $\lev{\fA(X)} \le \lev{X}+s$ for all $\bV$-smooth transitive $G$-sets $X$. We say that a conjugacy set $D \subset G$ is \defn{small} if $\fA_D$ is. We now investigate small conjugacy sets. We begin with two linear algebra lemmas.

\begin{lemma} \label{lem:glsmall-1}
Let $V$ be an infinite dimensional $\bF$-vector space, let $W_1, \ldots, W_{q-1}$ be proper subspaces of $V$, and let $W_q, \ldots, W_n$ be finite dimensional subspaces of $V$. Then $\bigcup_{i=1}^n W_i$ is a proper subset of $V$.
\end{lemma}

\begin{proof}
Let $e$ be the maximum of the dimensions of $W_q, \ldots, W_n$. For each $1 \le i \le q-1$, let $w_i$ be an element of $V$ that does not belong to $W_i$. Let $d$ be such that $q^{d-1}>(n-q)q^e$ and let $V_0$ be a $d$ dimensional subspace of $V$ containing each $w_i$. Then $W_i \cap V_0$ is a proper subspace of $V_0$ for $1 \le i \le q-1$. We thus find
\begin{displaymath}
\sum_{i=1}^n \# (W_i \cap V_0) \le (q-1)q^{d-1} + (n-q) q^e < q^d = \# V_0.
\end{displaymath}
Thus there is an element of $V_0$ that does not belong to $\bigcup_{i=1}^n W_i$, as required.
\end{proof}

The following is a linear analog of Lemma~\ref{lem:small-conj}.

\begin{lemma} \label{lem:glsmall-2}
Let $g \in G$ and let $s \in \bN$. Suppose that for every finite dimensional subspace $U$ of $\bV_0$ we have $\dim(U+gU) \le s+\dim{U}$. Then there is some $\alpha \in \bF^*$ such that $g-\alpha$ has rank $\le 2s$.
\end{lemma}

\begin{proof}
Suppose, by way of contradiction, that $g-\alpha$ has rank $>2s$ for all $\alpha \in \bF^*$. We construct subspaces $0=U_0 \subset \cdots \subset U_{s+1}$ of $\bV_0$, with the property that $\dim{U_i}=i$ and $U_i \cap gU_i=0$. This will yield a contradiction, as $\dim(U_{s+1}+gU_{s+1})=2(s+1)$.

Suppose $i \le s$ and $U_i$ has been defined; we define $U_{i+1}$. For $\alpha \in \bF^*$, let
\begin{displaymath}
W_{\alpha} = (g-\alpha)^{-1}(U_i+gU_i) \subset \bV_0.
\end{displaymath}
Since $U_i+gU_i$ has dimension $2i \le 2s$ and $\rank(g-\alpha)>2s$, it follows that $W_{\alpha}$ is a proper subspace of $\bV_0$. Also, let $W_0=U_i+gU_i$ and $W_{\infty}=g^{-1}U_i+U_i$; these are finite dimensional. By Lemma~\ref{lem:glsmall-1}, we see that $\bigcup_{\alpha \in \bP^1(\bF)} W_{\alpha}$ is a proper subset of $\bV_0$. Let $v$ be an element of $\bV_0$ that does not belong to any $W_{\alpha}$, and put $U_{i+1}=U_i+\bF v$.

We now verify that $U_{i+1}$ has the requisite properties. Since $v \not\in W_0$ we have $v \not\in U_i$, and so $U_{i+1}$ has dimension $i+1$. We claim that $U_{i+1} \cap gU_{i+1}=0$. Suppose $x$ belongs to the intersection. Then $x = u+\alpha v = gu'+\beta gv$ for $u,u' \in U_i$ and $\alpha,\beta \in \bF$. We thus see that $(\beta g-\alpha)v = u-gu' \in U_i+gU_i$. By our choice of $v$, this can only happen if $\alpha=\beta=0$. Thus $x \in U_i \cap g U_i$, and so $x=0$, as required. This completes the proof.
\end{proof}

We can now characterize small conjugacy sets.

\begin{proposition} \label{prop:gl-small}
A conjugacy set is small if and only if it is a finite union of c-smooth conjugacy classes.
\end{proposition}

\begin{proof}
It follows from Proposition~\ref{prop:gl-csmooth} that a finite union of c-smooth conjugacy classes is small. Now let $D$ be a small conjugacy set, and let $s$ be such that $\lev{\fA_D(X)} \le \lev{X}+s$ for all $\bV$-smooth transitive $G$-sets $X$. Let $g \in G$. Let $x_1, \ldots, x_n$ be linearly independent elements of $\bV_0$, and let $X \subset \bV_0^n$ be the orbit of $x=(x_1, \ldots, x_n)$. We have $\lev{X}=n$. The set $\fA_D(X)$ contains the orbit of the element $(x, gx)$. We thus see that the span of the set $\{x_i, gx_i\}_{1 \le i \le n}$ has dimension $\le n+s$. Lemma~\ref{lem:glsmall-2} shows that $\rank(g-\alpha) \le 2s$ for some $\alpha \in \bF^*$. Thus $g$ is c-smooth and belongs to one of finitely many conjugacy classes (depending on $s$) by Proposition~\ref{prop:gl-csmooth}.
\end{proof}

We next show that a small conjugacy set has isolated points. For a c-smooth element $g \in G$, there is a unique element $\alpha \in \bF^*$ such that $g-\alpha$ has finite rank (Proposition~\ref{prop:gl-csmooth}). Put $r(g)=\rank(g-\alpha)$.

\begin{proposition} \label{prop:gl-isolated}
Let $D$ be a small conjugacy set, and let $r$ be the maximum value of $r(g)$ over $g \in D$. Then any element $g \in D$ with $r(g)=r$ is isolated in $D$.
\end{proposition}

\begin{proof}
If $r=0$ then $D$ consists of scalar matrices and the result is clear. Thus suppose $r \ge 1$. Let $g$ be an element of $D$ with $r(g)=r$. Replacing $g$ with a conjugate, we assume $g \in \bF^* \cdot G_{2r}$ (Proposition~\ref{prop:gl-csmooth}). We can thus write
\begin{displaymath}
g = \begin{pmatrix} A & 0 \\ 0 & \alpha \end{pmatrix},
\end{displaymath}
for some $A \in \GL_{2r}(\bF)$ and some $\alpha \in \bF^*$ such that $A-\alpha$ has rank $r$. We now break the bottom right block of $A$ into two blocks, where the first has size $2r \times 2r$:
\begin{displaymath}
g = \begin{pmatrix} A & 0 & 0 \\ 0 & \alpha & 0 \\ 0 & 0 & \alpha \end{pmatrix},
\end{displaymath}
Now, suppose $h \in g G(4r) \cap D$. We then have
\begin{displaymath}
h = \begin{pmatrix} A & 0 & 0 \\ 0 & \alpha & 0 \\ 0 & 0 & \alpha B \end{pmatrix},
\end{displaymath}
for some invertible matrix $B$. Since $h$ belongs to $D$, we have $r(h) \le r$, and so there is some $\beta \in \bF^*$ such that $h-\beta$ has rank $\le r$. Looking at the middle block above, we see that $\beta=\alpha$. Since $A-\alpha$ has rank $r$, we must have $B=1$, and so $h=g$. Thus $gG(4r) \cap D = \{g\}$, which shows that $g$ is isolated in $D$, as required.
\end{proof}

\begin{remark}
Remark~\ref{rmk:isolated} can be used to give an alternative proof of the existence of an isolated point in a small conjugacy set (see Remark~\ref{rmk:isolated-alt}).
\end{remark}

\subsection{Property (S)}

The following two results establish a version of property (S) in the present situation. For a $G$-set $X$, let $N_X(s)$ be the number of orbits of $G(s)$ on $X$.

\begin{lemma} \label{lem:gl-S}
Let $X$ be a non-empty $\bV$-smooth $G$-set of level $r$. There exist positive real numbers $\alpha$ and $\beta$ such that $\alpha q^{rs} \le N_X(s) \le \beta q^{rs}$.
\end{lemma}

\begin{proof}
It suffices to treat the case where $X$ is transitive, so suppose $X=X_{\ell,m,n}$. Note that $r=2\ell+m+n$. Now, $X$ occurs as an orbit on $\bV_0^{\ell+m} \times \bV_1^{\ell+n}$, so the number of $G(s)$ orbits on $X$ is bounded above by the number of orbits on the larger space. As a $G(s)$-set, we have $\bV_0=\bF^s \times \bV_0(s)$, and similarly for $\bV_1$. Thus
\begin{displaymath}
G(s) \backslash (\bV_0^{\ell+m} \times \bV_1^{\ell+n}) \cong \bF^{rs} \times E, \qquad E = G(s) \backslash (\bV_0(s)^{\ell+m} \times \bV_1(s)^{\ell+n}).
\end{displaymath}
The set $E$ is independent of $s$, and so we can take $\beta=\# E$. Now, the number of $G(s)$ orbits on $X$ is bounded below by the number of $G(s)$ fixed points. This is a polynomial in $q^s$ of degree $r$ by \cite[Proposition~6.3]{interp}, and so the result follows.
\end{proof}

\begin{proposition} \label{prop:gl-S}
Let $X$ and $Y$ be $\bV$-smooth $G$-sets, and suppose $\lev{X}>\lev{Y}$. Then there exists a $\bV$-smooth $G$-set $Z$ such that $X \times Z$ has more orbits than $Y \times Z$.
\end{proposition}

\begin{proof}
By Lemma~\ref{lem:gl-S}, we can take $Z=X_{s,0,0}=G/G(s)$ for $s$ sufficiently large.
\end{proof}

\subsection{The main result} \label{ss:gl-main}

Let $k$ be a field of characteristic~0, and let $t$ be a non-zero element of $k$. By \cite[Corollary~6.17]{repst}, there is a unique $k$-valued measure $\mu_t$ for $G$ such that $\mu_t(\bV_0)=t$. This measure is quasi-regular by \cite[Proposition~6.18]{repst}. By \cite[Theorem~7.1]{repst}, nilpotent endomorphisms in $\uPerm(G, \mu_t)$ have trace~0. We let $\cT=\uRep(G, \mu_t)$ be the abelian envelope. We note that $\mu_t$ is regular (and thus $\cT$ is semi-simple) if and only if $t$ is not a power of $q$. In the semi-simple case, $\cT$ is the category studied by Knop \cite{Knop}. Let $\cY=\cY(G, \mu_t)$ be the category of Yetter--Drinfeld modules as in \S \ref{ss:yd}. The following is the main result of \S \ref{s:classical}:

\begin{theorem}
The functor $\cY \to \cZ(\cT)$ is an equivalence.
\end{theorem}

\begin{proof}
We apply Theorem~\ref{thm:YD}. We must show that if $D$ is a $\cT$-small conjugacy set then $D$ has an isolated point. Let $Y$ be a $G$-set such that $\cC(\fA_D(X))$ injects into $\cC(X \times Y)$ for all transitive $G$-sets $X$. Replacing $Y$ by a cover, we can assume $Y$ is $\bV$-smooth. Let $X$ be a transitive $\bV$-smooth $G$-set. Then
\begin{displaymath}
\lev{\fA_D(X)} \le \lev(X \times Y) = \lev{X} + \lev{Y},
\end{displaymath}
where in the first step we use Proposition~\ref{prop:gl-S} and the argument from Proposition~\ref{prop:schwartz-level}, and in the second step we use \eqref{eq:gl-lev}. We thus see that $D$ is a small conjugacy set, and it therefore has an isolated point by Proposition~\ref{prop:gl-isolated}.
\end{proof}

\subsection{The parabolic topology} \label{ss:parabolic}

Let $G^*$ be the group $\GL(\bV_0)$ of all linear automorphisms of $\bV_0$. This group acts oligomorphically on $\bV_0$, and we give it the induced topology as in \S \ref{ss:oligo}. We have a containment $G \subset G^*$, but the topology on $G$ is not the subspace topology. The open subgroups of $G$ are (finite index subgroups of) Levi subgroups, while the open subgroups of $G^*$ are (finite index subgroups of) parabolic subgroups. Related to this, the action of $G$ on $\bV$ does not extend to a smooth action of $G^*$. See \cite[\S 1.4(c)]{interp} or \cite[\S 15.6]{repst} for more details. We note that $(G, \bV)$ and $(G^*, \bV_0)$ are both closed oligomorphic permutation groups, and $G$ is dense in $G^*$.

Let $k$ and $t$ be as in \S \ref{ss:gl-main}. Then there is a quasi-regular measure $\mu^*_t$ for $G^*$ with $\mu^*_t(\bV_0)=t$ by \cite[\S 15.4]{repst}. Nilpotent endomorphisms in $\uPerm(G^*, \mu^*_t)$ have trace~0 by \cite[Theorem~15.15]{repst}, and so there is an abelian envelope $\uRep(G^*, \mu^*_t)$. It is known that $\uRep(G^*, \mu^*_t)$ is equivalent to the category $\cT=\uRep(G, \mu_t)$ \cite[\S 1.4(c)]{interp}. It is not difficult to see that $G^*$ has no c-smooth elements, other than scalars. Thus the category $\cY^*$ is quite small. In particular, $\cY^* \to \cZ(\cT)$ is not an equivalence; the image only captures the piece of $\cY=\cZ(\cT)$ corresponding to $\cC(\bF^*)$-modules.

We now explain how the failure of $\cY^* \to \cZ(\cT)$ to be an equivalence relates to our arguments. Let $D \subset G^*$ be the conjugacy class of a non-identity matrix $g \in G_n$. One can show that $D$ is $\cT$-small. We claim that $g$ is not an isolated point of $D$. To see this, write
\begin{displaymath}
g = \begin{pmatrix} A & 0 \\ 0 & 1 \end{pmatrix}
\end{displaymath}
with $A \in \GL_n(\bF)$, let $B$ be an $n \times \infty$ matrix with entries in $\bF$, and consider
\begin{displaymath}
\begin{pmatrix} 1 & B \\ 0 & 1 \end{pmatrix}
\begin{pmatrix} A & 0 \\ 0 & 1 \end{pmatrix}
\begin{pmatrix} 1 & B \\ 0 & 1 \end{pmatrix}^{-1}
= \begin{pmatrix} A & B-AB \\ 0 & 1 \end{pmatrix}.
\end{displaymath}
For any choice of $B$, this matrix belongs to $D$. Let $B_n$ be a matrix such that $B_n-AB_n$ is not zero, but its first $n$ columns are. Then the above matrix (with $B=B_n$) is never equal to $g$, but converges to $g$ in $G^*$ as $n \to \infty$. This proves the claim. It follows that $D$ has no isolated points. The upshot is that Theorem~\ref{thm:YD} cannot be applied with the group $G^*$.

\subsection{Other classical groups}

The group $G$ is (one version of) the infinite general linear group over $\bF$. The infinite symplectic, orthogonal, and unitary groups over $\bF$ are also oligomorphic. They have families of quasi-regular measures that give rise to pre-Tannakian categories. The theory is worked out in detail in \cite{interp}. The Drinfeld centers of these categories can be analyzed in exactly the same way as for $G$. In each case, we find that the center is described by Yetter--Drinfeld modules.

\subsection{The quantum Delannoy category} \label{ss:kriz}

We close \S \ref{s:classical} with a brief discussion of Kriz's quantum Delannoy category \cite{Kriz}. This is not an interpolation category, but it shares some features with the $\GL$ case discussed above. We let $k$ be any field in which $q \ne 0$.

Let $\bW_0$ be the $\bF$-vector space with basis $\{e_x\}_{x \in \bQ}$. Let $H=\Aut(\bQ, <)$ act on $\bW_0$ by permuting basis vectors. Let $U^*$ be the ``unipotent'' subgroup of $\GL(\bW_0)$. An element $g$ of $\GL(\bW_0)$ belongs to $U^*$ if for each $x \in \bQ$ we have $ge_x=e_x+\cdots$, where the remaining terms use basis vectors of the form $e_y$ with $y<x$. The group $H$ normalizes $U^*$, and we let $K^*=HU^*$. This group acts oligomorphically on $\bW_0$, and we endow it with the induced topology. Kriz showed that a $k$-valued regular measure $\mu^*$ exists, and that $\uPerm(K^*, \mu^*)$ has a semi-simple abelian envelope $\cT=\uRep(K^*, \mu^*)$.

The group $K^*$ is analogous to $G^*$, the ``parabolic'' version of $\GL$ discussed in \S \ref{ss:parabolic}. There is a ``Levi'' version of $K$, analogous to $G$ itself, as follows. Let $\bW_1$ be the $\bF$-vector space with basis $\{f_x\}_{x \in \bQ}$, and let $\bW=\bW_0 \oplus \bW_1$. Define a pairing on $\bW$ like the pairing on $\bV$. Let $U$ be the subgroup of $U^*$ consisting of matrices that are both row and column finite, and let $K=HU$. This acts oligomorphically on $\bW$, and we endow it with the induced topology. Note that $K$ has many c-smooth elements: indeed, any element fixing all but finitely many of the basis vectors is c-smooth. However, $K^*$ has no c-smooth elements except the identity.

The group $K$ has a regular measure $\mu$, and $\uPerm(K, \mu)$ has an abelian envelope equivalent to the category $\cT$ above. Using methods similar to those in this section, one can show that $\cZ(\cT)$ is equivalent to the category $\cY(K, \mu)$ of Yetter--Drinfeld modules for $K$. The second author will provide additional details in the forthcoming paper \cite{vecdel}.

\section{Drinfeld centers: the non-quasi-regular case} \label{s:nonqr}

In \S \ref{s:qr}, we developed tools for studying the Drinfeld center of an oligomorphic tensor category associated to a quasi-regular measure. In this section, we extend some of that work to the general case. The basic ideas are the same, but the proofs are more complicated since we can no longer use elements. We use this theory to compute the Drinfeld center of the second Delannoy category.

\subsection{Central algebras} \label{ss:central}

Fix a pro-oligomorphic group $G$ with measure $\mu$. Let $\cP$ be the Karoubi envelope of $\uPerm(G, \mu)$, let $\cT$ be a Karoubian $k$-linear rigid symmetric monoidal category, and let $\Phi \colon \cP \to \cT$ be a symmetric monoidal functor. We write $\cC(X)$ for both the object of $\cP$ and its image in $\cT$.

Ultimately, we are interested in the case where $\cT$ is abelian, and we would like to understand $\cZ(\cT)$. For the time being, we focus on the simpler problem of understanding $\cZ_{\cT}(\cP)$, the centralizer of $\cP$ in $\cT$. An object of $\cZ_{\cT}(\cP)$ is an object $M$ of $\cT$ equipped with a functorial isomorphism $P \otimes M \to M \otimes P$ for objects $P$ of $\cP$, satisfying the usual condition (see \S \ref{ss:center}). We will study $\cZ_{\cT}(\cP)$ using the following notion:

\begin{definition} \label{defn:central}
Let $R$ be an associative unital algebra in $\cT$. A \defn{central structure} on $R$ is a rule $\beta$ assigning to each $G$-set $X$ a map
\begin{displaymath}
\beta_X \colon \cC(X \times X) \to R
\end{displaymath}
in $\cT$, such that the following axioms hold:
\begin{enumerate}
\item If $X=\bbone$ is a one-point set then $\beta_X$ is the unit of $R$.
\item Let $X$ and $Y$ be $G$-sets. Then the following diagram commutes
\begin{displaymath}
\xymatrix{
\cC(X \times X) \otimes \cC(Y \times Y) \ar[d]_{\beta_X \otimes \beta_Y} \ar[d] \ar@{=}[r] &
\cC((X \times Y) \times (X \times Y)) \ar[d]^{\beta_{X \times Y}} \\
R \otimes R \ar[r] & R }
\end{displaymath}
where the bottom map is multiplication in $R$, and the top map is induced from the isomorphism of $G$-sets taking $(x_1,x_2,y_1,y_2)$ to $(x_1,y_1,x_2,y_2)$.
\item Let $f \colon Y \to X$ be a map of $G$-sets. Then the following diagram commutes
\begin{displaymath}
\xymatrix@C=4em{
\cC(Y \times X) \ar[r]^{(f \times \id)_*} \ar[d]_{(\id \times f)^*} & \cC(X \times X) \ar[d]^{\beta_X} \\
\cC(Y \times Y) \ar[r]^{\beta_Y} & R }
\end{displaymath}
Also, the analogous diagram with $\cC(X \times Y)$ in the top left also commutes.
\end{enumerate}
A \defn{central algebra} is an algebra equipped with a central structure, and a \defn{central module} is an object $M$ of $\cT$ equipped with a central structure on $\uEnd(M)$.
\end{definition}

If $\mu$ is quasi-regular and $\cT$ is the abelian envelope of $\cP$ then a central structure is the same as an $\cO(\pi)$-structure (see \S \ref{ss:Opi}). The following proposition explains the importance of central structures.

\begin{proposition}
The center $\cZ_{\cT}(\cP)$ is equivalent to the category of central modules in $\cT$.
\end{proposition}

\begin{proof}
Given an object $M$ of $\cZ_{\cT}(\cP)$, we have a functorial isomorphism
\begin{displaymath}
\cC(X) \otimes M \to M \otimes \cC(X),
\end{displaymath}
which, by adjunction, can be converted to a map
\begin{displaymath}
\cC(X \times X) \to \uEnd(M).
\end{displaymath}
This defines a central structure on $\uEnd(M)$. Conversely, given a central structure on $\uEnd(M)$, this process can be reversed to obtain a half-braiding on $M$. We leave the details to the reader.
\end{proof}

We observe a few simple properties of central structures.

\begin{proposition} \label{prop:pi-basic2}
Let $R$ be a central algebra.
\begin{enumerate}
\item Let $X=Y \amalg Z$ be a $G$-set. Then the restriction of $\beta_X$ to $\cC(Y \times Z)$ is zero, and the restriction of $\beta_X$ to $\cC(Y \times Y)$ is $\beta_Y$.
\item Suppose $X \to Y$ is an isomorphism of $G$-sets. Then the restriction of $\beta_X$ along the isomorphism $\cC(Y \times Y) \to \cC(X \times X)$ is $\beta_Y$.
\end{enumerate}
\end{proposition}

\begin{proof}
This is analogous to Proposition~\ref{prop:pi-basic}. We leave the details to the reader.
\end{proof}

\begin{remark}
Let $C$ be a c-smooth conjugacy class in $G$. Then $\cC(C)$ carries a natural central structure, as follows. For a $G$-set $X$, let $E(X) \subset C \times X \times X$ be the set of tuples $(g,x,y)$ with $y=gx$, and let $B_X$ be the indicator function of $E(X)$. Then $\beta_X \colon \cC(X \times X) \to \cC(C)$ is the matrix $B_X$. The linear map $B_X \colon \cC(X \times X) \to \cC(C)$ is given by
\begin{displaymath}
(B_X \phi)(g) = \int_X \phi(x, gx) \, dx.
\end{displaymath}
This permits one to define a category of Yetter--Drinfeld modules, just as in quasi-regular case (see \S \ref{ss:yd}).
\end{remark}

\subsection{Support} \label{ss:support2}

Let $\Phi \colon \cP \to \cT$ be as in \S \ref{ss:central}. For \S \ref{ss:support2}, we impose the following condition:
\begin{itemize}
\item[$(\ast)$] $\cT$ is rigid abelian and $\Phi \colon \cP \to \cT$ is fully faithful.
\end{itemize}
If $X$ is a transitive $G$-set then we have
\begin{displaymath}
\Hom_{\cT}(\bone, \cC(X)) = \Hom_{\cP}(\bone, \cC(X)) = k.
\end{displaymath}
It follows that $\cC(X)$ is a simple \'etale algebra in $\cT$ \cite[Corollary~5.2]{discrete}. In particular, if $f \colon Y \to X$ is a map of transitive $G$-sets then $f^* \colon \cC(X) \to \cC(Y)$ is an injective map in $\cT$, since it is an algebra homomorphism and $\cC(X)$ is simple, and so the dual map $f_* \colon \cC(Y) \to \cC(X)$ is surjective.

Let $R$ be a central algebra. For a transitive $G$-set $X$, define $\fA_R(X)$ to be the union of those orbits $W$ on $X \times X$ for which $\beta_X$ is non-zero on $\cC(W) \subset \cC(X \times X)$. As in the quasi-regular case, we call this the \defn{support} of $R$. We now show that it satisfies some of the same properties.

The following proofs mimic those from the quasi-regular case (see \S \ref{ss:support}), but without using elements. We will use string diagrams as a notational device. In these diagrams, the source is at the top and the target at the bottom. Nodes in the middle of the diagram represent (co)multiplication in $\cC(X)$ or $R$, and boxes represent $\beta$. For example, Definition~\ref{defn:central}(b) can be expressed diagrammatically as
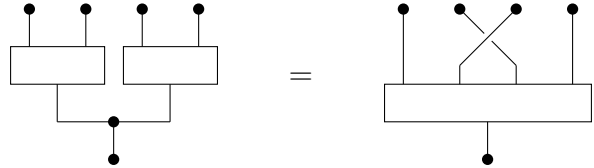
\begin{displaymath}
\begin{tikzpicture}[scale=1, baseline=(current bounding box.center)]
  \node[circle, fill, inner sep=1.5pt] () at (0,0) {};
  \draw (0,0)--(0,-0.5);
  \node[circle, fill, inner sep=1.5pt] () at (0.75,0) {};
  \draw (0.75,0)--(0.75,-0.5);
  \node[circle, fill, inner sep=1.5pt] () at (1.5,0) {};
  \draw (1.5,0)--(1.5,-0.5);
  \node[circle, fill, inner sep=1.5pt] () at (2.25,0) {};
  \draw (2.25,0)--(2.25,-0.5);

  \draw (-0.25,-0.5) rectangle ++(1.25,-0.5);
  \draw (1.25,-0.5) rectangle ++(1.25,-0.5);
  \draw (0.37,-1) -- (0.37,-1.5);
  \draw (0.37,-1.5) -- (1.12,-1.5);
  \node[circle, fill, inner sep=1.5pt] () at (1.12,-1.5) {};
  \draw (1.87,-1) -- (1.87,-1.5);
  \draw (1.87,-1.5) -- (1.12,-1.5);
  \draw (1.12,-1.5) -- (1.12,-2);
  \node[circle, fill, inner sep=1.5pt] () at (1.12,-2) {};
\end{tikzpicture}
\qquad = \qquad
\begin{tikzpicture}[scale=1, baseline=(current bounding box.center)]
  \node[circle, fill, inner sep=1.5pt] () at (0,0) {};
  \draw (0,0)--(0,-1);
  \node[circle, fill, inner sep=1.5pt] () at (0.75,0) {};
  \draw (0.75,0)--(1.5,-0.75);
  \draw (1.5,-0.75)--(1.5,-1);
  \node[circle, fill, inner sep=1.5pt] () at (1.5,0) {};
  \draw[white,line width=4pt] (1.3,-0.2)--(0.95,-0.55);
  \draw (1.5,0)--(0.75,-0.75);
  \draw (0.75,-0.75)--(0.75,-1);
  \node[circle, fill, inner sep=1.5pt] () at (2.25,0) {};
  \draw (2.25,0)--(2.25,-1);

  \draw (-0.25,-1) rectangle ++(2.75,-0.5);
  \draw (1.12,-1.5) -- (1.12,-2);
  \node[circle, fill, inner sep=1.5pt] () at (1.12,-2) {};
\end{tikzpicture}
\end{displaymath}
On the left, the two boxes are $\beta_X$ and $\beta_Y$, and the outputs of these are multiplied in $R$. On the right side, the box is $\beta_{X \times Y}$, which we regard as having four inputs.

\begin{proposition}
$\fA_R$ is a c-functor.
\end{proposition}

\begin{proof}
Let $f \colon Y \to X$ be a map of transitive $G$-sets. We have the following equalities of maps $\cC(Y)^{\otimes 2} \to R$:
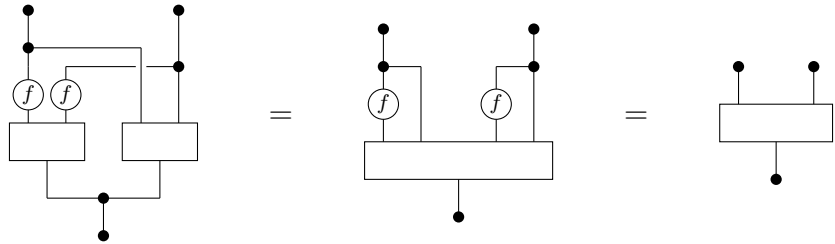
\begin{displaymath}
\begin{tikzpicture}[scale=1, baseline=(current bounding box.center)]
  \node[circle, fill, inner sep=1.5pt] () at (1,0) {};
  \draw (1,0)--(1,-0.5);
  \node[circle, fill, inner sep=1.5pt] () at (1,-0.5) {};
  \draw (1,-0.5)--(1,-0.75);
  \draw (1,-0.75) --  node[midway, draw, circle, fill=white, inner sep=1pt] {\tiny$f$} (1,-1.5);
  \draw (1,-0.5)--(2.5,-0.5);

  \node[circle, fill, inner sep=1.5pt] () at (3,0) {};
  \draw (3,0) -- (3,-0.75);
  \node[circle, fill, inner sep=1.5pt] () at (3,-0.75) {};
  \draw (3,-0.75) -- (1.5,-0.75);
  \draw (3,-0.75) -- (3,-1.5);
  \draw (1.5,-0.75) --  node[midway, draw, circle, fill=white, inner sep=1pt] {\tiny$f$} (1.5,-1.5);
  \draw (0.75,-1.5) rectangle ++(1,-0.5);
  \draw (1.25,-2) -- (1.25,-2.5);
  \draw (1.25,-2.5) -- (2,-2.5);

  \draw[white,line width=4pt] (2.5,-0.6)--(2.5,-0.9);
  \draw (2.5,-0.5)--(2.5,-1.5);
  \draw (2.25,-1.5) rectangle ++(1,-0.5);
  \draw (2.75,-2) -- (2.75,-2.5);
  \draw (2.75,-2.5) -- (2,-2.5);
  \node[circle, fill, inner sep=1.5pt] () at (2,-2.5) {};
  \draw (2,-2.5) -- (2,-3);
  \node[circle, fill, inner sep=1.5pt] () at (2,-3) {};
\end{tikzpicture}
\qquad = \qquad
\begin{tikzpicture}[scale=1, baseline=(current bounding box.center)]
  \node[circle, fill, inner sep=1.5pt] () at (1,0) {};
  \draw (1,0)--(1,-0.5);
  \node[circle, fill, inner sep=1.5pt] () at (1,-0.5) {};
  \draw (1,-0.5) --  node[midway, draw, circle, fill=white, inner sep=1pt] {\tiny$f$} (1,-1.5);
  \draw (1,-0.5)--(1.5,-0.5);
  \draw (1.5,-0.5)--(1.5,-1.5);

  \node[circle, fill, inner sep=1.5pt] () at (3,0) {};
  \draw (3,0) -- (3,-0.5);
  \node[circle, fill, inner sep=1.5pt] () at (3,-0.5) {};
  \draw (3,-0.5) -- (2.5,-0.5);
  \draw (2.5,-0.5) --  node[midway, draw, circle, fill=white, inner sep=1pt] {\tiny$f$} (2.5,-1.5);
  \draw (3,-0.5) -- (3,-1.5);

  \draw (0.75,-1.5) rectangle ++(2.5,-0.5);
  \draw (2,-2) -- (2,-2.5);
  \node[circle, fill, inner sep=1.5pt] () at (2,-2.5) {};
\end{tikzpicture}
\qquad = \qquad
\begin{tikzpicture}[scale=1, baseline=(current bounding box.center)]
  \node[circle, fill, inner sep=1.5pt] () at (1,0) {};
  \draw (1,0)--(1,-0.5);
  \node[circle, fill, inner sep=1.5pt] () at (2,0) {};
  \draw (2,0)--(2,-0.5);
  \draw (0.75,-0.5) rectangle ++(1.5,-0.5);
  \draw (1.5,-1)--(1.5,-1.5);
  \node[circle, fill, inner sep=1.5pt] () at (1.5,-1.5) {};
\end{tikzpicture}
\end{displaymath}
The circles labeled with $f$ represent the map $f_* \colon \cC(Y) \to \cC(X)$. On the left, the two boxes are $\beta_X$ and $\beta_Y$, in the middle the box is $\beta_{X \times Y}$, and on the right the box is $\beta_Y$. The first equality is simply Definition~\ref{defn:central}(b). Note that the composition $\cC(Y) \to \cC(Y \times Y) \to \cC(X \times Y)$ is just push-forward along the map $f \times \id \colon Y \to X \times Y$. Since $f \times \id$ identifies $Y$ with an orbit on $X \times Y$, the second equality follows from Proposition~\ref{prop:pi-basic2}. Note that this identity is simply the diagrammatic form of \eqref{eq:beta1}.

Now let $Z$ be a $G$-orbit on $\fA_R(Y)$. The above identity shows that the restriction of $\beta_Y$ to $\cC(Z)$ can be written as a product, where one of the factors is a map that factors through the restriction of $\beta_X$ to $\cC(f(Z))$. Since $\beta_Y$ is non-zero on $\cC(Z)$, it follows that $\beta_X$ must be non-zero on $\cC(f(Z))$, and so $f(Z) \subset \fA_R(X)$. We have therefore shown that $f$ maps $\fA_R(Y)$ to $\fA_R(X)$, which verifies Definition~\ref{defn:cset}(a).

We now prove Definition~\ref{defn:cset}(b). Put
\begin{displaymath}
W = (f \times \id)^{-1}(\fA_R(X)) \subset Y \times X.
\end{displaymath}
Note that $\id \times f \colon Y \times Y \to Y \times X$ maps $\fA_R(Y)$ to $W$ by the previous paragraph; we must show that it maps onto $W$. Fix a $G$-orbit $W_0$ on $W$. Let $V_1, \ldots, V_r$ be the $G$-orbits on $(\id \times f)^{-1}(W_0)$, and let $\ol{W}_0=(f \times \id)(W_0)$. Now, we have a commutative diagram
\begin{displaymath}
\xymatrix@C=5em{
\cC(W_0) \ar[r]^{(f \times \id)_*} \ar[d]_{(\id \times f)^*} & \cC(\ol{W}_0) \ar[d]^{\beta_X} \\
\bigoplus_{i=1}^r \cC(V_i) \ar[r]^{\beta_Y} & R }
\end{displaymath}
Since $\ol{W}_0 \subset \fA_R(X)$, it follows that $\beta_X$ is not identically~0 on $\cC(\ol{W}_0)$. Since the top map is surjective (as explained above), it follows that the top path in the diagram is non-zero. Hence the bottom path is non-zero, and so $\beta_Y$ is non-zero on some $\cC(V_i)$. Thus $V_i \subset \fA_R(Y)$ for some $i$, and so $W_0$ belongs to the image of $(\id \times f)(\fA_R(Y))$, as required. This establishes Definition~\ref{defn:cset}(b). The proof of Definition~\ref{defn:cset}(c) is similar.
\end{proof}

We say that a central algebra $R$ is \defn{trivial} if $\beta_X$ factors as
\begin{displaymath}
\cC(X \times X) \to \bone \to R,
\end{displaymath}
where the first map is the pairing on $\cC(X) \otimes \cC(X)$, and the second is the unit of $R$. This is compatible with the definition of trivial given in \S \ref{ss:support}.

\begin{proposition} \label{prop:triv2}
A central algebra $R$ has trivial central structure if and only if $\fA_R$ is contained in the diagonal.
\end{proposition}

\begin{proof}
If $R$ has trivial central structure then $\fA_R$ is certainly contained in the diagonal. Conversely, suppose $\fA_R$ is contained in the diagonal. Let $X$ be a transitive $G$-set and let $f \colon X \to \bbone$ be the unique map. Consider the commutative diagram
\begin{displaymath}
\xymatrix@C=4em{
\cC(X \times \bbone) \ar[r]^{(f \times \id)_*} \ar[d]_{(\id \times f)^*} & \cC(\bbone \times \bbone) \ar[d]^{\beta_{\bbone}} \ar@{=}[r] & \bone \ar[d]^{\eta} \\
\cC(X \times X) \ar[r]^-{\beta_X} & R \ar@{=}[r] & R }
\end{displaymath}
where $\eta$ is the unit of $R$. We have $X \times X = \Delta(X) \amalg X^{[2]}$, where the first summand is the image of the diagonal map $\Delta$. Since $\beta_X$ vanishes on $X^{[2]}$, it must factor through the projection $\Delta^* \colon \cC(X \times X) \to \cC(X)$. Now, the composition
\begin{displaymath}
\xymatrix@C=3em{
X \ar[r]^-{\Delta} & X \times X \ar[r]^-{\id \times f} & X }
\end{displaymath}
is an isomorphism. The commutative diagram thus implies that $\beta_X$ factors as $\eta f_* \Delta^*$. Since $f_* \Delta^*$ is the pairing on $\cC(X) \otimes \cC(X)$, we see that the central structure is trivial.
\end{proof}

We next want to establish a ``small support'' result for central algebras. This will require some preparation. Let $R$ be a central algebra. For a $G$-set $X$, we let $\rho_X$ be the composition
\begin{displaymath}
\cC(X \times X) \to \cC(X) \otimes \cC(X \times X) \to \cC(X) \otimes R,
\end{displaymath}
where the first map is induced from the map of $G$-sets taking $(x,y)$ to $(x,x,y)$, and the second map is $\id \otimes \beta_X$. In diagrammatic form, $\rho_X$ is
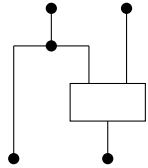
\begin{displaymath}
\begin{tikzpicture}[scale=1, baseline=(current bounding box.center)]
  \node[circle, fill, inner sep=1.5pt] () at (1,0) {};
  \draw (1,0)--(1,-0.5);
  \node[circle, fill, inner sep=1.5pt] () at (2,0) {};
  \draw (2,0)--(2,-1);
  \node[circle, fill, inner sep=1.5pt] () at (1,-0.5) {};
  \draw (1,-0.5)--(0.5,-0.5);
  \draw (0.5,-0.5)--(0.5,-2);
  \node[circle, fill, inner sep=1.5pt] () at (0.5,-2) {};
  \draw (1,-0.5)--(1.5,-0.5);
  \draw (1.5,-0.5)--(1.5,-1);
  \draw (1.25,-1) rectangle ++(1,-0.5);
  \draw (1.75,-1.5)--(1.75,-2);
  \node[circle, fill, inner sep=1.5pt] () at (1.75,-2) {};
\end{tikzpicture}
\end{displaymath}
We now establish some basic properties of $\rho$.

\begin{lemma} \label{lem:rho-homo}
The map $\rho_X$ is an algebra homomorphism.
\end{lemma}

\begin{proof}
Let $a,b \colon \cC(X)^{\otimes 4} \to \cC(X) \otimes R$ be the two maps in question, i.e., $a$ is the map where we first apply $\rho_X \otimes \rho_X$ and then multiply in $\cC(X) \otimes R$, and $b$ is the map where we first multiply in $\cC(X \times X)$ and then apply $\rho_X$. We must show $a=b$. We have the following identity of maps $\cC(X)^{\otimes 4} \to \cC(X) \otimes R$:
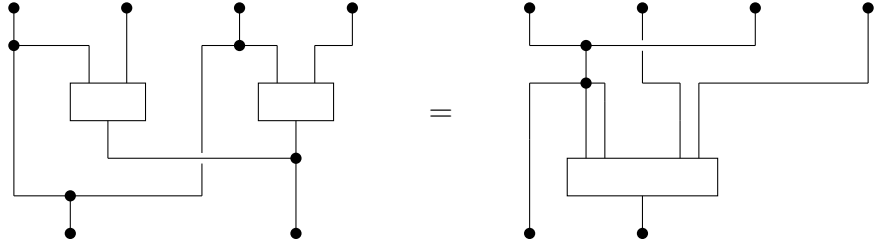
\begin{displaymath}
\begin{tikzpicture}[scale=1, baseline=(current bounding box.center)]
  \node[circle, fill, inner sep=1.5pt] () at (0,0) {};
  \draw (0, 0) -- (0,-0.5);
  \node[circle, fill, inner sep=1.5pt] () at (1.5,0) {};
  \draw (1.5,0) -- (1.5,-1);
  \node[circle, fill, inner sep=1.5pt] () at (3,0) {};
  \draw (3,0) -- (3,-0.5);
  \node[circle, fill, inner sep=1.5pt] () at (4.5,0) {};
  \draw (4.5,0) -- (4.5,-0.5);

  \node[circle, fill, inner sep=1.5pt] () at (0,-0.5) {};
  \draw (0,-0.5)--(0,-2.5);
  \draw (0,-2.5)--(0.75,-2.5);
  \node[circle, fill, inner sep=1.5pt] () at (0.75,-2.5) {};
  \draw (0.75,-2.5)--(0.75,-3);
  \node[circle, fill, inner sep=1.5pt] () at (0.75,-3) {};
  \draw (0,-0.5)--(1,-0.5);
  \draw (1,-0.5)--(1,-1);
  \draw (0.75,-1) rectangle ++(1,-0.5);
  \draw (1.25,-1.5)--(1.25,-2);

  \node[circle, fill, inner sep=1.5pt] () at (3,-0.5) {};
  \draw (3,-0.5)--(2.5,-0.5);
  \draw (2.5,-0.5)--(2.5,-2.5);
  \draw (2.5,-2.5)--(0.75,-2.5);
  \draw (3,-0.5)--(3.5,-0.5);
  \draw (3.5,-0.5)--(3.5,-1);
  \draw (4.5,-0.5)--(4,-0.5);
  \draw (4,-0.5)--(4,-1);
  \draw (3.25,-1) rectangle ++(1,-0.5);
  \draw[white,line width=4pt] (1.5,-2)--(3.5,-2);
  \draw (1.25,-2)--(3.75,-2);
  \draw (3.75,-1.5)--(3.75,-2);
  \node[circle, fill, inner sep=1.5pt] () at (3.75,-2) {};
  \draw (3.75,-2)--(3.75,-3);
  \node[circle, fill, inner sep=1.5pt] () at (3.75,-3) {};
\end{tikzpicture}
\qquad = \qquad
\begin{tikzpicture}[scale=1, baseline=(current bounding box.center)]
  \node[circle, fill, inner sep=1.5pt] () at (0,0) {};
  \draw (0, 0) -- (0,-0.5);
  \node[circle, fill, inner sep=1.5pt] () at (1.5,0) {};
  \draw (1.5,0) -- (1.5,-1);
  \draw (1.5,-1) -- (2,-1);
  \draw (2,-1) -- (2,-1.5);
  \draw (2,-1.5) -- (2,-2);
  \node[circle, fill, inner sep=1.5pt] () at (3,0) {};
  \draw (3,0) -- (3,-0.5);
  \node[circle, fill, inner sep=1.5pt] () at (4.5,0) {};
  \draw (4.5,0) -- (4.5,-1);
  \draw (4.5,-1) -- (2.25,-1);
  \draw (2.25,-1) -- (2.25,-2);

  \draw[white,line width=4pt] (2.75,-0.5)--(1,-0.5);
  \draw (3,-0.5)--(0.75,-0.5);
  \draw (0,-0.5)--(0.75,-0.5);
  \node[circle, fill, inner sep=1.5pt] () at (0.75,-0.5) {};
  \draw (0.75,-0.5)--(0.75,-1);
  \node[circle, fill, inner sep=1.5pt] () at (0.75,-1) {};
  \draw (0.75,-1)--(0,-1);
  \draw (0,-1)--(0,-1.75);
  \draw (0.75,-1)--(0.75,-2);
  \draw (0.75,-1)--(1,-1);
  \draw (1,-1)--(1,-1.25);
  \draw (1,-1.25)--(1,-2);

  \draw (.5,-2) rectangle ++(2,-0.5);
  \draw (1.5,-2.5)--(1.5,-3);
  \draw (0,-1.75)--(0,-3);
  \node[circle, fill, inner sep=1.5pt] () at (1.5,-3) {};
  \node[circle, fill, inner sep=1.5pt] () at (0,-3) {};
\end{tikzpicture}
\end{displaymath}
The left side above is the diagram for the map $a$. The equality comes from the Frobenius algebra axioms for $\cC(X)$, together with Definition~\ref{defn:central}(b). Now, we have the following identity of maps $\cC(X)^{\otimes 3} \to R$:
\begin{equation} \label{eq:rho-homo}
\begin{tikzpicture}[scale=1, baseline=(current bounding box.center)]
  \node[circle, fill, inner sep=1.5pt] () at (0,0) {};
  \draw (0,0)--(0,-0.5);
  \draw (0,-0.5)--(0.25,-0.5);
  \node[circle, fill, inner sep=1.5pt] () at (1,0) {};
  \draw (1,0)--(1,-0.5);
  \draw (1,-0.5)--(1.5,-0.5);
  \draw (1.5,-0.5)--(1.5,-1);
  \node[circle, fill, inner sep=1.5pt] () at (2,0) {};
  \draw (2,0)--(2,-0.5);
  \draw (2,-0.5)--(1.75,-0.5);
  \draw (1.75,-0.5)--(1.75,-1);

  \node[circle, fill, inner sep=1.5pt] () at (0.25,-0.5) {};
  \draw (0.25,-0.5) -- (0.25,-1);
  \draw (0.25,-0.5) -- (0.5,-0.5);
  \draw (0.5,-0.5) -- (0.5,-1);
  \draw (0,-1) rectangle ++(2,-0.5);
  \draw (1,-1.5)--(1,-2);
  \node[circle, fill, inner sep=1.5pt] () at (1,-2) {};
\end{tikzpicture}
\qquad = \qquad
\begin{tikzpicture}[scale=1, baseline=(current bounding box.center)]
  \node[circle, fill, inner sep=1.5pt] () at (0,0) {};
  \draw (0, 0)--(0,-0.5);
  \draw (0, -0.5)--(0.5,-0.5);
  \draw (0.5,-0.5)--(0.5,-1);
  \node[circle, fill, inner sep=1.5pt] () at (1,0) {};
  \draw (1,0)--(1,-0.5);
  \draw (1,-0.5)--(1.5,-0.5);
  \node[circle, fill, inner sep=1.5pt] () at (2,0) {};
  \draw (2,0)--(2,-0.5);
  \draw (2,-0.5)--(1.5,-0.5);
  \node[circle, fill, inner sep=1.5pt] () at (1.5,-0.5) {};
  \draw (1.5,-0.5)--(1.5,-1);
  \draw (0.25,-1) rectangle ++(1.5,-0.5);
  \draw (1,-1.5)--(1,-2);
  \node[circle, fill, inner sep=1.5pt] () at (1,-2) {};
\end{tikzpicture}
\end{equation}
where the box on the left is $\beta_{X \times X}$, and the box on the right is $\beta_X$. To see this, write $X^2=\Delta(X) \sqcup X^{[2]}$, where $X^{[2]}$ is the off diagonal in $X^2$. Thus
\begin{displaymath}
\cC(X)^{\otimes 3} = \cC(X) \otimes \cC(\Delta(X)) \oplus \cC(X) \otimes \cC(X^{[2]}).
\end{displaymath}
The left side of \eqref{eq:rho-homo} vanishes on the second term above by Proposition~\ref{prop:pi-basic2}; moreover, on the first term it agrees with the right side, by the same proposition. Now, using \eqref{eq:rho-homo} in the previous equation, we find that $a$ is equal to
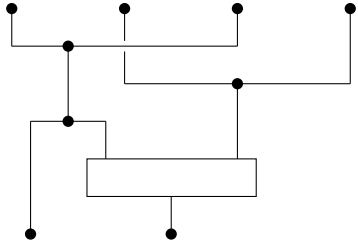
\begin{displaymath}
\begin{tikzpicture}[scale=1, baseline=(current bounding box.center)]
  \node[circle, fill, inner sep=1.5pt] () at (0,0) {};
  \draw (0, 0) -- (0,-0.5);
  \node[circle, fill, inner sep=1.5pt] () at (1.5,0) {};
  \draw (1.5,0) -- (1.5,-1);
  \draw (1.5,-1) -- (3,-1);
  \node[circle, fill, inner sep=1.5pt] () at (3,0) {};
  \draw (3,0) -- (3,-0.5);
  \node[circle, fill, inner sep=1.5pt] () at (4.5,0) {};
  \draw (4.5,0) -- (4.5,-1);
  \draw (4.5,-1) -- (3,-1);
  \node[circle, fill, inner sep=1.5pt] () at (3,-1) {};
  \draw (3,-1) -- (3,-2);

  \draw[white,line width=4pt] (2.75,-0.5)--(1,-0.5);
  \draw (3,-0.5)--(0.75,-0.5);
  \draw (0,-0.5)--(0.75,-0.5);
  \node[circle, fill, inner sep=1.5pt] () at (0.75,-0.5) {};
  \draw (0.75,-0.5)--(0.75,-1.5);
  \node[circle, fill, inner sep=1.5pt] () at (0.75,-1.5) {};
  \draw (0.75,-1.5)--(0.25,-1.5);
  \draw (0.75,-1.5)--(1.25,-1.5);
  \draw (1.25,-1.5)--(1.25,-2);

  \draw (1,-2) rectangle ++(2.25,-0.5);
  \draw (2.12,-2.5)--(2.12,-3);
  \draw (0.25,-1.5)--(0.25,-3);
  \node[circle, fill, inner sep=1.5pt] () at (2.12,-3) {};
  \node[circle, fill, inner sep=1.5pt] () at (0.25,-3) {};
\end{tikzpicture}
\end{displaymath}
and this is exactly $b$.
\end{proof}

\begin{lemma} \label{lem:rho-supp}
Let $W$ be a $G$-orbit on $X \times X$. Then $W \subset \fA_R(X)$ if and only if the restriction of $\rho_X$ to $\cC(W)$ is non-zero.
\end{lemma}

\begin{proof}
It is clear that if $W$ is not contained in $\fA_R(X)$ then $\rho_X$ is zero on $\cC(W)$. Conversely, suppose that $W$ is contained in $\fA_R(X)$. Letting $\epsilon \colon \cC(X) \to \bone$ be the augmentation map, we have $(\epsilon \otimes \id) \circ \rho_X = \beta_X$. Since $\beta_X$ is non-zero on $\cC(W)$, it follows that $\rho_X$ is as well.
\end{proof}

\begin{proposition} \label{prop:rho-inj}
Let $R$ be a central algebra. Then for any transitive $G$-set $X$, the map
\begin{displaymath}
\rho_X \colon \cC(\fA_R(X)) \to \cC(X) \otimes R
\end{displaymath}
is injective.
\end{proposition}

\begin{proof}
Let $W_1, \ldots, W_n$ be the $G$-orbits on $X \times X$. Since $\cP \to \cT$ is fully faithful, each $\cC(W_i)$ is a simple \'etale algebra of $\cT$, and so
\begin{displaymath}
\cC(X \times X) = \bigoplus_{i=1}^n \cC(W_i)
\end{displaymath}
is the decomposition of $\cC(X \times X)$ into simple \'etale algebras. Since $\rho_X$ is an algebra homomorphism (Lemma~\ref{lem:rho-homo}), its kernel is an ideal of $\cC(X \times X)$. Every ideal of $\cC(X \times X)$ is a sum of some subset of the $\cC(W_i)$ \cite[Corollary~5.2]{discrete}; thus, relabeling if necessary, we assume that the kernel of $\rho_X$ is $\bigoplus_{i=r+1}^n \cC(W_i)$ for some $r$. It follows that $\rho_X$ restricts to an injection $\bigoplus_{i=1}^r \cC(W_i) \to \cC(X) \otimes R$. It is clear that the restriction of $\rho_X$ to $\cC(W_i)$ is non-zero for $1 \le i \le r$, and zero for $r<i \le n$, and so $\fA_R(X)=W_1 \amalg \cdots \amalg W_r$ by Lemma~\ref{lem:rho-supp}. The result follows.
\end{proof}

We are finally ready to formulate and prove the ``small support'' result.

\begin{definition} \label{defn:Tsmall2}
A c-functor $\fA$ is \defn{$\cT$-small} if there is an object $M$ of $\cT$ such that for all transitive $G$-sets $X$ there is an injection $\cC(\fA(X)) \to \cC(X) \otimes M$ in $\cT$.
\end{definition}

\begin{corollary} \label{cor:Tsmall2}
The support of a central algebra $R$ is $\cT$-small.
\end{corollary}

\begin{proof}
This follows immediately from Proposition~\ref{prop:rho-inj}.
\end{proof}

\begin{remark}
Definition~\ref{defn:Tsmall2} is compatible with the previous definition of $\cT$-small (Definition~\ref{defn:Tsmall}). Indeed, suppose $\mu$ is quasi-regular and $\cT$ is the abelian envelope of $\cP$. If $\fA$ is $\cT$-small according to Definition~\ref{defn:Tsmall} then it is obviously so according to Definition~\ref{defn:Tsmall2}. Conversely, if $\fA$ is $\cT$-small according to Definition~\ref{defn:Tsmall2}, then we see that it is so according to Definition~\ref{defn:Tsmall} by choosing an injection $M \to \cC(Y)$ for some $G$-set $Y$. This is possible since, in this case, every object of $\cT$ is a sub (or quotient) of an object from $\cP$.
\end{remark}

Combining all of the above work, we obtain the following result.

\begin{theorem} \label{thm:Ztriv2}
Suppose that any $\cT$-small c-functor is contained in the diagonal. Then the natural functor $\cT \to \cZ_{\cT}(\cP)$ is an equivalence.
\end{theorem}

\begin{proof}
Let $M$ be an object of $\cZ_{\cT}(\cP)$. Then the support of the central algebra $\uEnd(M)$ is $\cT$-small (Corollary~\ref{cor:Tsmall2}), and thus, by our assumption, contained in the diagonal. Thus the central structure on $M$ is trivial (Proposition~\ref{prop:triv2}), and so the half-braiding on $M$ is induced from the symmetric structure on $\cT$, as required.
\end{proof}

\subsection{The second Delannoy category} \label{ss:second}

Let $G$ be the oligomorphic group $\Aut(\bQ, <)$ discussed in \S \ref{ss:delannoy}. As stated there, for any field $k$, this group has exactly four measures $\mu_1, \ldots, \mu_4$; here we label the measures as in \cite[\S 4.2]{fake}. The ``classical'' Delannoy category is associated to the regular measure $\mu_1$. Let $\cP$ be the Karoubi envelope of $\uPerm(G, \mu_2)$. In \cite{fake}, the authors construct\footnote{Our $\cP$ and $\cT$ are denoted $\cA$ and $\cD$ in \cite{fake}.} a pre-Tannakian category $\cT$ and a fully faithful tensor functor $\Phi \colon \cP \to \cT$. The category $\cT$ is called the \defn{(abelian) second Delannoy category}. We now determine its Drinfeld center. Since $\mu_2$ is not quasi-regular, we cannot use the theory developed in \S \ref{s:qr}, and must instead use the theory developed above.

We begin by analyzing $\cT$-small c-functors. Since $G$ is a closed oligomorphic group on a countable set, c-functors correspond to closed conjugacy sets (Theorem~\ref{thm:cset}). We say that a conjugacy set $D$ is \defn{$\cT$-small} if $\fA_D$ is. Since $G$ satisfies the conditions of Proposition~\ref{prop:dcl}, we have a notion of small for c-functors and conjugacy sets (see \S \ref{ss:strong-conj}).

\begin{lemma} \label{lem:fake-Tsmall}
A $\cT$-small conjugacy set $D$ is small, and thus contained in $\{1\}$.
\end{lemma}

To prove this, we use some structural results about $\cT$, which we now recall. A \defn{weight} is a word $\lambda$ in the alphabet $\{\ww, \bb\}$. Write $\ell(\lambda)$ for the length of $\lambda$. For each weight $\lambda$, there is a simple object $\bS_{\lambda}$ \cite[\S 6.2]{fake} and indecomposable tilting object $\bT_{\lambda}$ \cite[\S 6.5]{fake}; we will say that $\bS_{\lambda}$ and $\bT_{\lambda}$ have \defn{length} $\ell(\lambda)$. The classes $[\bS_{\lambda}]$ form a $\bZ$-basis for the Grothendieck ring $\rK(\cT)$, as do the classes $[\bT_{\lambda}]$. The change of basis is upper triangular: we have $[\bS_{\lambda}] = [\bT_{\lambda}]+\cdots$ where the remaining terms use classes of tilting modules of shorter length \cite[Corollary~6.9]{fake}, and similarly $[\bT_{\lambda}]=[\bS_{\lambda}]+\cdots$ where the remaining terms are classes of simple modules of shorter length (this follows from the definition of $\bT_{\lambda}$).

For each weight $\lambda$, there is an indecomposable object $M_{\lambda}$ of $\cP$ \cite[\S 5.2]{fake}, and the functor $\Phi \colon \cP \to \cT$ maps $M_{\lambda}$ to $\bT_{\lambda}$ \cite[Corollary~6.14]{fake}. The decomposition of $M_{\lambda} \otimes M_{\mu}$ into indecomposable objects is given in \cite[Theorem~5.18]{fake}. The rule shows that every indecomposable summand $M_{\nu}$ satisfies $\ell(\nu) \le \ell(\lambda)+\ell(\mu)$, and that there are summands with $\ell(\nu)=\ell(\lambda)+\ell(\mu)$. Since $\Phi$ is a tensor functor, $\bT_{\lambda} \otimes \bT_{\mu}$ decomposes according to the same rule. Combined with the change of basis discussed above, we see that the simple constituents of $\bS_{\lambda} \otimes \bS_{\mu}$ have length $\le \ell(\lambda)+\ell(\mu)$, and that a simple constituent of length $\ell(\lambda)+\ell(\mu)$ exists.

Let $\bQ^{(n)}$ be the subset of $\bQ$ consisting of tuples $(x_1, \ldots, x_n)$ with $x_1<\cdots<x_n$. These are exactly the transitive $G$-sets \cite[Corollary~16.2]{repst}. We note that $\bQ^{(n)}$ has level $n$ (Proposition~\ref{prop:level}). We have a decomposition
\begin{displaymath}
\cC(\bQ^{(n)}) = \bigoplus_{\ell(\lambda) \le n} \bT_{\lambda}^{\oplus m(\lambda)} \qquad
m(\lambda) = \binom{n-1}{\ell(\lambda)-1},
\end{displaymath}
by \cite[Corollary~5.17]{fake}. In particular, we see that the simple constituents of $\cC(\bQ^{(n)})$ have length $\le n$, and that a simple constituent of length $n$ exists. Therefore, if $X$ is a non-empty $G$-set then the level of $X$ is the maximal length of a simple in $\cC(X)$.

We are finally ready to return to the proof.

\begin{proof}[Proof of Lemma~\ref{lem:fake-Tsmall}]
Let $M$ be an object of $\cT$ such that we have an injection $\cC(\fA_D(X)) \to \cC(X) \otimes M$ for each transitive $G$-set $X$. Let $s$ be the maximum length of a simple object appearing in $M$. The simple constituents of $\cC(\bQ^{(n)}) \otimes M$ have length $\le n+s$, and so $\fA_D(\bQ^{(n)})$ has level $\le n+s$. This shows that $D$ is small. Since $G$ has no non-trivial finitary elements, its only small conjugacy sets are $\emptyset$ and $\{1\}$ (Proposition~\ref{prop:small-conj}), which completes the proof.
\end{proof}

We now come to the main result.

\begin{theorem}
The functor $\cT \to \cZ(\cT)$ is an equivalence.
\end{theorem}

\begin{proof}
Theorem~\ref{thm:Ztriv2}, combined with Lemma~\ref{lem:fake-Tsmall}, shows that $\cT \to \cZ_{\cT}(\cP)$ is an equivalence. Since the bounded derived category of $\cT$ is equivalent to the bounded homotopy category of $\cP$ \cite[Proposition~6.15]{fake}, it follows that the natural functor $\cZ(\cT) \to \cZ_{\cT}(\cP)$ is an equivalence. The result follows.
\end{proof}

\begin{remark}
The measures $\mu_2$ and $\mu_3$ are interchanged by the outer automorphism of $G$ that reverses the line, and so the categories $\uPerm(G, \mu_2)$ and $\uPerm(G, \mu_3)$ are equivalent \cite[Proposition~7.16]{delmap}. Thus all of the above discussion extends to the third Delannoy category. For $\mu_4$, we do not yet know if there is an associated pre-Tannakian category. However, the center of $\cP = \uPerm(G, \mu_4)^{\rm kar}$ can be determined using the methods developed in this paper; the functor $\cP \to \cZ(\cP)$ is an equivalence.
\end{remark}

\section{Further discussion}

In this final section, we discuss a few miscellaneous topics.

\subsection{The center of the category of $G$-sets}

Let $G$ be a pro-oligomorphic group. Let $\bS(G)$ be the category of (finitary smooth) $G$-sets. This is symmetric monoidal under the cartesian product. We now determine its Drinfeld center, under mild technical assumptions. The following is the key definition; see \cite[\S 4.2]{FY} and \cite[\S 3.4]{KW} for discussion of this concept for ordinary groups.

\begin{definition}
A \defn{crossed $G$-set} is a (finitary smooth) $G$-set $M$ equipped with a function $\vert \cdot \vert \colon M \to G$ such that $\vert gm \vert = g \vert m \vert g^{-1}$. A \defn{map} of crossed $G$-sets $f \colon N \to M$ is a map of $G$-sets such that $\vert f(n) \vert = \vert n \vert$. We write $\bT(G)$ for the category of crossed $G$-sets.
\end{definition}

Suppose $M$ is a crossed $G$-set. Since the stabilizer of $m$ is an open subgroup of $G$, it follows that $\vert m \vert$ is a c-smooth element of $G$. Thus $\vert \cdot \vert$ is valued in the group $G_{\sm}$ of c-smooth elements; since $M$ is finitary, the image of $\vert \cdot \vert$ in $G_{\sm}$ is a finite union of conjugacy classes. For a $G$-set $X$, we define a map
\begin{displaymath}
\beta_X \colon X \times M \to M \times X, \qquad (x, m) \mapsto (m, \vert m \vert \cdot x).
\end{displaymath}
One readily verifies that this defines a half-braiding on $M$. We thus have a functor $\bT(G) \to \cZ(\bS(G))$. We leave to the reader the verification that it is fully faithful.

\begin{theorem} \label{thm:ZS}
Suppose $(G, \Omega)$ is a closed oligomorphic permutation group, with $\Omega$ countable. Then the functor $\bT(G) \to \cZ(\bS(G))$ is an equivalence.
\end{theorem}

\begin{proof}
Define $G^*$ to be the inverse limit of $G/G(A)$ over finite subsets $A$ of $\Omega$. This is, essentially by definition, the topological closure of $G$ in the space of injections $\Omega \to \Omega$. It is clear that $G^*$ is a monoid. Since $G$ is closed, the group of invertible elements in $G^*$ is $G$. If $X$ is a $G$-set then there is a unique continuous action of $G^*$ on $X$ extending the action of $G$. To see this, suppose $g \in G^*$ and $x \in X$ are given. Then $x$ is fixed by some $G(A)$, and there is some $h \in G$ such that $ga=ha$ for all $a \in A$. The action is defined by $gx=hx$.

Let $M$ be an object in the center, and let
\begin{displaymath}
\beta_X \colon X \times M \to M \times X
\end{displaymath}
denote the half-braiding on $M$. If $X=\bbone$ is a one-point $G$-set then $\beta_X$ is the identity on $M$, i.e., we have $\beta_X(\ast, m)=(m, \ast)$. From this and the naturality of $\beta$, it follows that for any $G$-set $X$ and $x \in X$, we have $\beta_X(x,m)=(m,y)$ for some $y \in X$. Taking the inverse limit of $\beta_{G/G(A)}$ over finite subsets $A$ of $\Omega$, we obtain a bijection
\begin{displaymath}
\beta^* \colon G^* \times M \to M \times G^*.
\end{displaymath}
Define $\vert \cdot \vert \colon M \to G^*$ by
\begin{displaymath}
\beta^*(1,m) = (m, \vert m \vert).
\end{displaymath}
Essentially by definition, we have
\begin{displaymath}
\beta_{G/G(A)}(1,m) = (m, \vert m \vert)
\end{displaymath}
for all finite subsets $A \subset \Omega$. Suppose $X$ is a $G$-set and $x \in X$. There is some $G(A)$ fixing $x$, and so we have a map $G/G(A) \to X$ taking 1 to~$x$. From the naturality of $\beta$, we find
\begin{displaymath}
\beta_X(x,m) = (m, \vert m \vert x).
\end{displaymath}
For $g \in G$, we have
\begin{displaymath}
(gm, g\vert m \vert x) = g \beta_X(x,m) = \beta_X(gx,gm) = (gm, \vert gm \vert gx).
\end{displaymath}
Thus $\vert gm \vert$ and $g\vert m \vert g^{-1}$ act the same on all $G$-sets, and are therefore equal.

The above discussion shows that $M$ is essentially a crossed $G$-set, except that $\vert \cdot \vert$ maps to $G^*$. To complete the proof, it suffices to show that $\vert \cdot \vert$ maps to $G$. Fix $m$, and put $g=\vert m \vert$. Since $\beta^*$ is a bijection, there is some $h \in G^*$ such that $\beta^*(h,m) = (m,1)$. Arguing similarly to the above, we find that $\beta_X^{-1}(m,x)=(hx,m)$ for all $G$-sets $X$ and $x \in X$. We thus see that $ghx=hgx=x$ for all $x$ and $X$, which implies $gh=hg=1$. Hence $g$ is invertible, and thus belongs to $G$.
\end{proof}

\begin{remark}
Let $M$ be an object of $\cZ(\bS(G))$ and write $\beta$ for the half-braiding. For a transitive $G$-set $X$, define $\fA_M(X) \subset X \times X$ to be the set of all pairs $(x,y)$ for which there exists some $m \in M$ such that $\beta(x,m)=(m,y)$. It is easy to directly verify that $\fA_M$ is a c-functor. Thus, in the setting of the theorem, we have $\fA_M=\fA_D$ for some closed conjugacy set $D \subset G$ (Theorem~\ref{thm:cset}). This set $D$ is the image of $\vert \cdot \vert \colon M \to G$.
\end{remark}

\subsection{The MZ operation} \label{ss:mz}

Two objects $X$ and $Y$ of a braided monoidal category $(\cZ,c)$ are said to \defn{centralize} each other if
\begin{displaymath}
c_{Y,X} \circ c_{X,Y} = \id_{X \otimes Y}
\end{displaymath}
(\cite[Definition~8.20.1]{EGNO}). For example, if $\cC$ is a symmetric monoidal category then $Z \in \cZ(\cC)$ centralizes $\cC$ if and only if $Z \in \cC \subset \cZ(\cC)$. 

The \defn{M\"uger center} $\cM(\cZ)$ of $(\cZ,c)$ is the full subcategory of objects that centralize all other objects. It is a symmetric monoidal subcategory of $\cZ$. If in addition $\cZ$ is $k$-linear abelian then it is called \defn{non-degenerate} if $\cM(\cZ)=\Vec_k$. 

We write $\cM\cZ$ for the composition of the Drinfeld center and M\"uger center operations on symmetric categories. Thus $\cM\cZ(\cC)$ is a monoidal subcategory of $\cC$. We now make some comments on this operation in the case of pre-Tannakian categories over an algebraically closed field $k$. 

The following proposition is well-known; see \cite[Corollary~8.20.13]{EGNO} for the fusion case and \cite[8.6.3]{EGNO}, \cite{Shimizu} in general.

\begin{proposition}\label{nondeg}
If $\cC$ is a finite (not necessarily braided) tensor category over $k$ then $\cZ(\cC)$ is non-degenerate, i.e., $\cM\cZ(\cC)=\Vec_k$.
\end{proposition} 

This proposition fails, however, when $\cC$ is a nested union of finite tensor categories. For example, if $G$ is a profinite group then by Proposition~\ref{prop:Opi-center}, $\cZ(\Rep{G})$ is the category of conjugation-equivariant finite length sheaves of $k$-vector spaces on $G$, and such a sheaf must be supported on c-smooth elements, i.e., those generating a finite conjugacy class. Thus, denoting by $G_{\sm} \subset G$ the (normal) subgroup of smooth elements and by $\ol{G_{\sm}}$ its closure in $G$, we obtain that $\cM\cZ(\Rep{G})=\Rep(G/\ol{G_{\sm}})$. In particular, if $G$ has no non-trivial c-smooth elements (e.g., $G = \SL_2(\bZ_p)/\{ \pm 1 \}$), then $\cM\cZ(\cC)=\cC$. 

Nevertheless, we have:

\begin{proposition}\label{nondeg1}
Proposition~\ref{nondeg} holds for the category $\cC=\Rep(G)$ where $G$ is a connected affine group scheme over $k$, i.e., $\cM\cZ(\cC)=\Vec_k$.  
\end{proposition} 

\begin{proof}
In this case, by Proposition \ref{prop:Opi-center}, $\cZ(\cC)$ is the category of conjugation-equivariant finite dimensional $\cO(G)$-modules. Let us show that it is non-degenerate. Let $V\in \cM\cZ(\cC) \subset \cC$, i.e., $V$ is a representation of $G$. Let $\fm$ be the augmentation ideal of $\cO(G)$. Since $G$ is connected, we have $\bigcap_{n\ge 1} \fm^n=0$, thus 
\begin{equation} \label{inclu}
\cO(G)\subset \varprojlim \cO(G)/\fm^n.
\end{equation}  
Note that $G=\varprojlim G_r$ where $G_r$ are of finite type, so $\cO(G)/\fm^n = \varinjlim \cO(G_r)/\fm_r^n$, where $\fm_r$ is the augmentation ideal in $\cO(G_r)$. Thus $\cO(G)/\fm^n$ is a locally finite $G$-module, hence centralizes $V$. This implies that the composition $V \to \cO(G) \otimes V \to (\cO(G)/\fm^n) \otimes V$ of the coaction with reduction modulo $\fm^n \otimes V$ is just the tensor product of the unit map $\iota \colon \bone \hookrightarrow \cO(G)/\fm^n$ with $1_V$. Since this holds for all $n$, \eqref{inclu} implies that $G$ acts trivially on $V$, as desired. 
\end{proof} 

Proposition \ref{nondeg1} extends verbatim to connected affine supergroup schemes. However, it fails in the disconnected case. For example, working over the complex numbers, for $\cC = \Rep{\bO_{2n}}$ we have $\cM\cZ(\cC) = \Rep(\bZ/2)$. Indeed, the conjugacy classes with $\det=-1$ in $\bO_{2n}$ do not contribute to $\cZ(\cC)$, since they are infinite, so $\cZ(\cC) = \cZ(\Rep{\SO_{2n}})^{\bold Z/2}$, the equivariantization 
of $\cZ(\Rep{\SO}_{2n})$ by $\bZ/2=\bO_{2n}/\SO_{2n}$, which implies the claim since $\cM\cZ(\Rep{\SO_n}) = \Vec_{\bC}$ by Proposition~\ref{nondeg1}.

We also have the following proposition, giving more such examples. For a tensor category $\cC$, let $S^n\cC:=(\cC^{\boxtimes n})^{\fS_n}$, the equivariantization of the $n$-th Deligne tensor power of $\cC$ by  $\fS_n$.

\begin{proposition} \label{smash}
Let $\cC$ be a pre-Tannakian category. If $\cC$ is infinite then $\cZ(S^n\cC) = S^n\cZ(\cC)$, thus  $\cM\cZ(S^n\cC) = S^n\cM\cZ(\cC)$. 
\end{proposition} 

\begin{proof}
Let $\pi=\pi(\cC)$, and $F \colon S^n\cC \to \cC$ be the natural functor. By Proposition~\ref{prop:Opi-center}, for every object $M$ of $\cZ(S^n\cC)$, we have a decomposition $F(M)=\bigoplus_{s \in \fS_n} F(M)_s$, where $F(M)_s$ is an $\mathcal O(\pi)^{\otimes n}$-module. Suppose $s(j)\ne j$ for some $j\in [1,n]$, and regard $F(M)_s$ as an $\mathcal O(\pi)$-module using the $j$-th factor in $\mathcal O(\pi)^{\otimes n}$. By the fundamental theorem for Hopf modules \cite[Proposition~2.3.3]{AG}, this module is free, i.e., we have $F(M)_s \cong \cO(\pi) \otimes V$, where $V\in \cC$. Since $F(M)_s$ must have finite length, if $\cC$ is infinite then $F(M)_s=0$. This implies the statement.
\end{proof}

The same result with the same proof applies more generally when $\fS_n$ is replaced by its subgroup $G$: if $\cC$ is infinite then $\cZ((\cC^{\boxtimes n})^G) = (\cZ(\cC)^{\boxtimes n})^G$ and $\cM\cZ((\cC^{\boxtimes n})^G) = (\cM\cZ(\cC)^{\boxtimes n})^G$.

For disconnected supergroup schemes, we have the following proposition, in which for simplicity we restrict to finite type. Let $G$ be an affine supergroup scheme of finite type over $k$ with even part $\overline G$, and let $\varepsilon \in \ol{G}(k)$ be a (central) element of order $\le 2$ inducing the parity operator on $\cO(G)$. Let $\cC = \Rep(G,\varepsilon)$ be the category of representations of $G$ on superspaces in which $\varepsilon$ acts by parity (i.e., a general finitely tensor-generated super-Tannakian category over $k$). Let $\ol{G}^\circ$ be the identity component of $\ol{G}$ and $Z$ be the centralizer of $\ol{G}^\circ$ in $\ol{G}$; note that $\varepsilon \in Z(k)$. Thus $Z \ol{G}^\circ \subset \ol{G}$ is a closed normal subgroup scheme, and $\ol{G}/Z\ol{G}^\circ$ is a finite group. 

\begin{proposition} \label{nondeg2}
We have $\cM\cZ(\cC) = \Rep(\ol{G}/Z\ol{G}^\circ)$. Thus for any finitely tensor-generated super-Tannakian category $\cC$, we have $\cM\cZ^2(\cC)=\Vec_k$.
\end{proposition} 

\begin{proof}
The proof is similar to the proof of Proposition~\ref{nondeg1}, using that an element $z \in \ol{G}(k)$ belongs to a finite conjugacy class if and only if $z \in Z(k)$. 
\end{proof} 

This motivates the following conjecture. 
 
\begin{conjecture} \label{conj:mz}
If $\cC$ is a finitely tensor-generated pre-Tannakian category of moderate growth, then $\cM\cZ(\cC) = \Rep(K)$, where $K$ is a finite group. Thus $(\cM\cZ)^2(\cC) = \Vec_k$. 
\end{conjecture}

Indeed, by Deligne's theorem \cite{Deligne2} and Proposition~\ref{nondeg2}, Conjecture~\ref{conj:mz} holds in characteristic zero. Moreover, we have:

\begin{proposition} \label{nondeg3}
Conjecture~\ref{conj:mz} holds for Frobenius exact categories $\cC$ over $k$ of characteristic $p>0$. 
\end{proposition} 

\begin{proof}
By \cite{CEO}, $\cC$ is the category of representations of an affine group scheme $G$ in the Verlinde category $\Ver_p(k)$ (compatible with $\pi(\Ver_p(k))$). Thus the proposition can be proved in the same way as Proposition~\ref{nondeg1}. The only ingredient that needs justification is the Krull intersection theorem $\bigcap_{n\ge 1} \fm^n=0$ for the augmentation ideal $\fm \subset \cO(G)$ for connected $G$, which in $\Ver_p(k)$ is \cite[Proposition~3.9]{Venkatesh}.
\end{proof}

We expect that a similar method should apply if $\cC$ fibers over a finite pre-Tannakian category. This conjecturally includes all cases by \cite[Conjecture~1.4]{BEO}.

On the other hand, in non-moderate growth there are examples of finitely tensor-generated categories with $(\cM\cZ)^2(\cC)\ne \Vec_k$; in fact, there are examples when one gets $\Vec_k$ only after applying $\cM\cZ$ $m$ times, where $m$ is arbitrarily large. Namely, let $\cC$ be a semi-simple pre-Tannakian category over a field $k$ of characteristic zero. Recall that Mori \cite{Mori} defines the smash product $S^t \cC$ of $\cC$ with Deligne's category $\uRep(\fS_t)$ for $t\in k \setminus \bN$, which is again a semi-simple pre-Tannakian category. For instance, $S^t\Vec_k = \uRep(\fS_t)$. One can think of $S^t\cC$ as an interpolation of the construction $S^n\cC$. It is clear that if $\cC$ is finite tensor-generated, then so is $S^t\cC$. Interpolating the proof of Proposition~\ref{smash}, we obtain

\begin{proposition} \label{smash1}
If $\cC$ is infinite then $\cZ(S^t\cC) = S^t\cZ(\cC)$, thus $\cM\cZ(S^t\cC)=S^t\cM\cZ(\cC)$. 
\end{proposition} 

Let $\cD(t_1, \ldots, t_m)$ be the iterated smash product $S^{t_1} \cdots S^{t_m} \Vec_k$, where each $t_i \in k \setminus \bN$. Using that $\cM\cZ(\uRep(\fS_t)) = \Vec_k$ (\cite{FHL}), Proposition~\ref{smash1} implies:

\begin{corollary}\label{smash2}
We have
\begin{displaymath}
\cM\cZ(\cD(t_1, \ldots, t_m)) = \cD(t_1,...,t_{m-1}).
\end{displaymath}
In particular, $(\cM\cZ)^{m-1} (\cD(t_1, \ldots, t_m)) = \uRep(\fS_{t_1})$ and $(\cM\cZ)^m (\cD(t_1, \ldots, t_m)) = \Vec_k$.
\end{corollary} 

Let us say that a pre-Tannakian category $\cC$ is \defn{$\cM\cZ$-solvable} if $(\cM\cZ)^n(\cC) = \Vec_k$ for some $n$. Thus Conjecture~\ref{conj:mz} (which is a theorem in characteristic zero) implies that any finitely tensor-generated pre-Tannakian category of moderate growth is $\cM\cZ$-solvable. Also, by Corollary~\ref{smash2} the category $\cD(t_1, \ldots, t_m)$ is $\cM\cZ$-solvable. In fact, we don't know any 
interpolation categories that are not $\cM\cZ$-solvable. 

Yet our results show that a finitely tensor-generated pre-Tannakian category need not be $\cM\cZ$-solvable. The simplest such example is provided by the Delannoy category: in this case $\cZ(\cC)=\cC$ (Theorem~\ref{mainthm1}), and so $\cM\cZ(\cC)=\cC$. Let us call a pre-Tannakian category $\cC$ \defn{inert} if $\cM\cZ(\cC) = \cC$. Thus the Delannoy category is inert. Another example of an inert category is the equivariantization of the Delannoy category under $\bZ/2$ acting by flipping the line. This shows that an inert category can have a tensor subcategory that is not inert. Also, as we have mentioned, if $G$ is a profinite group then $\Rep(G)$ is inert if and only if $G$ has no nontrivial finite conjugacy classes. 

If $\cC$ is inert, one may wonder if one can ``engage'' any of its objects to braid nontrivially with 
other objects by possibly extending $\cC$ to a bigger pre-Tannakian category and then taking the Drinfeld center. Namely, we ask

\begin{question}
Suppose $\cC \subset \cB$ are pre-Tannakian categories (so $\cC$ is a tensor subcategory of $\cB$). Is $\cM\cZ(\cC)$ necessarily contained in $\cM\cZ(\cB)$? In particular, if $\cC$ is inert, must $\cM\cZ(\cB)$ contain $\cC$? E.g., if $\cC$ is the Delannoy category? 
\end{question}

The answer to this question is ``yes'' if $\cC$ is the representation category of an affine group scheme $G$. Indeed, in this case we can consider the de-equivariantized category $\cZ(\cB)_G$, which is a $G$-crossed category in the sense of Turaev. Every object $M$ of this category decomposes into components $M_g$, $g\in G(k)$. It is easy to see that $M_g=0$ unless $g$ is c-smooth, which implies the statement. For example, if $\cC$ is inert (which happens if and only if $G$ is a profinite group without nontrivial c-smooth elements) then $M_g=0$ for all $g\ne 1$, which means that $\cZ(\cB)_G$ is actually braided and $G$ acts by braided autoequivalences. Thus, performing $G$-equivariantization, we get back $\cZ(\cB)$ with $\cM\cZ(\cB) \supset \cC$. 

Another example when the answer is yes is Kriz's quantum Delannoy category $\cB$ (\S \ref{ss:kriz}), which contains the ordinary Delannoy category $\cC$. It is easy to deduce from our description of $\cZ(\cB)$ that $\cM\cZ(\cB) = \cC$.

\subsection{A generalization of the center} \label{ss:generalization}

Let $G$ be a finite group, let $H=G \times G$, and identify $G$ with the diagonal subgroup of $H$. The Drinfeld center of $\Rep{G}$ can be described in the following two ways: (a) it is the category of $G$-equivariant sheaves on $H/G$; (b) it is the category of $\Rep{H}$ linear endofunctors of $\Rep{G}$. We can generalize the study of the Drinfeld center by considering these categories for a general inclusion of groups $G \subset H$. We now make a few comments on such a generalization in the oligomorphic setting.

Let $G$ and $H$ be pro-oligomorphic groups with quasi-regular measures $\mu$ and $\nu$. Suppose that nilpotents have trace zero, and so we have pre-Tannakian categories $\cT=\uRep(G, \mu)$ and $\cS=\uRep(H, \nu)$. Furthermore, suppose that $G$ is a closed subgroup of $H$, and there is a corresponding restriction functor
\begin{displaymath}
\Phi \colon \cS \to \cT.
\end{displaymath}
The existence of $\Phi$ requires some compatibility between the measures; see \cite[\S 4.1]{delmap}. The functor $\Phi$ has a left adjoint at the level of ind-categories:
\begin{displaymath}
\Phi_! \colon \Ind{\cT} \to \Ind{\cS}.
\end{displaymath}
Let $A=\Phi(\Phi_!(\bone))$. This is a commutative algebra in $\Ind{\cT}$, which is intuitively the coordinate ring of $H/G$. Thus $A$-modules play the role of $G$-equivariant sheaves on $H/G$. One can show that the category of $A$-modules is equivalent to the category of exact co-continuous $\cS$-linear functors $\Ind{\cT} \to \Ind{\cT}$.

It should be possible to generalize the methods of \S \ref{s:qr} to analyze the category of $A$-modules. Here is the kind of result we have in mind. Let $X=H/G$ and let $X_{\sm}$ be the set of elements $x \in X$ that are stabilized by an open subgroup of $G$, i.e., $X_{\sm}$ is the set of smooth points for the $G$-action. Then $\cC(X_{\sm})$ is an algebra in $\Ind{\cT}$ with approximate units, as in Remark~\ref{rmk:yd}(b). In favorable circumstances, we expect that $\Mod_A^{\fin}$ is equivalent to $\Mod^{\fin}_{\cC(X_{\sm})}$. Recall that the $\fin$ superscript means we consider modules where the underlying object belongs to $\cT$ (and not $\Ind{\cT}$). This is the analogous statement to Theorem~\ref{thm:YD}: indeed, when $G$ is the diagonal in $H=G \times G$, we find $X_{\sm}=G_{\sm}$ is the group of c-smooth elements in $G$ and $\Mod_A^{\fin}=\cZ(\cT)$.

\begin{example} \label{ex:order}
Work over a field of characteristic~0. Let $\cS=\uRep(\fS_{-1})$, let $\cT$ be the Delannoy category, and let $\Phi \colon \cS \to \cT$ be the functor taking basic Frobenius algebra of $\cS$ to the Schwartz space $\cC(\bQ)$. Here $H=\fS$ is the infinite symmetric group and $G=\Aut(\bQ,<)$. We identify $H$ with the group of all permutations of $\bQ$, so that $G$ is a subgroup of $H$. We identify $X=H/G$ with the set of total orders on $\bQ$ that are abstractly isomorphic to the standard order, meaning they are dense and without endpoints.

The $G$ action on $X$ has many smooth points. Here is an example. Define an order $\prec$ on $\bQ$ by using the standard order on the intervals $(-\infty,0)$ and $(0,\infty)$, but putting $a \prec 0 \prec b$ when $a$ is positive and $b$ is negative. Then $\prec$ is preserved by the open subgroup $G(0)$ of $G$ that fixes~0, and is therefore a smooth point of $X$. More generally, one can choose $n$ points in $\bQ$, use the standard order (or its reverse) on each interval between consecutive points, but permute the intervals and points. We show below that all smooth points of $X$ arise by this construction. This gives an explicit description of the $G$-set $X_{\sm}$, which should translate to an explicit description of $\Mod_A^{\fin}$. It would be interesting to explore this example in more detail.

Let $\prec \in X_{\sm}$, and let $G(A)$ be an open subgroup of $G$ fixing $\prec$, where $A$ is a finite subset of $\bQ$. Thus if $a \prec b$ then $ga \prec gb$, for all $a,b \in \bQ$ and $g \in G(A)$. Let $I_0, \ldots, I_m$ be the open intervals making up $\bQ \setminus A$. Suppose we have points $x<y$ and $x'<y'$ in $I_r$. Then there exists $g \in G(A)$ such that $gx=x'$ and $gy=y'$. Thus $x \prec y$ if and only if $x' \prec y'$. This shows that the restriction of $\prec$ to $I_r$ is either the standard order or its reverse. Similarly, if we have $x,x' \in I_r$ and $y,y' \in I_s$ with $r \ne s$ then there is $g \in G(A)$ such that $gx=x'$ and $gy=y'$, and so $x \prec y$ if and only if $x' \prec y'$. This shows that $\prec$ is constant on $I_r \times I_s$ for $r \ne s$. The same reasoning shows that $\prec$ is constant on $I_r \times \{a\}$ for $a \in A$. Thus $\prec$ is as described in the previous paragraph.
\end{example}

\begin{remark}
The closure of a smooth $G$-orbit on $X$ need not be a finite union of smooth $G$-orbits. Indeed, we have already seen this for a conjugation action in Remark~\ref{rmk:cmooth-not-isolated}. We mention another example coming from the situation in Example~\ref{ex:order}. For $c \in \bR$, let $\prec_c \in X$ be the total order the restricts to the standard order on $(-\infty, c)$ and $(c, \infty)$, and satisfies $a<b$ when $a \in (c,\infty)$ and $b \in (-\infty,c)$; additionally, if $c$ is rational then we declare $a<c<b$ for such $a$ and $b$. The set $\{\prec_c\}_{c \in \bQ}$ is a $G$-orbit, and we have already seen that its points are smooth. One easily sees that $\prec_c$, for any $c \in \bR$, belongs to the closure of this orbit; note that this element is not smooth when $c$ is irrational. Thus the orbit closure contains points that are not smooth.
\end{remark}

\begin{remark}
One can attempt to generalize even further. For instance, one can study $G$-equivariant sheaves on $H/K$, for a closed pro-oligomorphic subgroup $K \subset H$. In this case, we require a category $\cR=\uRep(K, \theta)$ with a restriction functor $\Psi \colon \cS \to \cR$, and we take $A=\Phi(\Psi_!(\bone))$. Even more generally, suppose $X$ is a set equipped with a continuous (but not necessarily smooth) action of $G$ (e.g., $X=H/K$). Then perhaps there is some notion similar to measure that will allow one to define $G$-equivariant sheaves on $X$. It would be interesting to clarify this notion and explore further, e.g., are there sheaf operations associated to a continuous map $X \to Y$ of $G$-sets?
\end{remark}

\end{document}